\documentclass[11pt, reqno]{amsart}

\numberwithin{equation}{section}

\usepackage{amssymb,amsmath,amscd}
\usepackage{amsthm}
\newtheorem{theorem}{Theorem}[section]
\newtheorem{theorem*}{Theorem }
\newtheorem{proposition}{Proposition}[section]

\newtheorem{corollary}{Corollary}[section]

\newtheorem{proposition*}{Proposition A\!\!}
\newtheorem{corollary*}{Corollary A\!\!}
\newtheorem{lemma}{Lemma}[section]
\newtheorem{remark}{Remark}[section]

\DeclareMathOperator{\End}{End}

\DeclareMathOperator{\id}{id}
\DeclareMathOperator{\Id}{Id}
\DeclareMathOperator{\tr}{tr}

\DeclareMathOperator{\res}{\mathbf{res}}

\DeclareMathOperator{\Aff}{Aff}

 \usepackage{mathtools}
\usepackage{lipsum}
\usepackage[bbgreekl]{mathbbol}

\DeclareSymbolFontAlphabet{\mathbb}{AMSb}
\DeclareSymbolFontAlphabet{\mathbbl}{bbold}
\DeclareMathOperator{\rr}{\mathbbl{r}}
  \DeclareMathOperator{\dd}{\mathbbl{d}}

 \usepackage[
 pdftex=true, colorlinks=true,  citecolor=cyan,   
 citebordercolor={0 .255 .255} , 
  linkbordercolor={0 .255 .255},  
  linkcolor=cyan
 ]{hyperref}

  \usepackage{cite}

 \def\equationautorefname~#1\null{(#1)\null}
 
\usepackage{multirow}
 \usepackage{stmaryrd}
 \usepackage{graphicx}
 \usepackage{caption}
 \usepackage{booktabs}
 \usepackage{threeparttable}

\DeclareMathOperator{\Det}{\mathrm{Det}}
\DeclareMathOperator{\Mat}{Mat}
\DeclareMathOperator{\Co}{Co}
\DeclareMathOperator{\Str}{Str}

\renewcommand{\det}{\mathbf{det}}
\renewcommand{\tr}{\mathbf{tr}}

\makeatletter
\ifcsname phantomsection\endcsname
    \newcommand*{\qrr@gobblenexttocentry}[5]{}
\else
    \newcommand*{\qrr@gobblenexttocentry}[4]{}
\fi
\newcommand*{\addsubsection}{%
    \addtocontents{toc}{\protect\qrr@gobblenexttocentry}%
    \subsection}
\makeatother

\usepackage{lipsum}

\begin{document}
 

\title[bi-differential operators on  a simple real   Jordan algebra]{Conformally covariant bi-differential operators on  a simple real   Jordan algebra}

\author{Salem Ben Sa\"id, Jean-Louis Clerc and Khalid Koufany }
\address{Institut Elie Cartan de Lorraine, Universi\'e de Lorraine}
\email{(Salem.Bensaid / Jean-Louis.Clerc / Khalid.Koufany) @univ-lorraine.fr}

\dedicatory{\`A Jacques Faraut avec reconnaissance}

 \maketitle
 
 \begin{abstract}
 For a simple real Jordan algebra $V,$ a family of bi-differential operators from $\mathcal{C}^\infty(V\times V)$ to $\mathcal{C}^\infty(V)$ is constructed. These operators are covariant under the rational action of the conformal group of $V.$ They generalize the classical {\em Rankin-Cohen} brackets (case $V=\mathbb{R}$).
 \end{abstract}

 \section*{Introduction}
The {\em Rankin-Cohen brackets}  are well known  examples of \emph{covariant bi-differential operators}  (see e.g. \cite{El,  Cohen, PZ,   Pevzner2, zagier}). They appeared in the theory of modular forms as constant coefficients holomorphic bi-differential operators on the upper half-plane in $\mathbb C$. Such operators are covariant with respect to representations of the holomorphic discrete series of the Lie group $\mathrm{SL}(2,\mathbb R)$. They even have predecessors, known as \emph{transvectants} (see e.g. \cite{Olver}), much used in  the classical  theory of invariants. The transvectants can be interpreted as constant coefficients bi-differential operators on $\mathbb R$  covariant with respect to representations of $\mathrm{SL}(2,\mathbb R)$ containing a finite dimensional subrepresentation. More generally, the principal series  representations of  $\mathrm{SL}(2,\mathbb R)$ are indexed by $(\lambda, \varepsilon)\in \mathbb C\times  \{\pm\}$. Given two such representations $\pi_{\lambda, \varepsilon}$ and  $\pi_{\mu,\eta}$, and given a positive  integer $N,$ there exists a family of constant coefficients bi-differential operators $B_{\lambda, \mu}^{(N)}$ on $\mathbb R$ which are covariant with respect to $(\pi_{\lambda, \varepsilon}\otimes \pi_{\mu, \eta}, \pi_{\lambda+\mu+2N, \varepsilon\eta})$. Moreover, the family depends rationally on the parameters $\lambda$ and $\mu$. For special  values of  $\lambda$ and $\mu$ (corresponding to  cases where the representations are reducible), they coincide with the Rankin-Cohen brackets or the transvectants.

There are many efforts to study and construct such operators in more general geometric situations (see e.g. \cite{or, kp, kp2, PZ,  Z, choi}). In this paper, we describe a method for building such operators   in the context of   Jordan algebras. To be  more explicit, let  $V$ be a simple real Jordan algebra, and let $\Co(V)$ be its   conformal group.  
It turns out to be more convenient  to work with  a   group $G$ which is  locally isomorphic   to $\mathrm{Co}(V)$; more precisely  $G$ is a twofold covering   of the {\em proper conformal group} (for details see \autoref{sec-DPS-KS}). This group is a simple Lie group which acts rationally on $V.$ The subgroup $P$ of affine conformal maps is a maximal parabolic subgroup of $G$. The pair $(G,P)$ has the following properties :
\begin{itemize}
\item[$-$]  the unipotent radical of $P$ is abelian,
\item[$-$]  $P$ is conjugate to its opposite subgroup $\overline P$.
\end{itemize}
Conversely, a pair $(G,P)$, where $G$ is a simple Lie group and $P$ is a parabolic subgroup of $G$,  satisfying  the two above properties is known to be associated to a simple real Jordan algebra. For earlier use of Jordan algebras from this point of view, see e.g. \cite{BSZ, HKM, KS, sahi1, sahi2, Mollers}.

The map
\[V\ni v \mapsto n_v \overline{P},\]
where $n_v$ is the translation $x\mapsto x+v$, has a dense open image in $\mathcal{X}=G/\overline P$ (the \emph{big Bruhat cell}).

The characters of $\overline P$ are parametrized by $(\lambda, \varepsilon) \in \mathbb C\times \{ \pm\}$ and we form the corresponding line bundles $\mathcal E_{\lambda, \varepsilon}$ over $\mathcal{X}$. The natural action of $G$ on $\Gamma(\mathcal E_{\lambda, \varepsilon})$, the space of  sections of  $\mathcal E_{\lambda, \varepsilon}$, gives raise to a  smooth representation $\pi_{\lambda, \varepsilon}$.  These representations constitute the \emph{degenerate principal series} of $G.$ This \emph{compact realization} of $\pi_{\lambda, \varepsilon}$ has another useful realization on functions defined on $V,$ called the \emph{non-compact} realization.

Let us come to the crucial result of this paper. Given two representations $\pi_{\lambda, \varepsilon}$ and $\pi_{\mu, \eta}$, we  construct a   differential operator on $\mathcal{X}\times \mathcal{X}$ \[F_{(\lambda ,\varepsilon),(\mu,\eta)}: \Gamma(\mathcal E_{\lambda, \varepsilon} \boxtimes \mathcal E_{\mu, \eta})
\longrightarrow \Gamma(\mathcal E_{\lambda+1, -\varepsilon} \boxtimes \mathcal E_{\mu+1, -\eta})
\]
which is covariant with respect to $(\pi_{\lambda, \varepsilon}\otimes \pi_{\mu,\eta}, \pi_{\lambda+1, -\varepsilon}\otimes \pi_{\mu+1,-\eta})$. To construct covariant bi-differential operators is then easy. In fact, let $\res$ be the restriction map from $\mathcal{X}\times \mathcal{X}$ to the diagonal $\{(x,x), \; x\in \mathcal{X}\}\simeq \mathcal{X}$. For any $N\in \mathbb N^*,$ the bi-differential operator 
\[B_{(\lambda,\varepsilon), (\mu,\eta)}^{(N)} : = \res \circ F_{(\lambda+N-1, \varepsilon), (\mu+N-1,\eta)}\circ\dots\circ F_{(\lambda,\varepsilon), (\mu,\eta)}\ 
\]
is covariant with respect to $(\pi_{\lambda, \varepsilon} \otimes \pi_{\mu, \eta}, \pi_{\lambda+\mu+2N, \varepsilon\eta})$.
This idea of obtaining covariant \emph{bi-differential operators on $\mathcal{X}$} from a covariant \emph{differential operator on $\mathcal{X}\times \mathcal{X}$} is reminiscent of the use of the $\Omega$-process in the classical construction of the transvectants.

The construction of $F_{(\lambda,\varepsilon),( \mu,\eta)}$ follows a process that was introduced in \cite{bc} for $V=\mathbb R^{n,0}$ and $G={SO}_0(1,n+1)$ and further used  in \cite{Clerc} for $V=\Mat(n,\mathbb  R)$ and   $G= SL(2n,\mathbb R).$ This approach  uses two main ingredients : the first one is the normalized Knapp-Stein operators $\widetilde J_{\lambda, \varepsilon}$ for the degenerate principal series which   intertwines $\pi_{\lambda, \varepsilon}$ and $\pi_{\frac{2n}{r}-\lambda, \varepsilon}$, where $n$ and $r$ denote   the dimension and the rank of $V$, respectively. The second   ingredient is an  operator $M$ which  in the non-compact realization  becomes  the multiplication operator is given by
\begin{equation*} 
Mf(x,y) = \det(x-y) f(x,y),\quad x,y\in V\times V.
\end{equation*}
Here $\det$ is the determinant polynomial of the Jordan algebra $V.$ The operator $M$ has a ``universal"  intertwining property due to the covariance of $\det(x-y)$ under the diagonal action of the group $G.$

The operator $F_{(\lambda, \varepsilon), (\mu,\eta)}$ corresponds to the following commutative diagram 
$$
\begin{CD}
\mathcal H_{(\lambda,\varepsilon), (\mu,\eta)} @> F_{(\lambda, \varepsilon), (\mu,\eta)}>>  \mathcal H_{(\lambda+1,-\varepsilon),(\mu+1,-\eta)}\\
@V \widetilde J_{\lambda,\varepsilon}\otimes \widetilde J_{\mu,\eta} V  V @AA\widetilde J_{\frac{2n}{r}-\lambda-1, -\varepsilon}\otimes\, \widetilde J_{\frac{2n}{r}-\mu-1,-\eta}A\\
\mathcal H_{(\frac{2n}{r}-\lambda,\varepsilon), (\frac{2n}{r}-\mu,\eta)} @>M> >\mathcal H_{(\frac{2n}{r}-\lambda-1, -\varepsilon),(\frac{2n}{r}-\mu-1,-\eta)}
\end{CD}
$$
where $\mathcal H_{(\lambda,\varepsilon), (\mu,\eta)} $ stands for $\Gamma(\mathcal E_{\lambda, \varepsilon}\boxtimes \mathcal E_{\mu, \eta})$, i.e. the space of smooth sections of the tensor bundle $\mathcal E_{\lambda, \varepsilon}\boxtimes \mathcal E_{\mu, \eta}$ over $\mathcal{X}\times \mathcal{X}$.

The fact that $F_{(\lambda, \varepsilon), (\mu,\eta)}$ is covariant (for the diagonal action of $G$) is an obvious consequence of its definition. The main difficulty is to show that $F_{(\lambda, \varepsilon), (\mu,\eta)}$ is a \emph{differential} operator.
This is done by working in the non-compact realization. Indeed, there are three steps. First, we prove the following  \emph{main identity}\footnote{We use the notation, for any $a \in\mathbb{R}^*$, $a^{s,+}=|a|^s$  and $a^{s,-}=\text{sign}(a)| a |^s$.} (\autoref{Dstgen}):

\begin{theorem*}[Main identity]
For $(s,\varepsilon)$ and $(t,\eta)$ in $\mathbb C\times \{\pm\}$, there exists a differential operator $D_{s,t}$ on $V\times V$ such that  for all  $(x,y)\in V^\times \times V^\times$, we have 
\begin{equation*}
\det\left(\frac{\partial}{\partial x}-\frac{\partial}{\partial y}\right)\circ \left[\det(x)^{s,\varepsilon}\det(y)^{t,\eta}\right]=\left[\det(x)^{s-1,-\varepsilon}\det(y)^{t-1,-\eta}\right]\circ D_{s,t}.
\end{equation*} 
The differential operator $D_{s,t}$ has polynomial coefficients in $x,y$ and also in $s,t$.
\end{theorem*}

The second step uses the Fourier transform $\mathcal F$ on $V$. The distributions $\det (x)^{s,\varepsilon}$, which are defined for $\Re(s)$ large enough, can be extended  by analytic continuation in   $s$, yielding a meromorphic family of tempered distributions on $V.$ Their Fourier transforms can be computed using the \emph{local Zeta functional equations}  on $V$ (\autoref{E-F-Z-moins}, \autoref{E-F-Z-eucl} and \autoref{E-F-Z-pq}). Let $I_{s,\varepsilon}$ be the convolution operator with the distribution $\det(x)^{s,\varepsilon},$ and denote by $E_{s,t}$  the differential operator with polynomials coefficients on $V$ defined by
\[\mathcal F\circ E_{s,t} = D_{s,t}\circ \mathcal F.
\]
 The Fourier transform counterpart of the main identity is the following \emph{main theorem}.
\begin{theorem*}[Main theorem] For $(s,\varepsilon)$ and $(t,\eta)$ in $\mathbb C\times \{\pm\},$ we have
\begin{equation*}
M\circ \left(I_{s,\varepsilon}\otimes I_{t,\eta}\right) =\kappa(s,t) \left(I_{s+1,-\varepsilon}\otimes I_{s+1,-\eta}\right) \circ E_{-s-\frac{n}{r}, -t-\frac{n}{r}},
\end{equation*}
where $\kappa(s,t)$ is rational function on $\mathbb C\times \mathbb C$. 
\end{theorem*}
The third step is to prove that for $\lambda$ and  $\mu$ generic, the local expression of $F_{(\lambda, \varepsilon),(\mu,\eta)}$ in the non compact setting is equal to $F_{\lambda,\mu}:=E_{\frac{n}{r}-\lambda, \frac{n}{r}-\mu}$, 
thus establishing that $F_{(\lambda, \varepsilon),(\mu,\eta)}$ is a differential operator. 

Many of our proofs (in particular the Fourier transform computations) depend on the type of the real Jordan algebra. However,  for two simple real Jordan algebras $V_1$ and $V_2$ which are real forms of the same simple complex Jordan algebra $\mathbb{V}$, the differential operator $F_{\lambda,\mu}$  constructed for $V_1\times V_1$ and the corresponding one  for $V_2\times V_2$ are restrictions of a common holomorphic differential operator  on $\mathbb{V}\times \mathbb{V}$. This remark advocates for a more algebraic construction (valid over $\mathbb C$) of the objects we have constructed by analytical means.

Let us describe the content of this paper. After a general formulation of Leibnitz's formula for  applying any constant coefficients differential operator to a product of two functions (\autoref{Leibnitz}) we obtain (\autoref{sec-Bernstein}) the \emph{Bernstein identity} for $\det(x)^{s,\varepsilon}$ in each of the four types of simple real Jordan algebras.  We then construct (\autoref{section-Dst}) the family of differential operators  $D_{s,t}$ by proving the \emph{main identity} above. The delicate Fourier calculations using the \emph{local Zeta functional equations} for a  real Jordan algebra are presented in \autoref{sec-Zeta}. Although these matters are largely known, some aspects are new : the functional equation for  $\det(x)^{s,-}$ in the non-euclidean split case,  the functional equation for  $\det(x)^{s,\pm}$  in the euclidean case   and the role played by  two ``new'' distributions $Z_s^\text{even}$ and $Z_s^\text{odd}$ (see \autoref{Z-odd-even}).   In \autoref{sec-Est} we define  the family $E_{s,t}$  and we prove Theorem 2 above. Up to this point, we mostly use  classical abelian harmonic analysis. We now shift to semi-simple harmonic analysis.  Some facts and results about the \emph{conformal group} of a simple real Jordan algebra, the associated \emph{degenerate principal series} and the corresponding \emph{Knapp-Stein operators} are presented in  \autoref{sec-DPS-KS}. The  main theorem in terms of covariance property with respect to representations of the  group  $G$ is given in \autoref{sec-cov-Fst}. The obtained covariance properties are then used to construct covariant bi-differential operators in \autoref{sec-Bst}. 
Finally, \autoref{sec-Rpq} is devoted to the case $V=\mathbb{R}^{p,q}$ and $G=\mathrm{O}(p,q)$  where more explicit computations can be made.  Two appendices are added : the first one contains the classification of simple real Jordan algebras, the second one intends to compensate for the  rare literature on the subject of simple real Jordan algebras and offers references and/or proofs.

 \tableofcontents

\section{A generalized Lebnitz's formula}\label{Leibnitz} 
To avoid confusion in the notation, we will apply the following conventional notation for constant coefficients differential operators on a finite dimensional vector space $V$ over $\mathbb{F}=\mathbb{R}$ or $\mathbb{C}$ endowed with an $\mathbb{F}$-bilinear non-degenerate form $(\cdot\,,\,\cdot)$ : for $p\in\mathbb{C}[V]$ let $p(\frac{\partial}{\partial x})$ to  
 be the (uniquely determined constant  coefficients) differential operator such that
\[p\left(\frac{\partial}{\partial x} \right)e^{(x,y)}(x) = p(y) e^{(x,y)}.
\]

Throughout  this section we will assume that $V$ is   a real finite dimensional vector space endowed with a euclidean inner product $(\cdot\,,\,\cdot)$.

 Let $\mathcal{P}(V)$ be the space of polynomials on $V$ with real coefficients. 
Define on $\mathcal P(V)$ the \emph{Fischer inner product} by
\begin{equation*}
(p,q)_F =\left(p\left(\frac{\partial }{\partial x}\right)q\right)(0).
\end{equation*}
The Fischer inner product is an euclidean inner product on $\mathcal P(V)$ and it satisfies
\begin{equation*}
(p,qr)_F = \left(r\left(\frac{\partial}{\partial x}\right) p,q\right)_F.
\end{equation*}

Fix     $\mathbf{p}$ in  $\mathcal P(V)$ and  let $\mathcal W({\bf p})$ be the subspace of $\mathcal{P}(V)$ generated by the partial derivatives of $\mathbf{p}$. Let $ \rho= \dim  \mathcal W({\bf p})$ and choose an orthonormal basis $(\mathbf{p}_1,\dots, \mathbf{p}_\rho)$ of $\mathcal W({\bf p})$.
For any $p\in\mathcal{P}(V)$, we define $p^\flat$ by 
$$p^\flat=p\left(\frac{\partial}{\partial x}\right)\!\mathbf{p}.$$
\begin{theorem}[Generalized Leibnitz's formula]\label{Lebnitz-fg} Let $f,g\in C^\infty(V)$. Then
\begin{equation}\label{Rfg}
\mathbf{p}\left(\frac{\partial}{\partial x}\right) (fg) = \sum_{j=1}^\rho \left(\mathbf{p}^\flat_j\left(\frac{\partial}{\partial x}\right)f\right) \left(\mathbf{p}_j\left(\frac{\partial}{\partial x}\right)g\right).
\end{equation}
 This formula can be  rewritten as 
\begin{equation}\label{Rfg2}
\mathbf{p\!}\left(\frac{\partial}{\partial x}\right)(fg) = \sum_{i=1}^\rho \sum_{j=1}^\rho  \big(\mathbf{p}\, ,\mathbf{p}_i\mathbf{p}_j\big)_F\,\left(\mathbf{p}_i\left(\frac{\partial}{\partial x}\right)f\right)  \left(\mathbf{p}_j\left(\frac{\partial}{\partial x}\right) g\right).
\end{equation}
\end{theorem}
\begin{proof}
Without loss of generality, it is possible to assume that $f$ and $g$ belong to $\mathcal P(V)$.  Since the differential operators $\mathbf{p\!}\left(\frac{\partial}{\partial x}\right),$ $ \mathbf{p}_j\left(\frac{\partial}{\partial x}\right)$ and $\mathbf{p}^\flat_j\left (\frac{\partial}{\partial x}\right) $     are invariant by translations, it is enough to prove \autoref{Rfg}  at $x=0$.

 From the definition of the Fischer inner product, we have
\[\mathbf{p}\left(\frac{\partial}{\partial x}\right)( fg)(0) = \left(\mathbf{p}\,,\,fg\right)_F= \left(f\left(\frac{\partial}{\partial x}\right)\mathbf{p}\,,\,  g \right)_F.
\]
Now  $\displaystyle  f\left(\frac{\partial}{\partial x}\right)\mathbf{p}$ belongs to $\mathcal W({\bf p})$, and therefore
\[f\left(\frac{\partial}{\partial x}\right)\mathbf{p} = \sum_{j=1}^\rho a_j(f) \mathbf{p}_j,
\]
where 
\begin{eqnarray*}
a_j(f) = \left( f\left(\frac{\partial}{\partial x}\right)\mathbf{p}\,,\,\mathbf{p}_j\right)_F
&=& (\mathbf{p}\,,\, f\mathbf{p}_j)_F\\ 
&=&\left( \mathbf{p}_j\left(\frac{\partial}{\partial x}\right)\mathbf{p}\,,\,f\right)_F 
= (\mathbf{p}^\flat_j\, ,f)_F= \mathbf{p}^\flat_j\left (\frac{\partial}{\partial x}\right)
f(0).
 \end{eqnarray*}
Hence
$$ \mathbf{p}\left(\frac{\partial}{\partial x}\right) (fg) (0)= \sum_{j=1}^\rho \left(\mathbf{p}^\flat_j\left (\frac{\partial}{\partial x}\right)
f\right)(0)\left( \mathbf{p}_j\left(\frac{\partial}{\partial x}\right)g \right)(0).$$
Now formula \autoref{Rfg2} can be obtained from \autoref{Rfg} by expressing the polynomials $\mathbf{p}^\flat_j$ in the basis $(\mathbf{p}_1,\dots, \mathbf{p}_\rho)$.
 Indeed, as $\mathbf{p}^\flat_j$ belongs to $\mathcal W({\bf p})$, we have
$$\mathbf{p}^\flat_j = \sum_{i=1}^\rho \big(\mathbf{p}^\flat_j\, , \,\mathbf{p}_i\big)_F\, \mathbf{p}_i,$$
with $$(\mathbf{p}^\flat_j, \mathbf{p}_i)_F= \left( \mathbf{p}_j\left(\frac{\partial}{\partial x}\right) \mathbf{p}, \mathbf{p}_i\right)_F
= \big(\mathbf{p},\mathbf{p}_i\mathbf{p}_j\big)_F.
$$
Therefore \autoref{Rfg2} follows immediately from \autoref{Rfg}.
\end{proof}
In the sequel we will also need a formula for computing $\mathbf{p}\left(\frac{\partial}{\partial x}\right)(fgh)$. 
\begin{proposition} Let $f,g,h\in C^\infty(V)$. Then
\begin{equation}\label{Rfgh} 
\mathbf{p}\left(\frac{\partial}{\partial x}\right)(fgh) = \sum_{i=1}^\rho \sum_{j=1}^\rho \sum_{k=1}^\rho a_{ijk}\left( \mathbf{p}_i\left(\frac{\partial}{\partial x}\right)f\right)\left( \mathbf{p}_j\left(\frac{\partial}{\partial x}\right)g\right)\left( \mathbf{p}_k\left(\frac{\partial}{\partial x}\right)h\right),
\end{equation}
where
\[a_{ijk} = \sum_{\ell=1}^\rho \big(\mathbf{p}, \mathbf{p}_i \mathbf{p}_\ell \big)_F\big(\mathbf{p}_\ell,  \mathbf{p}_j\mathbf{p}_k \big)_F.
\]
\end{proposition}
\begin{proof} This is a direct consequence of \autoref{Lebnitz-fg} where  formula \autoref{Rfg2} is used twice.

\end{proof}
 
\section{The Bernstein identity for $\det(x)^{s,\varepsilon}$}\label{sec-Bernstein}
\addsubsection{Background on simple real Jordan algebras}\label{background}

We will recall first some preliminary results on simple real Jordan algebras and fix notation. For more details we refer the reader to \cite{Bertram, BK, FK, Helwig}.

Let $V$ be a $n$-dimensional {\it real   Jordan algebra} with unit element ${\bf 1}$  and of rank $r$ (see \autoref{Appendix-B}). Denote by $L(x)\in \End(V)$ the multiplication by $x\in V$ and let $P(x)$ be the  quadratic operator defined by  
\begin{equation}\label{quad-rep}
P(x):=2L(x)^2-L(x^2).
\end{equation}

The {\it Jordan trace} $\tr$ is a linear form on $V$ and the {\it Jordan determinant} $\det$ is a homogeneous polynomial on $V$ of degree $r.$   In particular, they satisfy $\tr({\bf 1})=r$ and  $\det({\bf 1})=1.$  

 A real Jordan algebra $V$ is called {\it semisimple} if the symmetric bilinear  form
 $$(x,y):=\tr(xy)$$
 is non-degenerate on $V\times V$.   If in addition $V$ has no non-trivial ideal, then  $V$ is called  {\it simple}. If the  bilinear  form $(\cdot\,,\,\cdot) $ is positive definite,  then  $V$ is called    {\it euclidean} Jordan algebra. 
  

 An involutive automorphism $\alpha$ of $V$  is called  {\it Cartan involution} of $V$   if the symmetric bilinear form
\begin{equation}\label{innerprod}
\langle x ,y\rangle:=(\alpha(x),y)
\end{equation}
 is positive definite. 
 For  every semisimple Jordan algebra such a Cartan  involution always exists  and  two Cartan involutions are conjugate by an automorphism of $V$ (see \cite{Helwig}). 
  With respect  to the involution $\alpha$, the following orthogonal decomposition  holds
 $$V=V_+\oplus V_- ,$$
where $V_+:=\{x\in V;\;\; \alpha(x)=  x\}$ and $V_-:=\{x\in V;\;\; \alpha(x)=  -x\}$. The eigenspace $V_+$ is always a euclidean Jordan subalgebra of $V$ with the  same unit element ${\bf 1}.$ Notice that when $V$ is euclidean then $V_+=V$ and $V_- =\{0\}.$  Denote by $n_+$ and $r_+$  the dimension and the rank of $V_+$, respectively. 

An element $c\in V$ is said to be  {\it idempotent} if $c^2=c$. Further,  two idempotents $c_1$ and $c_2$ are called {\it orthogonal} if $c_1c_2=0$.  A non-zero idempotent  is called {\it primitive} if it cannot be written as the  sum of two non-zero orthogonal  idempotents.    

Every  set $\{c_1,\ldots,c_k\}$ of orthogonal primitive idempotents in $V_+$ with the additional condition $c_1+ \cdots+c_k={\bf 1}$ is called a {\it Jordan frame} in $V_+.$ By \cite[Theorem III.1.2]{FK}  the cardinal   of a  Jordan frame is always equal to the rank $r_+$ of $V_+$, and two Jordan frames are conjugate by an automorphism of $V_+$ (see \cite[Corollary IV.2.7]{FK}).

Fix a Jordan frame $\{c_1,\ldots,c_{r_+}\}$ in $V_+$.  By \cite[Proposition III.1.3]{FK} the spectrum of the multiplication operator $L(c_k)$ by $c_k$  is $\{0,\frac{1}{2},1\}$. Further, the operators  $L(c_1), \ldots, L(c_{r_+})$ commute and therefore are simultaneously diagonalizable. This yield  the following  {\it Peirce decomposition} 
\begin{equation}\label{ds}
V=\bigoplus_{1\leq i\leq j\leq r_+} V_{ij},
\end{equation}
where
$$\begin{array}{rll}
V_{ii}&:=V(c_i,1), & \quad1\leq i\leq r_+,\\
V_{ij}&:=V(c_i,\frac{1}{2})\cap V(c_j,\frac{1}{2}),& \quad 1\leq i<j\leq r_+.
\end{array}$$
Here  $V(c,\lambda)$ denotes the eigenspace of $L(c)$ corresponding to the eigenvalue $\lambda.$
Since the operators $L(c_k)$, for $0\leq k\leq r_+$, are symmetric with respect to the inner product \autoref{innerprod}, the direct sum \autoref{ds} is orthogonal.

Denote by $d$ the common dimension of the subspaces $V_{ij}$ ($i<j$) 
and by $e+1$ the common dimension of the subalgebras $V_{ii}.$   Then
$$\dim V=n=r_+(e+1)+ {r_+(r_+-1)} {d\over 2}.$$
 The Jordan algebra $V$ is called {\it split}  if $V_{ii}=\mathbb{R}c_i$ for every $1\leq i\leq r_+$ (equivalently $e=0$), otherwise $V$ is called {\it non-split}. By \cite{Helwig}, if $V$ is split  then $r=r_+$, otherwise $r=2r_+.$ We pin down   that every euclidean Jordan algebra is split.      
 

 There is  a classification of simple real Jordan algebra  given in  \cite{Helwig}  (see also \cite{Loos, HKM}). We refer the reader to \autoref{Appendix-A}  for the complete  list. 
More precisely, we have four types of algebras :
 
 \begin{tabular}{lll}
 Type I &:& $V$ is euclidean.\\ 
 Type II  &: & $V$ is split non-euclidean.\\ 
 Type III &:&  $V$ is  non-split  with no complex structure.\\
 Type IV &:&  $V$ is  non-split  with a complex structure.
 \end{tabular}
 
\smallskip
 
 \noindent Notice that every simple real Jordan algebra   is either :
\begin{enumerate}  
\item[(a)] a real form of a simple complex Jordan algebra (it is the case of type I, II, III);  or  
\item[(b)] a simple complex Jordan algebra viewed as a real one (it is the case of type IV).
\end{enumerate}
 
When  $V$ is a simple real Jordan algebra, the integers $n$, $r$, $d$ and $e$ are called the {\em structure constants of $V$}. 
 
 When dealing  with a simple   Jordan algebra over $\mathbb{C}$ we use the symbol $\mathbb{V}$ and we  denote by  $\mathbbl{n}$, $\mathbbl{r}$ and  $\mathbbl{d}$   its structure constants (over $\mathbb C$).

 
  Below we will establish the \emph{Bernstein} identity for $\det(x)^{s,\varepsilon},$ which takes slightly different forms depending on the type of the Jordan algebra.

We first consider the case where $V$ is a euclidean Jordan algebra where we recall from\cite{FK}  the Bernstein  identity  on the  (open) \emph{cone of squares} $\Omega$. We then use it to give the corresponding identity  for a simple complex Jordan algebra which is essentially obtained  by a holomorphic extension. This formula is  the key ingredient  to get the Bernstein identity on the set of invertible elements in each of the four  types of simple real Jordan algebras.

\addsubsection{The Bernstein identity for the cone $\Omega$}
Let $V$ be a simple euclidean Jordan algebra and let $\Omega$ be the associated symmetric cone (see \cite{FK} for more details on euclidean Jordan algebras). Throughout this paper, we will denote the Jordan determinant $\det$ of a euclidean Jordan algebra by $\Delta$. 
We recall that $\Delta(x)>0$ for all $x\in \Omega$.

For $k,l\in \mathbb N$,  let $b_{k,l}$ be the polynomial defined on $\mathbb{C}$ by
\begin{equation}\label{bkl}
b_{k,l} (s) = s\left(s+\frac{l}{2}\right)\ldots\left(s+(k-1)\frac{l}{2}\right).
\end{equation}  

\begin{proposition}[see {\cite[Proposition VII.1.4]{FK}}] For $s\in \mathbb C$  and $x\in \Omega$, we have
\begin{equation}\label{BSdetOmega}
\Delta\left(\frac{\partial }{\partial x}\right) \Delta(x)^s = b_{r,d}(s)\Delta(x)^{s-1}.
\end{equation}
\end{proposition}

\addsubsection{The Bernstein identity for a complex Jordan algebra}
Let $\mathbb V$ be a  simple complex Jordan algebra. 
Then $\mathbb V$ has a   simple euclidean real form $V$.
 The  structure constants   $\mathbbl{n}$, $\mathbbl{r}$ and  $\mathbbl{d}$  of $\mathbb V$ (viewed as a complex Jordan algebra) coincide with the  structure constants   $n$, $r$ and  $d$ of $V$ (see \autoref{Appendix-B}).  
As above, we will denote  the determinant of $V$ by $\Delta$. Then the determinant of $\mathbb V$ is the holomorphic extension  of $\Delta$, still  denoted by $\Delta$ (see \autoref{Appendix-B}).

Let $\mathbb V^\times := \{ z\in \mathbb V, \; \Delta(z)\neq 0\}$ be the open set of invertible elements in $\mathbb{V}$.    
\begin{proposition} For $s\in \mathbb C$ and for  $z\in \mathbb V^\times$, we have
\begin{equation}\label{BSdethol}
\Delta\left(\frac{\partial}{\partial z}\right)\Delta(z)^s = b_{{\rr}, {\dd}}(s) \Delta(z)^{s-1},
\end{equation}
where the powers of $\Delta$ are computed using a local branch of $\log \Delta$ near $z$.
\end{proposition}

\begin{proof}
   Let $z_0\in \mathbb V^\times$.   
   The subset $\mathbb V^\times $ being pathwise-connected,  it is possible to choose a simply connected neighborhood $\mathcal O_0\subset \mathbb V^\times$ of $z_0$ such that $\mathcal O_0\cap \Omega \neq \emptyset$. Here $\Omega$ denotes the symmetric cone in $V$. By analytic continuation, there is a branch of $\log \Delta (z)$ on $\mathcal O_0$ which coincides  with $\ln  \Delta (x) $ on $\mathcal O_0\cap \Omega$. For $s\in \mathbb C,$  use this branch   to define $\Delta(z)^s$ for  $z\in\mathcal O_0$ by
$\Delta(z)^s= e^{s\,\log  \Delta (z)}$ and similarly for $\Delta(z)^{s-1}$.  On $\mathcal O_0\cap \Omega$, \autoref{BSdethol} reduces to \autoref{BSdetOmega}. Now, both hand sides of \autoref{BSdethol} are holomorphic functions on $\mathcal O_0$. As they coincide on $\mathcal O_0\cap \Omega$, they have to coincide on $\mathcal O_0$, and in particular at $z_0$. Since  $z_0$ was arbitrary chosen in $\mathbb V^\times,$ the conclusion follows. Needless to say, if the identity is valid for \emph{one} local branch of $\log \Delta (z)$, then it is true  for \emph{any} local branch. 
\end{proof}

\addsubsection{The Bernstein identity for a real Jordan algebra of  type IV}
Let $V$ be a simple complex Jordan algebra viewed as a  simple real  Jordan algebra. Recall that its (real) determinant $\det$ satisfies
\[\det(z) = \Delta(z)\overline{\Delta(z)},
\]
where $\Delta$ is the (complex) determinant of $V$.  
 Denote by   $\mathbbl{n}$, $\mathbbl{r}$ and  $\mathbbl{d}$   the  structure constants  of $V$ viewed as a complex Jordan algebra,     and let   $e$, $n$, $r$   and $d$ be the structure constants  of $V$ viewed as a real Jordan algebra. Then, $e=1$,  $n=2\mathbbl{n}$, $r=2\mathbbl{r}$ and $d=2\mathbbl{d}$.

\begin{proposition} For $s\in \mathbb C$ and for $z\in V^\times$, we have
\begin{equation}\label{BSdetCR}
\Delta\left(\frac{\partial} {\partial z}\right)\overline{\Delta}\left(\frac{\partial}{\partial \overline{z}}\right) \det(z)^s=b_{{\rr}, {\dd}}(s)^2 \det(z)^{s-1}.
\end{equation}
\end{proposition}

\begin{proof}
As $\det(z)>0$ on $V^\times$,  then $\det(z)^s$ is well defined. For any local branch of $\log \Delta(z)$, we have
\[\Delta(z)^s \,\overline{\Delta(z)}^s = \det(z)^s.
\]
First  apply  the holomorphic differential operator $\Delta\left(\frac{\partial} {\partial z}\right)$ to this equality using \autoref{BSdethol}. Then apply  the conjugate holomorphic operator $\overline{\Delta}\left(\frac{\partial}{\partial \overline{z}}\right)$ to get \autoref{BSdetCR}.
\end{proof}

\addsubsection{The Bernstein  identity for a real Jordan algebra of  type I and II}
Let $V$ be   a split simple real Jordan algebra  (either euclidean or non-euclidean, see \autoref{Appendix-A}).  Then its complexification $\mathbb V$ is a simple  complex Jordan algebra. Let $\Delta$ be the determinant of $\mathbb V$.  In this case, the determinant $\det$ of $V$ (which is the restriction of $\Delta$ to $V$) takes both positive and negative values. Hence we introduce on $V^\times$ the two expressions $\det(x)^{s,+}$ and $\det(x)^{s,-}$ using the convention that  for $\xi\in \mathbb R^*$ and $s\in \mathbb C$,

\begin{equation}\label{t-plus-moins}
\xi^{s,\varepsilon}=
\begin{cases} 
 \vert \xi\vert^s &\text{for } \varepsilon=+\\
 \mathrm{sign}(\xi) \vert \xi\vert^s &\text{for } \varepsilon=- 
\end{cases}
\end{equation}
 In this case the relation between the structure constants of $V$ and $\mathbb{V}$ is given by  $r=\mathbbl{r}$, $n=\mathbbl{n}$ and $d=\mathbbl{d}.$
\begin{proposition} For $(s,\varepsilon)\in  \mathbb C\times\{\pm\}$ and   for $x\in V^\times$, we have
\begin{equation}\label{BSdetred}
\det\left(\frac{\partial}{\partial x}\right)\det(x)^{s,\varepsilon} = b_{\mathbbl{r},\mathbbl{d}}(s) \det(x)^{s-1,-\varepsilon}.
\end{equation}
\end{proposition}

\begin{proof} Let $x_0\in V^\times$ and assume first that $\det (x_0)= \Delta(x_0) >0$. Let  $\mathrm{Log}\,z$ be the principal branch of the logarithm function over $\mathbb C\smallsetminus (-\infty, 0]$. Then, for $z$  in a small neighborhood of $x_0$ in $\mathbb V^\times $, we may choose $\mathrm{Log} \Delta(z)$ as a branch of $\log \Delta(z)$, and therefore, 
 \[\det(x)^{s,-\varepsilon} = \Delta(x)^s,\qquad \det(x)^{s-1, \varepsilon} = \Delta(x)^{s-1}, \]
for all $x\in V^\times$ close to $x_0$. Hence, on the small neighborhood of $x_0$,  \autoref{BSdetred} follows from
\autoref{BSdethol}.

We now assume that $\det(x_0)<0$. For $z$ in a neighborhood of $x_0$ in $\mathbb{V}$, we choose $\mathrm{Log}(-\Delta(z))+\pi\sqrt{-1}$ as a local branch of $\log \Delta (z)$. Then, for $x\in V^\times$ close to $x_0$,
\[ \Delta(x)^s = \vert\det(x)\vert^se^{s\pi \sqrt{-1}},\qquad \Delta(x)^{s-1} = \vert\det(x)\vert^{s-1}e^{(s-1)\pi\sqrt{-1}}\ . \]
For $\varepsilon=+$, we have
$$
\det (x)^{s,\varepsilon} = \vert \det (x)\vert^s = e^{-s\pi\sqrt{-1}} \Delta(x)^s,$$
and
$$\det(x)^{s-1, -\varepsilon}= -\vert \det(x)\vert^{s-1}= e^{-s\pi\sqrt{-1}} \Delta(x)^{s-1}.
$$
For $\varepsilon=-$, we have 
$$
\det(x)^{s,\varepsilon}= -\vert \det(x)\vert^s = e^{(-s+1)\pi\sqrt{-1}} \Delta(x)^s,$$
and
$$
\det(x)^{s-1,-\varepsilon}= \vert \det(x)\vert^{s-1} = e^{-(s-1)\pi\sqrt{-1}} \Delta(x)^{s-1}.$$
 Hence,  on the small neighborhood of $x_0$, the identity \autoref{BSdetred} follows from \autoref{BSdethol}.  As $x_0$ was chosen arbitrary in $V^\times,$ the conclusion holds true for every $x\in V^\times.$
\end{proof}

\addsubsection{The Bernstein  identity for a real Jordan algebra of  type III}
Let $V$ be a non-split  simple real  Jordan algebra without complex structure   and let $\det$ be its determinant. Then its complexification $\mathbb V$ is a simple complex Jordan algebra and we denote by  $\Delta$ its determinant. In this case, $\det (x)=\Delta(x)\geq 0$ for all $x\in V$ (see \autoref{Appendix-B}), 
  $n=\mathbbl{n},$ 
$r=\mathbbl{r}$ and $\mathbbl{d}=1,$  $2,$  $k-2$ ($k\geq 2$) in accordance with $d=4,$ $8,$ $0$ (see \autoref{Appendix-A}).

\begin{proposition} Let $s\in \mathbb C$ and  $x\in V^\times$, we have
\begin{equation}\label{BSdetnred}
\det\left(\frac{\partial}{\partial x}\right)\det(x)^s = b_{\mathbbl{r}, \mathbbl{d}}(s)\, \det(x)^{s-1}.
\end{equation}
\end{proposition}
\begin{proof} 
 Let $x_0\in V^\times$. Then $\Delta(x_0) =  \det (x_0)>0$, so that $\Re\, \Delta(z)>0$ for all $z$ in  a small neighborhood of $x_0$ in $\mathbb V^\times$. Hence we may use $\mathrm{Log} \Delta (z)$ as a local branch of $\log \det(z)$ near $x_0$. For this choice of branch and for $x$ near $x_0$, we have
\[\Delta(x)^s = \det(x)^{s},
\]
and therefore  \autoref{BSdetnred} follows from \autoref{BSdethol}.
\end{proof}

\section{Construction of the family   $D_{s,t}$}\label{section-Dst}
\addsubsection{Construction for the cone $\Omega$}\label{sub-str}
Let $V$ be a simple euclidean Jordan algebra with  structure constants $n$, $r$, and $d$. As above, let $\Delta$ be the determinant polynomial of $V,$ and let 
$\mathcal P = \mathcal P(V)$ be the space of real-valued polynomials on $V.$ Let $\mathcal W({\Delta})$ be the subspace of $\mathcal P$ generated by the partial derivatives of $\Delta$  (cf.  \autoref{Leibnitz}).  
 
Let $\Str(V)$ be the  structure group of $V$ defined as the set of $g\in\mathrm{GL}(V)$ such that
$$ (gx)^{-1}= (g^{-1})^t(x^{-1}), \quad \text{for all } \, x\in V^\times.
$$
The group $\Str(V)$ acts on the space $\mathcal P$  by
\[  \pi(g) p  = p\circ g^{-1},\quad g\in \Str(V).
\] 
The space $\mathcal W({\Delta})$ is invariant under this action of  $\Str(V)$. The representation $\pi$ (more exactly its complex extension)  was studied in \cite{FG}, showing that  it is multiplicity free and producing its explicit decomposition in irreducible components.

Fix a Jordan frame $\{ c_1,c_2,\dots, c_r\}$ of $ V$. Let $  \Delta_1,\dots, \Delta_r=\Delta$ be the associated \emph{principal minors} on $V$ (see \cite[Page 114]{FK}).  For our convenience  we  write $\Delta_0=1$. 
For $0\leq k\leq r$, we denote by  $\mathcal W({\Delta_k})$  the  subspace of $\mathcal{P}$ generated by $\pi(g)\Delta_k$, for $g\in\Str(V)$.
By {\cite[Proposition 6.1]{FG}}, the subspaces  $\mathcal W({\Delta_k})$ are absolutely  irreducible under the action of $\Str(V)$.

   For  $p\in  \mathcal W({\Delta})$, denote by $p^\sharp$ the function defined on $ V^\times$ by 
\[p^\sharp(x) = p(x^{-1}) \Delta(x).
\]
 The function $p^\sharp$  extends as a polynomial function   on $V$.


\begin{proposition}[see {\cite[Proposition XI.5.1]{FK}}] Let $p\in \mathcal W({\Delta_k})$. Then, for $x\in \Omega,$ we have
\begin{equation}\label{BS-for-p}
p\left( \frac{\partial}{\partial x}\right)\Delta^s(x) = b_{k,d}(s)\, p^\sharp(x)\, \Delta(x)^{s-1},
\end{equation}
where $b_{k,d}(s)$ is given by \autoref{bkl}.
\end{proposition}
 
 Let $d_k = \dim  \mathcal W({\Delta_k})$ and choose a basis $\{p_{j,k},\; 1\leq j\leq d_k\}$ of $\mathcal W({\Delta_k})$ which is orthonormal for the Fischer inner product.  
 In the present context, the Leibnitz formula \autoref{Rfgh} can be written   for $\mathbf{p}=\Delta$ as following :

\begin{proposition}\label{Deltafgh}
 Let $f,g,h$ be three smooth functions on $V$. Then there exist real numbers $a^{(lmn)}_{ijk}$ such that 
\begin{equation*}
 \Delta(fgh) =
  \sum_{l,m,n\geq 0 \atop    l+m+n=r} \sum_{i=1}^{d_l}\sum_{j=1}^{d_m}\sum_{k=1}^{d_n}a^{(lmn)}_{ijk}
\left(p_{i,l}\left( \frac{\partial}{\partial x}\right)f\right)\left(p_{j,m}\left( \frac{\partial}{\partial x}\right)
g\right)\left(p_{k,n}\left( \frac{\partial}{\partial x}\right)h\right).
\end{equation*}
\end{proposition}

The vector  space $V\times V$ is naturally endowed with an inner product, and so is the vector space $\mathcal P(V\times V)$ of polynomial functions on $V\times V.$ To each $p\in \mathcal P(V\times V)$, we associate a differential operator $p\left(\frac{\partial}{\partial x}, \frac{\partial}{\partial y} \right)$. When $p$ is given by $p(x,y) = q(x-y)$, with $q\in\mathcal{P}(V)$, we use the notation $p\left(\frac{\partial}{\partial x}, \frac{\partial}{\partial y} \right) = q\left(\frac{\partial}{\partial x}-\frac{\partial}{\partial y}\right)$.

\begin{theorem}\label{Dst}
 For any $s,t\in \mathbb C$, there exists a differential operator $D_{s,t}$ on $V\times V$ such that for any smooth function f on $\Omega \times \Omega$,
\begin{equation}\label{DstOmega}
\Delta\left( \frac{\partial}{\partial x} -  \frac{\partial}{\partial y}\right) \Delta(x)^s\Delta(y)^t f(x,y) =\Delta(x)^{s-1} \Delta(y)^{t-1} \left(D_{s,t} f\right)(x,y).
\end{equation}
The operator $D_{s,t}$ has polynomial coefficients in $x,y$ and in $s,t$.
\end{theorem}
\begin{proof} Let $\varphi$ and $\psi$ be two functions   so that 
$$\varphi(u,v) = \psi(u,v-u) \quad\text{ or equivalently} \quad \varphi(x,x+y) = \psi(x,y),$$
where $(u,v)\in\{(u,v)\in V\times V\;:\; u\in \Omega, v-u\in \Omega\}$ or equivalently $x,y\in \Omega$. Then
$$\Delta \left(  \frac{\partial}{\partial x} -   \frac{\partial}{\partial y} \right) \psi(x,y) =  \Delta\left( {\frac{\partial}{\partial u} } \right)\varphi   (u,v)_{\big | u=x, \,v=x+y}.$$
Assume that $\psi(x,y)=\Delta(x)^s\Delta(y)^tf(x,y).$ Then $\varphi(u,v)= \Delta(u)^s \Delta(v-u)^t f(u, v-u).$ 
Fix $v$ and apply now \autoref{Deltafgh} on the  open set $\Omega \cap (-v+\Omega)$ of $V$ to get
$$\Delta\left( {\frac{\partial}{\partial u} } \right)\left(\Delta(u)^s \Delta(v-u)^t f(u, v-u) \right)  =
\sum_{l=0}^r \sum_{i=1}^{d_l} 
q_{i,l}(u,v;s,t)\left(p_{i,l} \left( \frac{\partial}{\partial u} \right) f \right)(u, v-u),$$
where
$$q_{i,l}(u,v;s,t)=\sum_{m,n\geq 0 \atop    m+n=r-l}
\sum_{j=1}^{d_m}\sum_{k=1}^{d_n}  a_{ijk}^{(lmn)}  \,
\left(p_{j,m}\left(  \frac{\partial}{\partial u}  \right) \Delta(u)^s  \right)
 \left(p_{k,n}\left ( \frac{\partial}{\partial u}  \right) 
\Delta(v-u)^t\right).$$
Since 
$$ p_{k,n} \left( \frac{\partial}{\partial u} \right) 
\Delta(v-u)^t = (-1)^n \left(p_{k,n} \left( \frac{\partial}{\partial u}  \right)
\Delta(\,\cdot\,)^t \right)(v-u),$$
it follows from \autoref{BS-for-p} that 
\begin{multline}p_{j,m} \left(  \frac{\partial}{\partial u}  \right) \Delta(u)^s\, p_{k,n}  \left( \frac{\partial}{\partial u} \right) \Delta(v-u)^t  = \\
(-1)^n\,b_{m,d}(s) b_{n,d}(t) \Delta(u)^{s-1} \Delta (v-u)^{t-1}  p_{j,m}^\sharp(u)\, p_{k,n}^\sharp(v-u).
\end{multline}
The final formula is now a matter of putting the pieces together.
 \end{proof}

\addsubsection{Extension of $D_{s,t}$ to a  complex   Jordan algebra}
Let $\mathbb V$ be a simple complex Jordan algebra, and let $V$ be a euclidean real form of $\mathbb V$.  The differential operator $D_{s,t}$ constructed in  \autoref{Dst} has a natural extension to $\mathbb V\times \mathbb V$ as a holomorphic differential operator $\mathbb D_{s,t}$ by replacing $\frac{\partial}{\partial x}$ and $\frac{\partial}{\partial y}$ by $\frac{\partial}{\partial z}$ and $\frac{\partial}{\partial w}$, respectively, and extending holomorphically the polynomial coefficients to $\mathbb V\times \mathbb V$.

\begin{theorem}\label{Dsthol}
 For any smooth function f on $\mathbb V^\times \times \mathbb V^\times$, we have
\begin{equation}\label{DstRC2}
\Delta\left( \frac{\partial}{\partial z} -  \frac{\partial}{\partial w}\right) \Delta(z)^s\Delta(w)^t f(z,w) =\Delta(z)^{s-1} \Delta(w)^{t-1} \left({\mathbb D}_{s,t} f\right)(z,w),
\end{equation}
where the powers of $\Delta$ are computed with respect to a (any) local branch of $\log \Delta(\, .\,)$ near $z$ and near $w$.
\end{theorem}
\begin{proof} To prove this equality between two holomorphic differential operators, it is enough to prove the equality for  a holomorphic function $f$  on $\mathbb V^\times \times \mathbb V^\times$. But then, arguing as in the proof of  \autoref{BSdethol}, the equality follows from  \autoref{Dst} by analytic continuation, 
\end{proof}

\addsubsection{The construction  for  a  real  Jordan algebra of type IV}
Let $\mathbb V$ be a simple complex Jordan algebra, for which we keep notation as in the previous subsection. In particular $\Delta$ denotes its determinant. When $\mathbb{V}$ is viewed  as a (simple) real   Jordan algebra $V,$ its determinant polynomial is given by
$\det(z) = \Delta(z) \overline{\Delta(z)}.$

For  a holomorphic differential operator $\mathbb D$  on $\mathbb V$ with polynomial coefficients,  we let $\overline{\mathbb D}$ be its associated conjugate-holomorphic  differential operator.

\begin{theorem}For any smooth function $f$  defined on $ V^\times \times  V^\times$,  we have
\begin{equation}\label{doubleDst}
\det\left(\frac{\partial}{\partial z}-\frac{\partial}{\partial w}\right)\det(z)^s\det(w)^tf(z,w)= \det(z)^{s-1}\det(w)^{t-1} \left({\mathbb D}_{s,t}\overline{{\mathbb D}}_{s,t}f\right)(z,w).
\end{equation}
\end{theorem}
\begin{proof} We argue as in the proof of \autoref{BSdetCR}.  Indeed, \autoref{Dsthol} establishes the equality of two holomorphic differential operators. Compose each side with its conjugate differential operator and use the fact that a holomorphic differential operator commutes with
its conjugate-holomorphic differential operator to obtain \autoref{doubleDst}.
\end{proof}

\addsubsection{The construction  for  a real Jordan algebra of type I and II}
 Let $V$ be a  euclidean   or a split non-euclidean Jordan algebra (see \autoref{Appendix-A}).  Then its complexification $\mathbb V$ is a simple  complex Jordan algebra. 
 

 Let $\mathbb D=p\left(z,\frac{\partial}{\partial z}\right)$ be a holomorphic differential operator on $\mathbb{V}$. Then its restriction to $V$  is the differential operator $D=p(x,\frac{\partial}{\partial x})$. We extend this notation to differential operators on $V\times V$.
\begin{theorem}\label{main-id-split} For any smooth function $f$ on $V^\times\times V^\times$, we have
\begin{equation}\label{Dstred}
\det\left(\frac{\partial}{\partial x}-\frac{\partial}{\partial y}\right)\det(x)^{s,\varepsilon}\det(y)^{t,\eta}f(x,y)=\det(x)^{s-1,-\varepsilon}\det(y)^{t-1,-\eta}D_{s,t}f(x,y).
\end{equation}
\end{theorem}

\begin{proof}To prove the equality of two differential operators, it is enough to prove that they coincide on  polynomial functions. Hence we may assume that $f$ is the restriction to $V^\times\times V^\times$ of a holomorphic polynomial function on $\mathbb V\times \mathbb V$. We now argue as in the proof of \autoref{BSdetred} to deduce \autoref{Dstred} from \autoref{Dsthol}.
\end{proof}

\addsubsection{The construction   for a  real Jordan algebra of type III}
Let $V$ be a non-split  simple real  Jordan algebra without complex structure. Then its complexification $\mathbb V$ is a simple complex Jordan algebra. 
For   a holomorphic differential operator $\mathbb D$ on $\mathbb V\times \mathbb V$, we denote by $D$ its restriction to $V\times V.$

\begin{theorem} For any smooth function $f$ on $V^ \times\times V^\times$, we have
\begin{equation}\label{Dstnonred}
\det\left(\frac{\partial}{\partial x}-\frac{\partial}{\partial y}\right)\det(x)^{s}\det(y)^{t}f(x,y)= \det (x)^{s-1} \det (y)^{t-1} \big(D_{s,t} f\big)(x,y).
\end{equation}
\end{theorem}
\begin{proof} The proof goes along the same lines as that  of the Bernstein identity \autoref{BSdetnred}. 
\end{proof}

\addsubsection{General formulation of the main identity}
In summary we have proved the following statement.
\begin{theorem}[Main identity]\label{Dstgen}
 Let $V$ be a simple real Jordan algebra $V$. For $(s,\varepsilon)$ and $(t,\eta)$ in $\mathbb{C}\times\{\pm\}$, there exists a differential operator $D_{s,t}$ on $V\times V$ such that for any smooth function $f$ on $V^\times \times V^\times$,
\begin{equation*}
\det\left(\frac{\partial}{\partial x}-\frac{\partial}{\partial y}\right)\det(x)^{s,\varepsilon}\det(y)^{t,\eta}f(x,y)=\det(x)^{s-1,-\varepsilon}\det(y)^{t-1,-\eta}D_{s,t}f(x,y).
\end{equation*} 
The differential operator $D_{s,t}$ has polynomial coefficients in $x,y$ and in $s,t$.
\end{theorem}

\section{Local Zeta functional equations}\label{sec-Zeta}
Let $V$  be a simple  real  Jordan algebra. 
Let  $\mathcal{S}(V)$  be the space of rapidly decreasing smooth functions on $V$ and let $\mathcal{S}'(V)$ be its dual, the space of tempered distributions on $V$.    

For $(s,\varepsilon)\in\mathbb{C}\times\{\pm\}$, 
we consider the following \emph{local Zeta integrals}
\begin{equation}\label{defzeta}
Z_{s,\varepsilon}(f)=\int_V f(x) \det(x)^{s,\varepsilon}dx,\quad  f\in \mathcal{S}(V),
\end{equation}
where $dx$ is the Lebesgue measure on $V$ and $\det(x)^{s,\varepsilon}$ is defined by \autoref{t-plus-moins}.  For $s\in\mathbb{C}$ with $\Re\, s>0$, the function $\vert\det(x)^{s,\varepsilon}\vert$ is  locally integrable on $V$. Therefore, 
\autoref{defzeta} defines a tempered distribution on $\mathcal S'(V)$.

When $\varepsilon = +$, it is known that the $ \mathcal{S}'(V)$-valued  function $s\mapsto Z_{s,+}$ extends to a meromorphic function on $\mathbb{C}$ (see for instance \cite{FK} for the euclidean case, \cite{BSZ} for the non-euclidean case except $\mathbb{R}^{p,q}$ and \cite{GS} for $\mathbb{R}^{p,q}$).
 
 Suppose now $\varepsilon=-$.
When $V$ is non-split (that is of type III and IV), $\det (x)$ is non-negative on $V$, therefore $Z_{s,-} = Z_{s,+}$ and we may drop the index $\pm$ in the notation.
When $V$ is split (that is of type I and II), the   meromorphic extension of $Z_{s,-}$ can be deduced  from the Bernstein  identity  \autoref{BSdetred} for the split case. 

Below, we will compute the Fourier transform of the distributions $Z_{s,\varepsilon}$. It is a classical subject in the literature under the name of \emph{local Zeta functional equation} (see  e.g.  \cite{BSZ, BR, Muller, SF, SS, Kayoya}).

The dual space $V'$   of $V$ will be identified with $V$ via the non-degenerate bilinear form
$$(x,y) =\tr(xy).$$
The Fourier transform of $f\in\mathcal{S}(V)$ is defined by  
\begin{equation}\label{def-Fourier}
\mathcal{F}(f)(x)=\hat{f}(x)=\int_V e^{2\sqrt{-1}\pi(x, y)}f(y)dy,
\end{equation} 
and extend it by duality to  the space  $\mathcal{S}'(V)$ of tempered distributions.
For $p\in\mathcal{P}(V)$, we recall the following classical 
formulas:
\begin{equation}\label{formule-Fourier}
   \mathcal{F}\left(p\left(\frac{\partial}{\partial x}\right)f\right)(x) =p\left(-2\pi\sqrt{-1}x\right)\mathcal{F}(f)(x),\;\;
  \mathcal{F}\left(pf\right)(x) =p\left(\frac{1}{2\pi\sqrt{-1}} \frac{\partial}{\partial x}\right)\mathcal{F}(f)(x).
 \end{equation}

 \addsubsection{Zeta functional equations for a  non-euclidean  Jordan algebra except $ \mathbb R^{p,q}$} 
    In this subsection  $V$ is of type II, III and IV except $\mathbb R^{p,q}$. The case $V=\mathbb{R}^{p,q}$ will be treated in \autoref{sec-Rpq}.    Recall from above that   the determinant polynomial $\det $ takes only positives values whenever $V$ is of type III and IV  ($V$ is non-split), while $\det $ takes positives   as well as negative values whenever $V$  is of type II ($V$ is split).
  
Recall that  $r_+$ denotes the split rank of $V$, and therefore   $r=r_+$ if $V$ is split and 
   $r=2r_+$ if $V$ is non-split; see \autoref{background}. For  $s\in\mathbb{C}$, let 
   \begin{equation}\label{Gamma-Non-Eucl}
   \Gamma_V(s):=\prod_{k=1}^{r_+}\Gamma\left(\frac{s}{2}-(k-1)\frac{d}{4}\right).
   \end{equation}






The following result can be found in \cite[Theorem 4.4]{BSZ}.  Our $\det$ is related to $\nabla$  in \cite{BSZ} by the relation $|\det|=\nabla$ whenever $V$ is split and  by $|\det|=\nabla^2$ whenever  $V$ is non-split.
\begin{theorem}\label{E-F-Z-plus}
 For every $s\in\mathbb{C}$,
the Fourier transform of the tempered distribution $Z_{s,+}$ is given by
\begin{equation}\label{eq-E-F-Z-plus}
\mathcal{F}(Z_{s,+})= \begin{cases}
\displaystyle \pi^{-rs-\frac{n}{2}}\frac{\Gamma_V\left(s+\frac{n}{r}\right)}{\Gamma_V(-s)}Z_{-s-\frac{n}{r},+}&\text{ Type II (split case)}, \\
\displaystyle \pi^{-rs-\frac{n}{2}}\frac{\Gamma_V\left(2s+\frac{2n}{r}\right)}{\Gamma_V\left(-2s\right)}Z_{-s-\frac{n}{r},+}&\text{ Type III and IV (non-split case)}. \\
\end{cases}
\end{equation}
\end{theorem}
  
Now let us consider the case $\varepsilon=-$. Recall that when $V$ is non-split non-euclidean  (that is of type III and IV), $Z_{s,-} = Z_{s,+}$.
 
\begin{theorem}\label{E-F-Z-moins} 
Assume that $V $ is a split non-euclidean Jordan algebra ($\not \cong \mathbb R^{p,q}$) of dimension $n$ and rank $r.$
 We have 
\begin{equation*}
 \mathcal{F}(Z_{ s,-}) = (\sqrt{-1})^r\pi^{-rs-\frac{n}{2}}  \frac{\Gamma_V(s+1+\frac{n}{r})}{\Gamma_V(-s+1)}
   Z_{  -s-{n\over r},-}.
  \end{equation*}
 \end{theorem}
 \begin{proof} 
  Recall from the Bernstein identity \autoref{BSdetred} that\footnote{In this case $r=\mathbbl{r}$ and $d=\mathbbl{d}$.}
$$  \det(x)^{s,-}={1\over {b_{r,d}(s+1)}} \det\left(\frac{\partial}{\partial x}\right)   \det(x)^{s+1,+}.$$
 In view of \autoref{formule-Fourier}  and  the functional equation \autoref{eq-E-F-Z-plus} for   $Z_{s,+}$, we obtain
 \begin{eqnarray*}
 \mathcal{F}(Z_{s,-})
 &=& \frac{1}{b_{r,d}(s+1)} \det(-2\pi\sqrt{-1}x)  \mathcal F(Z_{s+1,+})\\
 &=& \frac{(-2\pi\sqrt{-1})^r}{b_{r,d}(s+1)}  \pi^{-r(s+1)-\frac{n}{2}}\frac{\Gamma_V(s+1+\frac{n}{r})}{\Gamma_V(-s-1)} \det(x)  Z_{-s-1-\frac{n}{r},+}   \\
  &=& \frac{(-2\sqrt{-1})^r  \pi^{-rs-\frac{n}{2}}}{b_{r,d}(s+1)}  \frac{\Gamma_V(s+1+\frac{n}{r})}{\Gamma_V(-s-1)}Z_{-s-\frac{n}{r},-}  \\
  &=&(\sqrt{-1})^r\pi^{-rs-\frac{n}{2}}  \frac{\Gamma_V(s+1+\frac{n}{r})}{\Gamma_V(-s+1)}Z_{-s-\frac{n}{r},-}.  
\end{eqnarray*}
  \end{proof}
 In summary we have proved that
  \begin{equation}\label{const-c(e,s)}
  \mathcal{F}(Z_{s,\varepsilon})=c(s,\varepsilon) Z_{-s-\frac{n}{r},\varepsilon},
  \end{equation}
  where
  \begin{equation}\label{valeur-c(e,s)}
 c(s,\varepsilon)=
  \begin{cases}   
 \displaystyle  \pi^{-rs-\frac{n}{2}}\frac{\Gamma_V\left(s+\frac{n}{r}\right)}{\Gamma_V(-s)}& \text{ if $\varepsilon=+$ and $V$ of type II (split) }\\
    \displaystyle (\sqrt{-1})^r\pi^{-rs-\frac{n}{2}}  \frac{\Gamma_V(s+1+\frac{n}{r})}{\Gamma_V(-s+1)}& \text{ if $\varepsilon=-$ and $V$ of type II (split)  }\\ 
  \displaystyle   \pi^{-rs-\frac{n}{2}}\frac{\Gamma_V\left(2s+\frac{2n}{r}\right)}{\Gamma_V(-2s)} &\text{ if $\varepsilon=\pm$ and $V$ of type III and IV (non-split)}.\\
  \end{cases}
  \end{equation}
  \begin{remark}{\rm 
 One can prove that the function $s\mapsto \widetilde{Z}_{s,\varepsilon}:=\tilde{c}(s,\varepsilon)^{-1}Z_{s,\varepsilon}$ where
$$\tilde{c}(s,\varepsilon)=
\begin{cases}
\Gamma_V(s+\frac{n}{r}) & \text{ if $\varepsilon=+$ and $V$ of type II (split) }\\
\Gamma_V(s+1+\frac{n}{r})  & \text{ if $\varepsilon=-$ and $V$ of type II (split) }\\
\Gamma_V(2s+\frac{2n}{r}) & \text{ if $\varepsilon=\pm$ and $V$ of type III and IV (non-split) },
\end{cases}$$
admits an analytic continuation as entire function of $s$ in $\mathbb{C}$. The case $\varepsilon=+$ goes back to \cite{BSZ}.
}
  \end{remark}  
  
 \addsubsection{Zeta functional equations for euclidean Jordan algebras}
Let  $V$ be a $n$-dimensional simple euclidean Jordan algebra (that is of type I) of rank $r$, and denote as usual by $\Delta$ its determinant. 
Let   $\{c_1,\ldots,c_r\}$ be a Jordan frame of $V.$ It is known (see for instance \cite{FK}) that every $x\in V$ can be written as 
 \begin{equation}\label{dec-spec}
 x=k\Big(\sum_{i=1}^r \lambda_jc_j\Big),
 \end{equation}
 where $\lambda_1\geq\cdots\geq \lambda_r$ and $k$  is an element of  the identity component  of the group of automorphisms of $V$. The $\lambda_i$'s in \autoref{dec-spec} are uniquely determined by $x$. Further, $x$ is invertible if and only if $\lambda_i\not=0$ for all $i$. We say that $x$ is of   {\it signature}   $(r-i, i)$ if $\lambda_1\geq \cdots\geq \lambda_{r-i}>0>\lambda_{r-i+1}\geq \cdots \geq \lambda_r$.
Denote by  $\Omega_i$ the set of all elements of signature $(r-i,i)$. Then the
 set of invertible elements $V^\times$  decomposes into the disjoint union as 
 $$V^\times =\bigcup_{i=0}^r\Omega_i.$$
  In particular $\Omega_0$ coincides with the symmetric cone $\Omega$ of $V$.  
  
  For $s\in\mathbb{C}$, let  
 $$ 
 \Gamma_\Omega(s)
 = (2\pi)^{\frac{n-r}{2}}\prod_{j=1}^r\Gamma\left(s-(j-1)\frac{d}{2}\right)
 $$
 be the Gindikin gamma function. 

For $s\in \mathbb C$, $ f\in\mathcal{S}(V)$ and  $0\leq i\leq r$, the Zeta integrals $Z_i(f,s)$ defined by
 $$Z_i(f,s)=\int_{\Omega_i} f(x)|\Delta(x)|^sdx $$
converges  for $\Re(s)>0$, and has a meromorphic continuation to   $\mathbb{C}$.
Further, they satisfy  the  functional equation
\begin{equation}\label{F-Eq-Eucl}
Z_i(\hat{f},s-\frac{n}{r})=(2\pi)^{-rs}e\left(\frac{rs}{2}\right)\Gamma_{\Omega}(s)\sum_{j=0}^ru_{ij}(s)Z_j(f,-s),
\end{equation}
where $u_{ij}(s)$ are polynomials in  $e\left(-{s}/{2}\right)$ with $e(z):=e^{2\pi\sqrt{-1}z}$; see \cite{SS, SF}.

  Put $x=e(-s/2)$ and write $u_{ij}(x)$ for $u_{ij}(s)$. Then, it is proved in  \cite{SF} that  the matrix coefficients $u_{ij}(x)$  satisfy  
\begin{equation}\label{E-eq-f-FS}
\sum_{i=0}^r y^iu_{ij}(x)=\xi^{-(r-j)}P_j(\xi x,y)P_{r-j}(1,\xi xy), \quad  \forall y\in \mathbb R,
\end{equation} 
where $\xi:=(\sqrt{-1})^{d(r+1)}$ and 
\begin{equation}\label{coef-mat-FS}  
P_j(x,y)=\begin{cases}
(x+y)^j & \text{if $d$ is even},\\
(x+y)^{ \lfloor  \frac{j}{2} \rfloor }(y-x)^{j-\lfloor  \frac{j}{2} \rfloor} &\text{if $d$ is odd}\, .
\end{cases}
\end{equation} 

Recall that if $x\in \Omega_i$,  then $x$ is of signature $(r-i,i)$, and therefore $\Delta(x)=(-1)^{i}|\Delta(x)|$. Thus we may rewrite the local Zeta integrals \autoref{defzeta} in terms of the $Z_i$'s as   follows 
\begin{equation}\label{Z+-Z-and Zj}
 \begin{array}{l @{\; =\;} l}
Z_{s,+}(f)&  \displaystyle \int_V f(x) |\Delta(x)|^sdx= \sum_{i=0}^r Z_i(f,s),\\
Z_{s,-}(f)&  \displaystyle \int_V f(x) \mathrm{sgn}(\Delta(x))|\Delta(x)|^sdx= \sum_{i=0}^r (-1)^i  Z_i(f,s).
 \end{array}
 \end{equation}
That is 
\begin{equation}\label{Z+-Z-pair-impair}
 \begin{array}{l @{\; =\;} l}
\displaystyle Z_{s,+}(f)&  \displaystyle \sum_{k=0}^{\lfloor \frac{r}{2}\rfloor} Z_{2k}(f,s)+ \sum_{k=0}^{\lfloor \frac{r-1}{2}\rfloor} Z_{2k+1}(f,s),\\
\displaystyle Z_{s,-}(f)&  \displaystyle \sum_{k=0}^{\lfloor \frac{r}{2}\rfloor} Z_{2k}(f,s)- \sum_{k=0}^{\lfloor \frac{r-1}{2}\rfloor} Z_{2k+1}(f,s).\\
 \end{array}
\end{equation}

 Let us introduce two more tempered distributions :
 \begin{equation}\label{Z-odd-even}
 Z^{\text{even}}_s(f)=Z^{\text{e}}_s(f)=\sum_{k=0}^{\lfloor \frac{r}{2}\rfloor} (-1)^kZ_{2k}(f,s)\quad\text{and }\quad  Z^{\text{odd}}_s(f)=Z^{\text{o}}_s(f)=\sum_{k=0}^{\lfloor \frac{r-1}{2}\rfloor}  (-1)^kZ_{2k+1}(f,s).
 \end{equation}
 We are now in a position to examine the functional equations for  $Z_{s,+}$ and $Z_{s,-}$.
 
 According to the classification of simple euclidean Jordan algebras (see \cite{FK, SF} or \autoref{Appendix-A}),  we will consider the following (all) possibilities  :
 $$
 \begin{array}{l @{\; :\quad} l}
 \text{Case (a)}  &  d\equiv0\ (\mathrm{mod}\ 4) \text{ or } d\equiv2\ (\mathrm{mod}\ 4) \text{ and } r \text{ odd},\\
  \text{Case (a')}  &  d\equiv 2\ (\mathrm{mod}\ 4) \text{ and }  r \text{ even},\\
   \text{Case (b)} &  r=2 \text{ and }  d \text{ odd},\\
      \text{Case (c)}  &  r  \text{ arbitrary and }  d=1.\\
 \end{array}
 $$
 
 \begin{theorem}\label{E-F-Z-eucl} For $s\in \mathbb C$, let $\gamma(s):=(2\pi)^{-rs}e(\frac{rs}{4})\Gamma_\Omega(s)$. Then   the following functional equations hold.
 
{\bf Case (a)} :  If $d\equiv 0\ (\mathrm{mod}\ 4)$ or  $d\equiv 2\ (\mathrm{mod}\ 4)$ and $r$ odd, then
\begin{equation*}
\mathcal{F}
\begin{pmatrix}
Z_{s,+} \\
Z_{s,-}
\end{pmatrix}
= 2^r \gamma(s+\frac{n}{r}) \mathbf{A}(s) \begin{pmatrix}
Z_{-s-\frac{n}{r},+} \\
Z_{-s-\frac{n}{r},-}
\end{pmatrix},
\end{equation*}
where
$$\mathbf{A}(s)=\begin{pmatrix}
\displaystyle \cos^r(\frac{\pi }{2}(s+\frac{n}{r})) & 0\\
0&\displaystyle  (\sqrt{-1})^r \sin^r(\frac{\pi }{2}(s+\frac{n}{r}))
\end{pmatrix}.$$

{\bf Case (a')} :  If $d\equiv 2\ (\mathrm{mod}\ 4)$ and  $r$ even, then
\begin{equation*}
\mathcal{F}
\begin{pmatrix}
Z_{s,+} \\
Z_{s,-}
\end{pmatrix}
= 2^r \gamma(s+\frac{n}{r}) \mathbf{A}(s) \begin{pmatrix}
Z_{-s-\frac{n}{r},+} \\
Z_{-s-\frac{n}{r},-}
\end{pmatrix},
\end{equation*}
where
$$\mathbf{A}(s)=\begin{pmatrix}
\displaystyle (\sqrt{-1})^r \sin^r(\frac{\pi }{2}(s+\frac{n}{r}))  & \displaystyle 0\\
0&  \displaystyle \cos^r(\frac{\pi }{2}(s+\frac{n}{r})) 
\end{pmatrix}.$$

{\bf Case (b\,-\,1)} :  If $r=2$ and $d\equiv 1\ (\mathrm{mod}\ 4)$, then
   
   \begin{equation}\label{eq-b-1}
\mathcal{F}
\begin{pmatrix}
Z_{s,+} \\
Z_{s,-}
\end{pmatrix}
= 4\sqrt{2}\gamma(s+\frac{n}{r}) \mathbf{A}(s) \begin{pmatrix}
Z_{-s-\frac{n}{r},+} \\
Z_{-s-\frac{n}{r},-}
\end{pmatrix},
\end{equation}
where
$$\mathbf{A}(s)=
  \left(\begin{array}{lr}
\displaystyle \sin(\frac{\pi}{2}(s+\frac{n+1}{2}))\cos(\frac{\pi}{2}(s+\frac{n}{2}))  & \displaystyle - \sin(\frac{\pi}{2}(s+\frac{n+1}{2}))\sin(\frac{\pi}{2}(s+\frac{n}{2})) \\
\displaystyle \cos(\frac{\pi}{2}(s+\frac{n+1}{2}))\cos(\frac{\pi}{2}(s+\frac{n}{2}))& \displaystyle \cos(\frac{\pi}{2}(s+\frac{n+1}{2}))\sin(\frac{\pi}{2}(s+\frac{n}{2}))
\end{array}\right).
$$

{\bf Case (b\,-\,2)} :  If  $r=2$ and $d\equiv 3\ (\mathrm{mod}\ 4)$, then 

  \begin{equation*}
\mathcal{F}
\begin{pmatrix}
Z_{s,+} \\
Z_{s,-}
\end{pmatrix}
= 4\sqrt{2}\gamma(s+\frac{n}{r}) \mathbf{A}(s) \begin{pmatrix}
Z_{-s-\frac{n}{r},+} \\
Z_{-s-\frac{n}{r},-}
\end{pmatrix},
\end{equation*}
where
$$\mathbf{A}(s)=
\left(\begin{array}{lr}
\displaystyle \cos(\frac{\pi}{2}(s+\frac{n+1}{2}))\cos(\frac{\pi}{2}(s+\frac{n}{2}))&\displaystyle  \cos(\frac{\pi}{2}(s+\frac{n+1}{2}))\sin(\frac{\pi}{2}(s+\frac{n}{2}))\\
\displaystyle \sin(\frac{\pi}{2}(s+\frac{n+1}{2}))\cos(\frac{\pi}{2}(s+\frac{n}{2}))  & \displaystyle - \sin(\frac{\pi}{2}(s+\frac{n+1}{2}))\sin(\frac{\pi}{2}(s+\frac{n}{2}))
\end{array}\right).
$$

{\bf Case (c\,-\,1)} :  If $d=1$ and $r\equiv 3\ (\mathrm{mod}\ 4)$, then
  \begin{equation}\label{44}
\mathcal{F}
\begin{pmatrix}
Z_{s,+} \\
Z_{s,-}
\end{pmatrix}
= -(-2\sqrt{-1})^{\lfloor\frac{r}{2} \rfloor}\gamma(s+\frac{n}{r})\sin^{\lfloor\frac{r}{2} \rfloor}(\pi(s+\frac{n}{r})) \mathbf{B}(s) \begin{pmatrix}
Z^{\mathrm{e}}_{-s-\frac{n}{r}} \\
Z^{\mathrm{o}}_{-s-\frac{n}{r}}
\end{pmatrix},
\end{equation}
where
$$\mathbf{B}(s)=
\left(\begin{array}{cc}
\displaystyle  \sqrt{-1} \sin(\frac{\pi }{2}(s+\frac{n}{r})) &  \displaystyle - \sqrt{-1} \sin(\frac{\pi }{2}(s+\frac{n}{r}))\\
\displaystyle  \cos(\frac{\pi }{2}(s+\frac{n}{r})) & \displaystyle  \cos(\frac{\pi }{2}(s+\frac{n}{r}))
\end{array}\right).
$$


{\bf Case (c\,-\,2)} :  If $d=1$ and $r\equiv 1\ (\mathrm{mod}\ 4) $, then
  
  \begin{equation*}
\mathcal{F}
\begin{pmatrix}
Z_{s,+} \\
Z_{s,-}
\end{pmatrix}
= (-2\sqrt{-1})^{\lfloor\frac{r}{2} \rfloor}\gamma(s+\frac{n}{r})\sin^{\lfloor\frac{r}{2} \rfloor}(\pi(s+\frac{n}{r})) \mathbf{B}(s) \begin{pmatrix}
Z^{\mathrm{e}}_{-s-\frac{n}{r}} \\
Z^{\mathrm{o}}_{-s-\frac{n}{r}}
\end{pmatrix},
\end{equation*}
where
$$\mathbf{B}(s)=
 \left(\begin{array}{cc}
 \displaystyle  \cos(\frac{\pi }{2}(s+\frac{n}{r})) &\displaystyle   \cos(\frac{\pi }{2}(s+\frac{n}{r}))\\
\displaystyle    \sqrt{-1} \sin(\frac{\pi }{2}(s+\frac{n}{r})) &  \displaystyle - \sqrt{-1} \sin(\frac{\pi }{2}(s+\frac{n}{r}))
\end{array}\right).
$$

{\bf Case (c\,-\,3)} :  If $d=1$ and $r\equiv 0\ (\mathrm{mod}\ 4) $, then
 
  \begin{equation*}
\mathcal{F}
\begin{pmatrix}
Z_{s,+} \\
Z_{s,-}
\end{pmatrix}
=  2^{\frac{r-1}{2}} e^{\sqrt{-1}\frac{\pi}{4}} \gamma(s+\frac{n}{r})\cos^{\frac{r}{2} }(\pi(s+\frac{n}{r})) \mathbf{A}(s) \begin{pmatrix}
Z_{-s-\frac{n}{r},+} \\
Z_{-s-\frac{n}{r},-}
\end{pmatrix},
\end{equation*}
where
$$\mathbf{A}(s)=
\begin{pmatrix}
-\sqrt{-1}& 1\\
1&-\sqrt{-1}
\end{pmatrix}.
$$

{\bf Case (c\,-\,4)} :  If $d=1$ and $r\equiv 2\ (\mathrm{mod}\ 4) $, then


   
    \begin{equation*}
\mathcal{F}
\begin{pmatrix}
Z_{s,+} \\
Z_{s,-}
\end{pmatrix}
=  2^{\frac{r-1}{2}} e^{\sqrt{-1}\frac{\pi}{4}} \gamma(s+\frac{n}{r})\cos^{\frac{r}{2} }(\pi(s+\frac{n}{r})) \mathbf{A}(s) \begin{pmatrix}
Z_{-s-\frac{n}{r},+} \\
Z_{-s-\frac{n}{r},-}
\end{pmatrix},
\end{equation*}
where
$$\mathbf{A}(s)=\begin{pmatrix}
1&-\sqrt{-1}\\
 -\sqrt{-1}&1
\end{pmatrix}.$$

   
 \end{theorem}
 
 \begin{proof}
 Recall from above that $x=e\left(-s/2\right)$. In view of  \autoref{Z+-Z-and Zj} and  \autoref{F-Eq-Eucl} we have
 \begin{equation}\label{eq-f-Z+Zj}
 Z_{s-\frac{n}{r},+}(\,\hat{f}\,) =(2\pi)^{-rs}e\left(\frac{rs}{2}\right)\Gamma_\Omega(s)\sum_{j=0}^r\left(\sum_{i=0}^r u_{ij}(x)\right) Z_j(f,-s),
 \end{equation}
 and
 \begin{equation}\label{eq-f-Z-Zj}
 Z_{s-\frac{n}{r},-}(\,\hat{f}\,)= (2\pi)^{-rs}e\left(\frac{rs}{2}\right)\Gamma_\Omega(s)\sum_{j=0}^r\left(\sum_{i=0}^r (-1)^iu_{ij}(x)\right) Z_j(f,-s).
 \end{equation}
We will use \autoref{E-eq-f-FS} and \autoref{coef-mat-FS} to   compute 
 $\sum_{i=0}^r u_{ij}(x)$ and $\sum_{i=0}^r (-1)^iu_{ij}(x)$ for any fixed $j$, $0\leq j\leq r$.
Observe first that $\xi=(\sqrt{-1})^{d(r+1)}$ reduces to
 $$ 
 \xi=\begin{cases}
 1 & \text{ in case (a)}\\
  -1 & \text{ in case (a')}\\
 (\sqrt{-1})^{3d}  & \text{ in case (b)}\\
 (\sqrt{-1})^{r+1} & \text{ in case (c)}.\\
 \end{cases}
  $$
 -- {\bf Case (a)} : Here  
  $d\equiv0\ (\mathrm{mod}\ 4)$ or  $d\equiv2\ (\mathrm{mod}\ 4)$  and  $r$ is odd. In this case 
 $$\begin{array}{l @{\; \; } l}
 \sum_{i=0}^r u_{ij}(x)&=  \displaystyle  (x+1)^{j}(1+x)^{r-j}\\
 & \displaystyle  =(1+x)^r\\
   &\displaystyle  =2^re\left(-\frac{rs}{4}\right)\cos^r\left(\frac{\pi s}{2}\right),
  \end{array}$$
  and
   $$\begin{array}{l @{\; \; } l}
  \sum_{i=0}^r (-1)^iu_{ij}(x)&= \displaystyle   (x-1)^j(1-x)^{r-j}\\
   \displaystyle  &= \displaystyle  (-1)^j(1-x)^r \\
  \displaystyle   & \displaystyle =
  (-1)^j (2\sqrt{-1})^r e\left(-\frac{rs}{4}\right) \sin^r\left(\frac{\pi s}{2}\right).
 \end{array}$$
Hence \autoref{eq-f-Z+Zj} reduces to
  $$
 Z_{s-\frac{n}{r},+}(\,\hat{f}\,)
 = 2^r (2\pi)^{-rs}  e\left(\frac{rs}{4}\right) \cos^r\left(\frac{\pi s}{2}\right) \Gamma_\Omega(s) Z_{-s,+}(f),
 $$
 and  \autoref{eq-f-Z-Zj} to
   $$
 Z_{s-\frac{n}{r},-}(\,\hat{f}\,)
 = (2\sqrt{-1})^r(2\pi)^{-rs} e\left(\frac{rs}{4}\right)\sin^r\left(\frac{\pi s}{2}\right) \Gamma_\Omega(s) Z_{-s,-}(f).
  $$
  -- {\bf Case (a')} :    Here   $d\equiv 2\ (\mathrm{mod}\ 4)$  and $r$ is even. In this case, 
  $$\begin{array}{l @{\; \; } l}
  \sum_{i=0}^r u_{ij}(x)&  \displaystyle = (-1)^{-(r-j)} (-x+1)^j(1-x)^{r-j}\\ 
 &  \displaystyle =(-1)^j (1-x)^r\\
 &   \displaystyle = (-1)^j (2\sqrt{-1})^r e\left(-\frac{rs}{4}\right)\sin^r\left(\frac{\pi s}{2}\right),
  \end{array}$$
  and
  $$\begin{array}{l @{\; \; } l}
 \sum_{i=0}^r (-1)^iu_{ij}(x)&= (-1)^{-(r-j)} (-x-1)^j(1+x)^{r-j} \\
 &=(1+x)^r\\
 &\displaystyle=  2^re\left(-\frac{rs}{4}\right)\cos^r\left(\frac{\pi s}{2}\right).
  \end{array}$$
  Then, by  \autoref{eq-f-Z+Zj}  and \autoref{eq-f-Z-Zj}, one has     
$$ 
 Z_{s-\frac{n}{r},+}(\,\hat{f}\,)  =(2\sqrt{-1})^r (2\pi)^{-rs}e\left(\frac{rs}{4}\right) \sin^r\left(\frac{\pi s}{2}\right) \Gamma_\Omega(s)Z_{-s,-}(f),
 $$
and  
$$ 
 Z_{s-\frac{n}{r},-}(\,\hat{f}\, )=2^r(2\pi)^{-rs}e\left(\frac{rs}{4}\right)\Gamma_\Omega(s)  \cos^r\left(\frac{\pi s}{2}\right)Z_{-s,+}(f).
 $$
-- {\bf Case (b)} :   In this case $r=2$ and $\xi=(\sqrt{-1})^{3d}$ with $d$ odd. 
By  \cite[page 481]{SF} we have
 $$ \sum_{i=0}^2 y^iu_{ij}(x)= \begin{cases} 
     1+x^2y^2 & \text{for } j=0\\
 x+\xi(1-x^2)y + xy^2, & \text{for } j=1 \\
  x^2+y^2, & \text{for } j=2
   \end{cases}$$
 for any  real number $y$. 
Then, for $\varepsilon=\pm$, \autoref{eq-f-Z+Zj} and \autoref{eq-f-Z-Zj} become
 $$
 Z_{s-\frac{n}{2},\varepsilon}(\,\hat{f}\,)= \gamma(s) \left[(1+x^2)\{Z_0(f,-s)+Z_2(f,-s)\}
  +(2x+\varepsilon\xi(1-x^2))Z_1(f,-s) \right].
  $$
  Since 
  $Z_{-s,+}(f)=Z_0(f,-s)+Z_1(f,-s)+Z_2(f,-s)$ and $Z_{-s,-}(f)=Z_0(f,-s)-Z_1(f,-s)+Z_2(f,-s)$,
   the above equality for $\varepsilon=+$ becomes  
    $$ 
 Z_{s-\frac{n}{2},+}(\,\hat{f}\,)= \frac{\gamma(s)}{2} \left[\left((1+x)^2 +\xi(1-x^2)\right) Z_{-s,+}(f)
  +\left((1-x)^2 -\xi(1-x^2)\right) Z_{-s,-}(f) \right],
  $$
and for $\varepsilon=-$ it becomes
  $$
 Z_{s-\frac{n}{2},-}(\,\hat{f}\,)= \frac{\gamma(s)}{2} \left[ \left((1+x)^2 -\xi(1-x^2)\right) Z_{-s,+}(f)  +\left((1-x)^2 +\xi(1-x^2)\right) Z_{-s,-}(f) \right].
  $$
  Since $d$ is odd, therefore $d\equiv 1\ (\mathrm{mod}\ 4)$ or $d\equiv 3\ (\mathrm{mod}\ 4)$. 
  
  {\bf (b\,-\,1)} : If $d\equiv 1\ (\mathrm{mod}\ 4)$, then $\xi=-\sqrt{-1}$, and we obtain 
      \begin{eqnarray*} 
  Z_{s-\frac{n}{2},+}(\,\hat{f}\,)&=&c_1(s)\left[\cos(\frac{\pi s}{2}) Z_+(f,-s)- \sin(\frac{\pi s}{2}) Z_{-s,-}(f)\right],\\
Z_{s-\frac{n}{2},-}(\,\hat{f}\,)&=&c_2(s)\left[\cos(\frac{\pi s}{2}) Z_+(f,-s)+ \sin(\frac{\pi s}{2}) Z_{-s,-}(f)\right],
      \end{eqnarray*}
 
   where $c_1(s)=2\sqrt{2}\gamma(s)\sin(\frac{\pi s}{2}+\frac{\pi}{4})$ and $c_2(s)=2\sqrt{2}\gamma(s)\cos(\frac{\pi s}{2}+\frac{\pi}{4})$.
   
  {\bf (b\,-\,2)} :  If  $d\equiv 3\ (\mathrm{mod}\ 4)$, then $\xi=\sqrt{-1}$ and
     \begin{eqnarray*} 
  Z_{s-\frac{n}{2},+}(\,\hat{f}\,)&=&c_2(s)\left[\cos(\frac{\pi s}{2}) Z_+(f,-s)+ \sin(\frac{\pi s}{2}) Z_{-s,-}(f)\right],\\
Z_{s-\frac{n}{2},-}(\,\hat{f}\,)&=&c_1(s)\left[\cos(\frac{\pi s}{2}) Z_+(f,-s)- \sin(\frac{\pi s}{2}) Z_{-s,-}(f)\right],
      \end{eqnarray*}
 
  with the same  $c_1(s)$ and $c_2(s)$ as in (b\,-\,1).
  
 -- {\bf Case (c)} : In this case $d=1$, $r$ is arbitrary  and $\xi=(\sqrt{-1})^{r+1}$. Then we shall consider four cases.
 
  {\bf (c\,-\,1)}   :  If $r\equiv 3\ (\text{mod } 4)$, then $r$ is odd, say $r=2\rho+1$, and $\xi=1$. Thus
$$\begin{array}{rl}
 \sum_{i=0}^r u_{ij}(x)&=  P_j(x,1)P_{r-j}(1,x)\\
 &= (x+1)^{\lfloor \frac{j}{2}\rfloor}(1-x)^{j-\lfloor \frac{j}{2}\rfloor}(x+1)^{\lfloor \frac{r-j}{2}\rfloor}(x-1)^{r-j-\lfloor \frac{r-j}{2}\rfloor},
  \end{array}$$
  and
  $$\begin{array}{r @{\,\,} l}
  \sum_{i=0}^r (-1)^iu_{ij}(x)&  =  P_j(x,-1)P_{r-j}(1,-x)\\
   &=    (x-1)^{\lfloor \frac{j}{2}\rfloor}(-1-x)^{j-\lfloor \frac{j}{2}\rfloor}(1-x)^{\lfloor \frac{r-j}{2}\rfloor}(-x-1)^{r-j-\lfloor \frac{r-j}{2}\rfloor}.
  \end{array}$$
  If  $j$ is even, say $j=2k$, then 
  $$\sum_{i=0}^r u_{ij}(x)=(-1)^k (x^2-1)^\rho(x-1) \quad\text{and}\quad \sum_{i=0}^r (-1)^i u_{ij}(x)=(-1)^{k+1} (x^2-1)^\rho(x+1).$$
If $j$ is odd, say $j=2k+1$, then 
  $$\sum_{i=0}^r u_{ij}(x)=(-1)^{k+1} (x^2-1)^\rho(x-1) \quad\text{and}\quad  \sum_{i=0}^r (-1)^iu_{ij}(x)=(-1)^{k+1} (x^2-1)^\rho(x+1).$$
   Therefore, by \autoref{eq-f-Z+Zj} and by \autoref{eq-f-Z-Zj}, we have
  \begin{eqnarray*}
    Z_{s-\frac{n}{r},+}(\,\hat{f}\,)&=& c_3(s) \Big[\sum_{k=0}^\rho (-1)^kZ_{2k}(f,-s)- \sum_{k=0}^\rho (-1)^{k}Z_{2k+1}(f,-s)\Big]\\
  & =& c_3(s)\left[Z_{-s}^{\text{e}}(f)-Z_{-s}^{\text{o}}(f)\right],
  \end{eqnarray*}
  where $c_3(s)=(-2\sqrt{-1})^{\rho+1}\gamma(s)\sin^{\rho}(\pi s) \sin(\frac{\pi s}{2}),$  and
  \begin{eqnarray*}
  Z_{s-\frac{n}{r},-}(\hat{f})&=& -c_4(s)\Big[\sum_{k=0}^\rho (-1)^{k}Z_{2k}(f,-s) +\sum_{k=0}^\rho (-1)^{k}Z_{2k+1}(f,-s)\Big]\\
   &=& -c_4(s)\left[Z_{-s}^{\text{e}}(f)+Z_{-s}^{\text{o}}(f)\right],
  \end{eqnarray*}
    where $c_4(s)=2(-2\sqrt{-1})^{\rho}\gamma(s)\sin^{\rho}(\pi s) \cos(\frac{\pi s}{2})$.

  
  {\bf (c\,-\,2)}    :  If $r\equiv 1\ (\text{mod } 4)$, then $r$ is odd, say $r=2\rho+1$, and $\xi=-1$. Then
 $$ 
 \sum_{i=0}^r u_{ij}(x)=  (-1)^{-(r-j)}P_j(-x,1)P_{r-j}(1,-x),$$ and
  $$\sum_{i=0}^r (-1)^iu_{ij}(x)=  (-1)^{-(r-j)}P_j(-x,-1)P_{r-j}(1,x).$$
  If $j$ is   even,  say $j=2k$, we have
  $$\sum_{i=0}^r u_{ij}(x)=(-1)^k (x^2-1)^\rho(x+1) \quad \text{and} \quad \sum_{i=0}^r (-1)^iu_{ij}(x)=(-1)^{k+1} (x^2-1)^\rho(x-1).$$ 
  If $j$ is odd, say $j=2k+1$, then
   $$\sum_{i=0}^r u_{ij}(x)=(-1)^{k} (x^2-1)^\rho(x+1) \quad \text{and} \quad \sum_{i=0}^r (-1)^iu_{ij}(x)=(-1)^{k} (x^2-1)^\rho(x-1). $$
 Therefore, by \autoref{eq-f-Z+Zj}, \autoref{eq-f-Z-Zj} and \autoref{Z-odd-even}, we have 
  \begin{eqnarray*}
  Z_{s-\frac{n}{r},+}(\,\hat{f}\,)&=& c_4(s) \Big[\sum_{k=0}^\rho (-1)^kZ_{2k}(f,-s)+ \sum_{k=0}^\rho (-1)^{k}Z_{2k+1}(f,-s)\Big]\\
  &=& c_4(s)\left[Z_{-s}^{\text{e}}(f)+Z_{-s}^{\text{o}}(f)\right],
  \end{eqnarray*}
  and
  \begin{eqnarray*}
  Z_{s-\frac{n}{r},-}(\,\hat{f}\,)&  = & c_3(s) \Big[-\sum_{k=0}^\rho (-1)^{k}Z_{2k}(f,-s)+ \sum_{k=0}^\rho (-1)^{k}Z_{2k+1}(f,-s)\Big]\\
    &   =& c_3(s)\left[-Z_{-s}^{\text{e}}(f)+Z_{-s}^{\text{o}}(f)\right],
  \end{eqnarray*}
      where the  $c_3(s)$ and $c_4(s)$ are the same as in (c\,-\,1).
 
  {\bf (c\,-\,3)}     :  If $r\equiv 0\ (\text{mod } 4)$, then $r$ is even, say $r=2\rho$, and $\xi=\sqrt{-1}$. Thus
$$
 \sum_{i=0}^r u_{ij}(x)=  (\sqrt{-1})^{-(r-j)}P_j(\sqrt{-1}x,1)P_{r-j}(1,\sqrt{-1}x),
 $$
 and
 $$\sum_{i=0}^r (-1)^iu_{ij}(x)=  (\sqrt{-1})^{-(r-j)}P_j(\sqrt{-1}x,-1)P_{r-j}(1,-\sqrt{-1}x).$$
 If $j$ is even,  then 
 $$\sum_{i=0}^r u_{ij}(x)=(1+x^2)^\rho \quad \text{and} \quad  \sum_{i=0}^r (-1)^iu_{ij}(x)=(1+x^2)^\rho.$$
 If $j$ is odd,  then 
 $$\sum_{i=0}^r u_{ij}(x)=-\sqrt{-1} (1+x^2)^\rho \quad \text{and} \quad  \sum_{i=0}^r (-1)^iu_{ij}(x)=\sqrt{-1} (1+x^2)^\rho. $$
 Therefore, by \autoref{eq-f-Z+Zj}, \autoref{eq-f-Z-Zj} and  \autoref{Z+-Z-pair-impair}, we have
  \begin{eqnarray*}
 Z_{s-\frac{n}{r},+}(\,\hat{f}\,)&=& c_5(s) \Big[\sum_{k=0}^\rho Z_{2k}(f,-s)-\sqrt{-1} \sum_{k=0}^\rho Z_{2k+1}(f,-s)\Big]\\
 &=&\frac{1}{\sqrt{2}}c_5(s)e^{\frac{\sqrt{-1}\pi}{4}}\left[-\sqrt{-1} Z_{-s,+}(f)+ Z_{-s,-}(f)\right],
  \end{eqnarray*}
 
 and
  \begin{eqnarray*}
   Z_{s-\frac{n}{r},-}(\,\hat{f}\,)&=& c_5(s) \Big[\sum_{k=0}^\rho Z_{2k}(f,-s)+\sqrt{-1} \sum_{k=0}^\rho Z_{2k+1}(f,-s)\Big]\\
   &=&\frac{1}{\sqrt{2}}c_5(s)e^{\frac{\sqrt{-1}\pi}{4}}\left[ Z_{-s,+}(f)-\sqrt{-1} Z_{-s,-}(f)\right],
  \end{eqnarray*}
    where  $c_5(s)=2^{\rho}\gamma(s)\cos^{\rho}({\pi s})$.     
    
 
 
  {\bf (c\,-\,4)}     :  If $r\equiv 2\ (\text{mod } 4)$, then $r$ is even, say $r=2\rho$, and $\xi=-\sqrt{-1}$. Thus
    $$ 
 \sum_{i=0}^r u_{ij}(x)=  (-\sqrt{-1})^{-(r-j)}P_j(-\sqrt{-1}x,1)P_{r-j}(1,-\sqrt{-1}x),$$
 and
 $$ \sum_{i=0}^r (-1)^iu_{ij}(x)=  (-\sqrt{-1})^{-(r-j)}P_j(-\sqrt{-1}x,-1)P_{r-j}(1,\sqrt{-1}x).$$
  If $j$ is even, 
  $$\sum_{i=0}^r u_{ij}(x)=(1+x^2)^\rho \quad\text{and}\quad \sum_{i=0}^r (-1)^i u_{ij}(x)=(1+x^2)^\rho.$$ 
If $j$ is odd, 
$$\sum_{i=0}^r u_{ij}(x)=\sqrt{-1} (1+x^2)^\rho \quad\text{and}\quad  \sum_{i=0}^r (-1)^i  u_{ij}(x)=-\sqrt{-1} (1+x^2)^\rho.$$
 Therefore, similarly to the case (c\,-\,3), we can prove that
  \begin{eqnarray*}
 Z_{s-\frac{n}{r},+}(\,\hat{f}\,)&=&c_5(s) \Big[\sum_{k=0}^\rho Z_{2k}(f,-s)+\sqrt{-1} \sum_{k=0}^\rho Z_{2k+1}(f,-s)\Big]\\
 &=&\frac{1}{\sqrt{2}}c_5(s)e^{\frac{\sqrt{-1}\pi}{4}} \left[ Z_{-s,+}(f)-\sqrt{-1}Z_{-s,-}(f)\right],
   \end{eqnarray*}
   and
     \begin{eqnarray*}
   Z_{s-\frac{n}{r},-}(\,\hat{f}\,)&=& c_5(s) \Big[\sum_{k=0}^\rho Z_{2k}(f,-s)-\sqrt{-1} \sum_{k=0}^\rho Z_{2k+1}(f,-s)\Big]\\
    &=&\frac{1}{\sqrt{2}}c_5(s)e^{\frac{\sqrt{-1}\pi}{4}} \left[-\sqrt{-1}Z_{-s,+}(f)+Z_{-s,-}(f)\right],
       \end{eqnarray*}
    where $c_5(s)$ is the same as in (c\,-\,3).
      \end{proof}
    



\section{Construction of the family $E_{s,t}$}\label{sec-Est}
Recall from  \autoref{Dstgen} the definition of the differential operator $D_{s,t}$ with $(s,t)\in\mathbb{C}\times\mathbb{C}$.
 Let $E_{s,t}$ be the differential operator with polynomial coefficients defined on the Schwartz space $\mathcal{S}(V\times V)$ by  
  \begin{equation}\label{def-E(s,t)}
 \mathcal{F} \circ E_{s,t}= D_{s,t} \circ \mathcal{F}.
  \end{equation}
  
  For $(s,\varepsilon)\in\mathbb{C}\times\{\pm\}$   let
  \begin{equation}\label{def-J-st}
  J_{s,\varepsilon}f(x):=\int_V f(y) \det(x-y)^{s,\varepsilon}  dy,\quad f\in\mathcal{S}(V).
  \end{equation}
  This integral, initially well defined for $\Re(s)\gg0$, can be extended meromorphically to $\mathbb{C}$. 
We may think of  $J_{s,\varepsilon}$ as a convolution operator  :
  $$J_{s,\varepsilon}f=Z_{s,\varepsilon}\ast f,$$
  where $Z_{s,\varepsilon}$ is the tempered distribution defining the local Zeta integral \autoref{defzeta}.
 
 Let $M$  be the multiplication operator defined on $\mathcal{S}(V\times V)$ by
  \begin{equation}\label{def-M}
  Mf(x,y)=\det(x-y)f(x,y).
  \end{equation}

  \begin{theorem}\label{them-main}
  For generic $(s,\varepsilon)$ and $(t,\eta)$ in $\mathbb{C}\times \{\pm\}$,   we have
\begin{equation*}
M\circ   [J_{s,\varepsilon}\otimes   J_{t,\eta}]=\kappa(s,t)  [J_{s+1,-\varepsilon}\otimes J_{t+1,-\eta}]\circ E_{-s-\frac{n}{r},-t-\frac{n}{r}}, 
\end{equation*}
where $\kappa(s,t)$ is a meromorphic function on $\mathbb C\times \mathbb C$.
  \end{theorem}
 
  \addsubsection{Proof of \autoref{them-main} in the non-euclidean case except $  \mathbb R^{p,q}$ (type II, III and IV)}
 Recall from \autoref{const-c(e,s)} that 
 $$\mathcal F (Z_{s,\varepsilon})= c(s,\varepsilon) Z_{-s-{n\over r},\varepsilon}$$
 where $c(s,\varepsilon)$ is as in \autoref{valeur-c(e,s)}. 
 Now,  by  the main identity in \autoref{Dstgen} for the non-euclidean case we have 
\begin{eqnarray*}
&&\mathcal F\Big(M\circ (J_{s,\varepsilon}\otimes J_{t,\eta})f\Big)(x,y)\\
&=&
\frac{1}{(2\pi\sqrt{-1})^r} \det\left(\frac{\partial}{\partial x} -\frac{\partial}{\partial y} \right)  \mathcal F (Z_{s,\varepsilon}) (x) \mathcal F (Z_{t,\eta}) (y)\mathcal F f (x,y)\\
&=&
\frac{c(s,\varepsilon)c(t,\eta)}{(2\pi\sqrt{-1})^r}   \det\left(\frac{\partial}{\partial x} -\frac{\partial}{\partial y} \right) Z_{-s-{n\over r},\varepsilon}(x) Z_{-t-{n\over r},\eta} (y)
\mathcal F f (x,y)\\
&=&
\frac{c(s,\varepsilon)c(t,\eta)}{(2\pi\sqrt{-1})^r}    Z_{-s-1-{n\over r},-\varepsilon}(x) Z_{-t-1-{n\over r},-\eta} (y)
D_{-s-\frac{n}{r},-t-\frac{n}{r}}\mathcal{F}f (x,y)\\
&=&
\frac{c(s,\varepsilon)c(t,\eta)}{(2\pi\sqrt{-1})^r}    Z_{-s-1-{n\over r},-\varepsilon}(x) Z_{-t-1-{n\over r},-\eta} (y)
 \mathcal{F}(E_{-s-\frac{n}{r},-t-\frac{n}{r}}f)(x,y)\\
 &=& \frac{c(s,\varepsilon)c(t,\eta)}{(2\pi\sqrt{-1})^r c(s+1,-\varepsilon)c(t+1,-\eta)} \mathcal F(Z_{s+1,-\varepsilon})(x) \mathcal F(Z_{t+1,-\eta})(y) 
\mathcal F(E_{-s-{n\over r}, -t-{n\over r}} f )(x,y)\\
&=& \frac{c(s,\varepsilon)c(t,\eta)}{(2\pi\sqrt{-1})^r c(s+1,-\varepsilon)c(t+1,-\eta)} 
\mathcal F\Big( (J_{s+1,-\varepsilon}\otimes J_{t+1,-\eta})\circ E_{-s-{n\over r}, -t-{n\over r}} f\Big)(x,y).
\end{eqnarray*}
By the injectivity of the Fourier transform we get 
$$M\circ (J_{s,\varepsilon}\otimes J_{t,\eta})
= 
\kappa(s,t)
(J_{s+1,-\varepsilon}\otimes J_{t+1,-\eta})\circ E_{-s-{n\over r}, -t-{n\over r}},$$
where
$$\begin{array}{rl}
\kappa(s,t)&=\displaystyle \frac{c(s,\varepsilon)c(t,\eta)}{(2\pi\sqrt{-1})^r c(s+1,-\varepsilon)c(t+1,-\eta) }\\
&=
\begin{cases}
 \displaystyle  \frac{(2\pi\sqrt{-1})^r}{b_{r,d}(s+1) b_{r,d}(s+1)  } & \text{ split case}\\
 \displaystyle  \frac{(-8{\pi}\sqrt{-1})^r}{b_{2r,d}(-2s-\frac{2n}{r})b_{2r,d}(2s+2) b_{2r,d}(-2t-\frac{2n}{r})b_{2r,d}(2t+2)}& \text{ non-split case.}
\end{cases}
\end{array}
$$
 

\addsubsection{Proof of \autoref{them-main} in the euclidean case (type I)}
 The proof of  \autoref{them-main} in the cases (a) and  (a') is similar to the one in the non-euclidean case. We will give a proof only in the cases (b\,-\,1) and  (c\,-\,1). We leave the remaining cases to the reader. 
  
  - {\bf Case   (b\,-\,1)} :  Here $r=2$ and $d\equiv 1\,(\mathrm{mod}\ 4)$.
  Recall from \autoref{eq-b-1}   that
  $$\mathcal{F}\begin{pmatrix} Z_{s,+}\\ Z_{s,-}\end{pmatrix}
  =4\sqrt{2}\gamma(s+\frac{n}{2})\mathbf{A}(s)\begin{pmatrix} Z_{-s-\frac{n}{2},+}\\ Z_{-s-\frac{n}{2},-}\end{pmatrix},
  $$
  where $$\mathbf{A}(s)=\begin{pmatrix} a_{+,+}(s) & a_{+,-}(s)\\ a_{-,+}(s) & a_{-,-}(s)\\ \end{pmatrix},$$
  and 
 \begin{eqnarray*}
  \displaystyle a_{+,+}(s)&=&\;\;\;\,\sin\frac{\pi}{2}(s+\frac{n+1}{2})\, \cos\frac{\pi}{2}(s+\frac{n}{2}),\\
     \displaystyle  a_{+,-}(s)&=&-\sin\frac{\pi}{2}(s+\frac{n+1}{2})\, \sin\frac{\pi}{2}(s+\frac{n}{2}),\\
     \displaystyle   a_{-,+}(s)&=&\;\;\;\,\cos\frac{\pi}{2}(s+\frac{n+1}{2})\, \cos\frac{\pi}{2}(s+\frac{n}{2}),\\
      \displaystyle    a_{-,-}(s)&=&\;\;\;\,\cos\frac{\pi}{2}(s+\frac{n+1}{2})\, \sin\frac{\pi}{2}(s+\frac{n}{2}).
      \end{eqnarray*}
Observe that for any $\varepsilon,\eta =\pm$, we have 
   \begin{equation}\label{coefA(s+11)}
 a_{\varepsilon,\eta}(s)=  -a_{-\varepsilon,-\eta}(s+1).
   \end{equation}
  Let $f\in \mathcal{S}(V\times V)$. 
     Since $J_{s,\varepsilon}$ is a convolution operator with $Z_{s,\varepsilon}$, we have 
\begin{eqnarray*}
 &&\mathcal{F}(M\circ (J_{s,\varepsilon}\otimes J_{t,\eta})f)(x,y)\\
 &=&\frac{1}{(2\pi\sqrt{-1})^r}\Delta\left(\frac{\partial}{\partial x}-\frac{\partial}{\partial y}\right)\left\{\mathcal{F}(Z_{s,\varepsilon})(x)\mathcal{F}(Z_{t,\eta})(y)\mathcal{F}f(x,y)\right\}\\
       &=&\frac{c (s,t)}{(2\pi\sqrt{-1})^r} \Delta\left(\frac{\partial}{\partial x}-\frac{\partial}{\partial y}\right) \left[\left(a_{\varepsilon,+}(s) Z_{-s-\frac{n}{2},+}(x)+a_{\varepsilon,-}(s) Z_{-s-\frac{n}{2},-}(x)\right) \right.\\
       &&\left.\left(a_{\eta,+}(t) Z_{-t-\frac{n}{2},+}(y)+a_{\eta,-}(t) Z_{-t-\frac{n}{2},-}(y)\right)  \mathcal{F}f(x,y)\right],
\end{eqnarray*}   
 where $c(s,t)=32\gamma(s+\frac{n}{2}) \gamma(t+\frac{n}{2})$. In view of  \autoref{coefA(s+11)} and  \autoref{main-id-split} we get
\begin{eqnarray*}
&&\mathcal{F}(M\circ (J_{s,\varepsilon}\otimes J_{t,\varepsilon})f)(x,y)\\
&=&\frac{c(s,t)}{(2\pi\sqrt{-1})^r} \left[\left(a_{-\varepsilon,+}(s+1) Z_{-s-1-\frac{n}{2},+}(x)+a_{-\varepsilon,-}(s+1)
Z_{-s-1-\frac{n}{2},-}(x)\right)\right. \\
&&\left.\left(a_{-\eta,+}(t+1) Z_{-t-1-\frac{n}{2},+}(y)+a_{-\eta,-}(t+1) Z_{-t-1-\frac{n}{2},-}(y)\right)\right] D_{-s-\frac{n}{2},-t-\frac{n}{2}}(\mathcal{F} f)(x,y),\\
        &=&\frac{c (s,t)}{(2\pi\sqrt{-1})^r c (s+1,t+1)} \mathcal{F}(Z_{s+1,-\varepsilon})(x)\mathcal{F}(Z_{t+1,-\eta})(y)  \mathcal{F} (E_{-s-\frac{n}{2},-t-\frac{n}{2}}f)(x,y)\\
       &=&\frac{c (s,t)}{(2\pi\sqrt{-1})^r c (s+1,t+1)}  \mathcal{F}((J_{s+1,-\varepsilon}\otimes J_{t+1,-\eta})  \circ  (E_{-s-\frac{n}{2},-t-\frac{n}{2}}f))(x,y)  .
\end{eqnarray*}      
 Now, using the inverse Fourier transform, we obtain 
        $$M\circ   [J_{s,\varepsilon}\otimes   J_{t,\eta}]=\kappa(s,t)  [J_{s+1,-\varepsilon}\otimes J_{t+1,-\eta}]\circ E_{-s-\frac{n}{r},-t-\frac{n}{r}},$$
        where 
                $$\kappa(s,t) = \frac{(2\pi\sqrt{-1})^r}{b_{r,d}(s+1)b_{r,d}(t+1)}.$$

     -- {\bf Case (c\,-\,2)}  : Here $d=1$ and $r\equiv 1\,(\mathrm{mod}\ 4)$. 
     Recall from  \autoref{44}  that 
 $$\mathcal{F}\begin{pmatrix} Z_{s,+}\\ Z_{s,-}\end{pmatrix}
  =(-2\sqrt{-1})^{\lfloor\frac{r}{2} \rfloor}\gamma(s+\frac{n}{r})\sin^{\lfloor\frac{r}{2} \rfloor}(\pi(s+\frac{n}{r})) \mathbf{B}(s)\begin{pmatrix}  Z^{\text{e}}_{-s-\frac{n}{2}}\\ Z^{\text{o}}_{-s-\frac{n}{2}}\end{pmatrix},
  $$
  where $$\mathbf{B}(s)=\begin{pmatrix} a_{+}^{\mathrm{e}}(s) & a_{+}^{\mathrm{o}}(s)\\ a_{-}^{\mathrm{e}}(s) & a_{-}^{\mathrm{o}}(s)\\ \end{pmatrix},$$
  and 
  $$\begin{array}{r @{\;=\;} l}
  a_{+}^\text{e}(s)&  \cos(\frac{\pi}{2}(s+\frac{n}{2})),\\         
    a_{+}^\text{o}(s)&   \cos(\frac{\pi}{2}(s+\frac{n}{2})).\\     
    a_{-}^\text{e}(s)&    \sqrt{-1} \sin(\frac{\pi}{2}(s+\frac{n}{2})),\\
      a_{-}^\text{o}(s)&     -\sqrt{-1} \sin(\frac{\pi}{2}(s+\frac{n}{2})),\\
         \end{array}$$
Observe that for any $\varepsilon=\pm$, we have
   \begin{equation}\label{55}
   a_{\varepsilon}^\text{e}(s) =-\sqrt{-1}a_{-\varepsilon}^\text{e}(s+1),\quad
   a_{\varepsilon}^\text{o}(s) =\sqrt{-1}a_{-\varepsilon}^\text{o}(s+1).
   \end{equation}
Let $f\in \mathcal{S}(V\times V)$.  We have,
 \begin{eqnarray*}
&& \mathcal{F}[(M\circ(J_{s,\varepsilon}\otimes J_{t,\eta}))f](x,y) \\
  &=&
  \frac{1}{(2\pi\sqrt{-1})^r} \det\left(\frac{\partial}{\partial x} -\frac{\partial}{\partial y}\right)
  \mathcal{F}(Z_{s,\varepsilon})(x)\mathcal{F}(Z_{t,\eta})(y)\mathcal{F}(f)(x,y),\\
  &=&    \frac{c(s,t) }{(2\pi\sqrt{-1})^r} \det\left(\frac{\partial}{\partial x} -\frac{\partial}{\partial y}\right) (a_{\varepsilon}^\text{e}(s)Z^{\text{e}}_{-s-\frac{n}{2}}(x)+a_{\varepsilon}^\text{o}(s)Z^{\text{o}}_{-s-\frac{n}{2}}(x))\times\\
  &&\qquad \times (a_{\eta}^\text{e}(t)Z^{\text{e}}_{-t-\frac{n}{2}}(y)+a_{\eta}^\text{o}(t)Z^{\text{o}}_{-t-\frac{n}{2}}(y))\mathcal{F}(f)(x,y),
     \end{eqnarray*}
     where $c(s,t)=(-2\sqrt{-1})^{r-1}\gamma(s+\frac{n}{r})\gamma(t+\frac{n}{r})\sin^{\lfloor\frac{r}{2} \rfloor}(\pi(s+\frac{n}{r})) \sin^{\lfloor\frac{r}{2} \rfloor}(\pi(t+\frac{n}{r}))$. 
     Using the fact that
     $$\begin{array}{r @{\; =\; } l}
    \det\left(\frac{\partial}{\partial x} -\frac{\partial}{\partial y}\right) [(Z^\text{e}_{-s-\frac{n}{r}}\otimes Z^\text{e}_{-t-\frac{n}{r}})\mathcal{F}(f)] & [(Z^\text{e}_{-s-1-\frac{n}{r}}\otimes Z^\text{e}_{-t-1-\frac{n}{r}})D_{-s-\frac{n}{r},-t-\frac{n}{r}}\mathcal{F}(f)],\\
      \det\left(\frac{\partial}{\partial x} -\frac{\partial}{\partial y}\right) [(Z^\text{o}_{-s-\frac{n}{r}}\otimes Z^\text{o}_{-t-\frac{n}{r}})\mathcal{F}(f)] & [(Z^\text{o}_{-s-1-\frac{n}{r}}\otimes Z^\text{o}_{-t-1-\frac{n}{r}})D_{-s-\frac{n}{r},-t-\frac{n}{r}}\mathcal{F}(f)],\\
       \det\left(\frac{\partial}{\partial x} -\frac{\partial}{\partial y}\right) [(Z^\text{e}_{-s-\frac{n}{r}}\otimes Z^\text{o}_{-t-\frac{n}{r}})\mathcal{F}(f)] &- [(Z^\text{e}_{-s-1-\frac{n}{r}}\otimes Z^\text{o}_{-t-1-\frac{n}{r}})D_{-s-\frac{n}{r},-t-\frac{n}{r}}\mathcal{F}(f)],\\
       \det\left(\frac{\partial}{\partial x} -\frac{\partial}{\partial y}\right) [(Z^\text{o}_{-s-\frac{n}{r}}\otimes Z^\text{e}_{-t-\frac{n}{r}})\mathcal{F}(f)] &- [(Z^\text{o}_{-s-1-\frac{n}{r}}\otimes Z^\text{e}_{-t-1-\frac{n}{r}})D_{-s-\frac{n}{r},-t-\frac{n}{r}}\mathcal{F}(f)],
         \end{array}$$
    and the identities \autoref{55} we obtain 
   \begin{eqnarray*}
    &&\mathcal{F}[(M\circ(J_{s,\varepsilon}\otimes J_{t,\eta}))f](x,y)\\
    &=&- \frac{c(s,t) }{(2\pi\sqrt{-1})^r c(s+1,t+1)}  \mathcal{F}(Z_{s+1,-\varepsilon}) \mathcal{F}(Z_{t+1,-\eta}) D_{-s-\frac{n}{r},-t-\frac{n}{r}}\mathcal{F}(f)(x,y)\\
   &=& - \frac{c(s,t) }{(2\pi\sqrt{-1})^r c(s+1,t+1)}  \mathcal{F}[(J_{s+1,-\varepsilon}\otimes J_{t+1,-\varepsilon})\circ E_{-s-\frac{n}{r},-t-\frac{n}{r}}(f)](x,y).
     \end{eqnarray*}
    By the injectivity of the Fourier transform we obtain the desired result.

 \section{The degenerate principal series and the Knapp-Stein operators}\label{sec-DPS-KS}
 
  \addsubsection{The conformal group of a  real   Jordan algebra}\label{subsec-Co}
 
 Let $V$ be a simple real   Jordan algebra.  Recall that $\Str(V)$ denotes the  structure group of $V$ (see \autoref{sub-str}).  If $\ell\in \Str(V)$, then for any $x\in V$
\begin{equation}\label{def-chi}
\det(\ell(x)) = \chi(\ell) \det (x),
\end{equation}
for some  $\chi(\ell)\in \mathbb R^*$, and  the function $\ell \mapsto \chi(\ell)$ is a character of $\Str(V)$. It is known that for  $x\in V^\times$, the quadratic operator $P(x)$ (see \autoref{quad-rep})    belongs to $\Str(V)$, and for any $x,y\in V$
\begin{equation*}
\det\big(P(x)y) =\det(x)^2\det(y),
\end{equation*}
 which implies
  \begin{equation}\label{chiP}
  \chi\big(P(x)\big) = \det(x)^2 .
 \end{equation}
 
  The \emph{inversion} $\imath$ is a rational transformation of $V$ defined on $V^\times$ by
\begin{equation*}
 \imath(x)= -x^{-1} .
\end{equation*}
Its differential at $x\in V^\times$ is given by
\begin{equation}\label{deriota}
D\imath(x) = P(x^{-1})=P(x)^{-1} .
\end{equation}
For $a\in V$, denote by $n_a$ the translation $x\mapsto x+a$. Let 
\[N:=\{ n_a,\ a\in V\}\]
 be the (abelian Lie) group of all translations.
The  \emph{conformal group} $\Co(V)$ of $V$ is by definition  the group of rational transforms of $V$ generated by $\Str(V)$, $N$ and the inversion $\imath$. It can be shown that $\Co(V)$ is a  simple Lie group (see \cite{Bertram, koe}). 
We will denote by $\Str(V)^+$ the subgroup of $\Str(V)$ defined by
\[\Str(V)^+= \{ \ell\in \Str(V),\quad \chi(\ell)>0\}.
\]
 Define the \emph{proper conformal group} $\Co(V)^+$ to be the group generated by $\Str(V)^+,$ $N$ and the inversion $\imath$.
 
 Let $\Aff(V)=\Str(V)\ltimes N$ (resp. $\Aff(V)^+=\Str(V)^+\ltimes N)$ be the group generated by $\Str(V)$ (resp. $\Str(V)^+)$ and $N$. Then $\Aff(V)$ is a parabolic subgroup of $\Co(V)$ equal to the normalizer of $N$ in $\Co(V)$. As $\Aff(V)^+ = \Co(V)^+\cap \Aff(V)$, $\Aff(V)^+$ is the normalizer of $N$ in $\Co(V)^+,$ and hence a parabolic subgroup of $\Co(V)^+.$ 
 
The center of $\Str(V)$ is the group of dilations $\{\delta_t: v\mapsto tv, \;t\in \mathbb R^*\}$. Let 
 $$A:=\{\delta_t,\quad t\in \mathbb {R}_{>0}\} .
 $$
 As $\chi(\delta_t) = t^r,$  $A$ is contained in $\Str(V)^+$. 
 
For $g\in \Co(V)^+$,  let 
\begin{equation}\label{Theta}
\Theta (g)=  \alpha\circ \imath\circ g\circ \imath\circ \alpha, 
\end{equation}
where $\alpha$ is  a Cartan involution of $V$ (see \autoref{background}).
As $\alpha$ is an automorphism of $V$, $\chi(\alpha) = 1$ and  $\alpha\in \Str(V)^+$. So $\Theta$ is an automorphism of $\Co(V)^+$, which can be shown to be a Cartan involution of $\Co(V)^+$. Moreover, 
\[\Theta(\imath) = \imath \quad{\text{and}}  \quad \imath\circ\delta_t\circ\imath=\delta_{1/t}, \]
 so that $\imath$ belongs to the maximal compact subgroup $(\Co(V)^+)^\Theta$
of $\Co(V)^+$ and is a representantative  of   the non trivial Weyl group element for $(\Co(V)^+,\Aff(V)^+)$.
Since $\alpha$ commutes to $\imath$, the involution $\Theta$  preserves $\Str(V)^+$. Let
\begin{equation*}
\bar{N} := \Theta(N) = \imath \circ N\circ \imath ,\quad \bar{P} := \Theta(P) =\Str(V)^+\ltimes \bar{N}\ .
\end{equation*}
   
 For any $g\in \Co(V)$, denote by $V_g$ the (open, dense) subset of $V$ where $g$ is defined. It is known that for $x\in V_g$, the differential $Dg(x)$ of $g$ at $x$ belongs to  $\Str(V)$.  Moreover for any $g\in \Co(V)$, the map
$x\mapsto Dg(x)^{-1}$ extends  \emph{polynomially}  from $V$ into $\End(V)$.
As a consequence of the chain rule and using \autoref{deriota} and \autoref{chiP}, for any $g\in \Co(V)^+$ and any $x\in V_g,$ the differential $Dg(x)$   belongs to $\Str(V)^+$. 
  For $g\in \Co(V)^+$ and $x\in V_g$, let
\begin{equation*}
c(g,x) := \chi\big(Dg(x)\big)^{-1}\ .
\end{equation*}
In view of the chain rule, the map $g\mapsto c(g,x)$ satisfies the following cocycle relation :
   \begin{equation}\label{cocyclec}
c(g_1g_2, x) = c\big(g_1,g_2(x)\big) c(g_2,x)\quad g_1,g_2\in \Co(V)^+,\quad  x\in V
\end{equation}
whenever both sides are defined.
Moreover,
 \begin{itemize}
\item[(i)] For any $\ell\in \Str(V)^+$ and $x\in V$
\[c(\ell, x) = \chi(\ell)^{-1}.
\]
\item[(ii)]  For any $v\in V$ and $x\in V$
\[c(n_v,x) = 1.
\]
\item[(iii)]  For any $x\in V^\times$,
\[c(\imath, x) = \det(x)^2.
\]
\end{itemize}
The cocycle property \autoref{cocyclec} 
satisfied by $c$ implies that, for any $g\in \mathrm{Co}(V)^+$, the function $x\mapsto c(g,x)$ is a rational function on $V$. On the other hand, for $\ell\in \mathrm{Str}(V)^+$, 
$   \chi(\ell)^{{2n}\over r}=\Det(\ell)^2$, so that for $x\in V_g$
\[
c(g,x)^{2n\over r} = \chi\big(Dg(x)^{-1}\big)^{2n\over r}= \Det\big(Dg(x)^{-1}\big)^2.
\]
Hence, for any $g\in \mathrm{Co}(V)^+$, the function $x\longmapsto c(g,x)^{2n\over r}$ extends as a polynomial on $V.$ But if a power\footnote{Recall that $2n\over r$ is an integer.} of a rational function coincides on a Zariski open subset with a polynomial, then the rational function  has to be a polynomial, and therefore $x\longmapsto c(g,x)$ extends as a polynomial on $V.$

\begin{proposition} For any $g\in \Co(V)^+,$ there exists a polynomial $p_g\in \mathcal{P}(V)$ such that, for $x\in V_g$
\begin{equation}\label{pg}
c(g,x) = p_g(x)^2.
\end{equation}
The polynomial $p_g$ is unique up to a sign $\pm$.
\end{proposition}

\begin{proof} In fact,
\begin{itemize}
\item[-]  if $g=n_v$ is a translation, then $p_{n_v} \equiv 1$ satisfies \autoref{pg},
 \item[-]  if $g=\ell\in \Str(V)^+$,   then $p_\ell\equiv\chi(\ell)^{-\frac{1}{2}}$  satisfies  \autoref{pg},
 \item[-] if  $g=\imath$, then $p_\imath(x) = \det(x)$ satisfies \autoref{pg}.
\end{itemize}
 Recall that  $\Co(V)^+$ is generated by $\Str(V)^+, N$ and $\imath$. Hence, for any $g\in \Co(V)^+$,  the existence of a rational function  satisfying \autoref{pg} follows from the cocycle relation \autoref{cocyclec}. But if the square of a rational function  is a polynomial, then the rational function  has to be a polynomial. For the uniqueness,  let $g\in \Co(V)^+$.  If $p_g$ and $q_g$ are two such that $p_g^2(x)=q_g^2(x)= c(g,x)$ for $x\in V_g$, then by continuity, $p_g^2 = q_g^2$  or equivalently $(p_g+q_g)(p_g-q_g) = 0$,   and $p_g = \pm q_g$ follows.
\end{proof}

Let \[ G = \Big\{(g, p_g); \quad g\in \Co(V)^+,\; p_g\in \mathcal{P}(V),\;  p_g(x)^2 = c(g,x)\, \forall x \in V\Big\},\] 
and define the product of two elements $\widetilde g_1 = (g_1, p_{g_1})$ and $\widetilde g_2=(g_2, p_{g_2})$ of $G$ by  
\[\widetilde g_1 \widetilde g_2=\big(g_1g_2, (p_{g_1}\circ g_2) p_{g_2}\big).g
\]
To see that this definition makes sense, observe that $(p_{g_1}\circ g_2)\,p_{g_2}$ is a rational function on $V$. Its square (wherever defined) satisfies 
\[\big((p_{g_1}\circ g_2)\,p_{g_2}\big)^2(x) = c(g_1,g_2(x))\, c(g_2,x) = c(g_1g_2, x),
\]
which extends as a polynomial on $V.$ As a consequence, $(p_{g_1}\circ g_2)\,p_{g_2}$ itself extends to $V$ as a polynomial, and its square equals   $c(g_1g_2, x)$, this proves that $\big(g_1g_2, (p_{g_1}\circ g_2) p_{g_2}\big)$ is an element of $G$.

Now $G$ can be endowed with a  Lie group structure yielding a twofold covering of $\Co(V)^+$.
The various notions defined (and results obtained) for $\Co(V)^+$ have their counterparts for $G$. The subgroups $A$ and $N$ of $\Co(V)^+$ are identified with subgroups of $G$ via
\[\delta_t\longmapsto (\delta_t, t^{-\frac{r}{2}}), \qquad n\longmapsto (n,1).
\] 
The inversion $\imath$ is identified with $(\imath, \det(x))$ and the Cartan involution 
 $\alpha$ is identified with $(\alpha,1)$.
Finally, let $L$ and $P$ be the subgroups of $G$ defined by
\begin{eqnarray*}
L&=&\{ \big(\ell, \pm  \,\chi(\ell)^{-\frac{1}{2}}\big);\quad \ell\in \Str(V)^+\},\\
P&=& \{ \big(p,\pm \, \chi(p)^{-\frac{1}{2}}\big);\quad p\in \Aff(V)^+ \},
\end{eqnarray*}
which are twofold (trivial) coverings of respectively $\Str(V)^+$ and $\Aff(V)^+$.
Above we have extended the character $\chi$ of $\Str(V)$ to a character of $\Aff(V)^+=\Str(V)^+\ltimes N$ by letting it acting trivially on $N$. Then $P$ is a maximal parabolic subgroup of $G$ with Langlands decomposition $P = LN$.
We keep the notation   $\Theta$ for the Cartan involution of $ G$ given by $\Theta(g) = \alpha\circ \imath \circ g\circ \imath\circ \alpha$. 
Bruhat decomposition of $G$ takes the form
 \begin{equation}\label{Bruhat}
 G=P\imath P \cup P.
 \end{equation}
  For a given $\widetilde g = (g, p_g)\in G$ and $x\in V_g$, we set $\widetilde g(x) := g(x)$. Further,  we wil use the following notation 
\begin{equation}\label{a}  
a(\widetilde g, x) := p_g(x),\quad x\in V.
\end{equation}
 Then $a(\widetilde g,\,.\,)$ is a polynomial on $V$ satisfying the cocycle relation
\begin{equation}\label{cocyclep}
a(\widetilde g_1\, \widetilde g_2, x) = a\big(\widetilde g_1, \widetilde g_2(x)\big)\, a(\widetilde g_2,x),\quad \widetilde g_1, \widetilde g_2\in  G, x\in V.
\end{equation}
We will abuse notation and denote elements of $G$ without tilde.
Remark that for all $g\in  G,$
\[g \text{ defined at } x \Longleftrightarrow a(g,x)\neq 0.
\]
The following equalities are immediate from above
$$\begin{array}{rllll}
a(n_v,x)&=&1 & \quad a_v\in N\\
a((\ell, \pm  \,\chi(\ell)^{-\frac{1}{2}}),x)&=&\pm (\chi(\ell))^{-\frac{1}{2}} & \quad  \ell\in \Str(V)^+\\
a(\imath,x)&=&\det(x) .& 
\end{array}$$
\begin{proposition} Let $ g\in  G$ and $x,y\in V$ such that $ g$ is defined at $x$ and $y$. Then
\begin{equation}\label{covdet}
\det\big( g(x) -  g(y)\big) = a( g,x)^{-1} \det(x-y) \,a( g, y)^{-1}.
\end{equation}
\end{proposition}
\begin{proof} For $\ g = (n_v,1)$ the identity  \autoref{covdet} is obvious. For $\ g = (\ell, \pm\chi(\ell)^{-\frac{1}{2}})$ we have
\[\det\big(\ell(x-y)\big) = \chi(\ell) \det(x-y)= \big(\pm\,\chi(\ell)^{-\frac{1}{2}})^{-1} \,\det(x-y) \,\big(\pm\,\chi(\ell)^{-\frac{1}{2}})^{-1}.
\]
  Finally, for $ g = (\imath, \det(x))$, we have 
$$
\det\big(\imath(x)-\imath(y)\big) = \det(-x^{-1} +y^{-1})=(\det \,x)^{-1} \det(x-y) (\det \,y)^{-1},
$$
which is essentially {\it Hua's formula}. The general formula follows by using the cocycle relation \autoref{cocyclep}.
\end{proof}

 \addsubsection{The degenerate principal series and Knapp-Stein intertwining operators}\label{sub-sec-deg-series}

The degenerate principal series is the family of representations of $G$ (smoothly) induced by the characters of the parabolic subgroup opposed to $P$. In this section, we will consider the \emph{non-compact} realization of these representations.

For $(\lambda, \varepsilon)\in \mathbb C\times \{ \pm\}$ and for  $g\in G$ let
\begin{equation}\label{degprinseries}
\big(\pi_{\lambda, \varepsilon}(g)f\big)(x) = a(g^{-1},x)^{-\lambda, \varepsilon} f\big(g^{-1}(x)\big) .
\end{equation}
In this formula, $f$ is a smooth function on $V$. As the action of $G$ on $V$ is not defined everywhere, we let
$\mathcal{H}_{\lambda,\varepsilon}$ to be the subspace of functions $f\in C^\infty(V)$ such that 
\[ \forall g\in  G:\quad 
V_g\ni x\longmapsto \pi_{\lambda, \varepsilon}(g) f (x) \;\text{extends as a smooth function on } V.
\]
By the Bruhat decomposition $ G =  P\cup  P\imath P$,   the space $\mathcal H_{\lambda, \varepsilon}$ can be realized as the subspace of functions $f\in C^\infty(V)$ such that
\[V^\times \ni x\longmapsto  \big(\pi_{\lambda, \varepsilon}(\imath) f \big)(x)\;  \text{ extends as a smooth function on } V\ .
\]
The space $\mathcal H_{\lambda, \varepsilon}$ is equipped with a topology using the semi-norms defining the usual topology of $C^\infty(V)$ for both $f$ and $\pi_{\lambda, \varepsilon}(\imath)f$. Then \autoref{degprinseries} defines a continuous representation of $ G$ on $\mathcal H_{\lambda, \varepsilon}$. 

The space $\mathcal H_{\lambda, \varepsilon}$ is not well fitted in order to use Fourier analysis on $V$. So we will replace the representation of the group by its infinitesimal version. 

Let $f$ be in the space $\mathcal{D}(V)$ of  smooth functions on $V$ with compact support. Then for $g$ close to the neutral element of $G$, $g^{-1}$ is defined on the support of $f$, so that $\pi_{\lambda, \varepsilon}(g) f$ is a well defined smooth function on $V$. Denote by $\mathfrak{g}$ the Lie algebra of $G$. For  $X\in \mathfrak g$, let
$${\rm d}\pi_{\lambda, \varepsilon}(X) f =  {{\rm d}\over {{\rm d}t}}\pi_{\lambda,\varepsilon}(\exp(tX))f_{\big | t=0}.$$
This defines a representation of $\mathfrak g$ on $\mathcal D(V)$.  If $g$ is close enough to the identity, then $a(g^{-1},x)>0$ for all  $x\in \mathrm{Supp}(f)$, and therefore  the infinitesimal representation $d\pi_{\lambda,\varepsilon}$  does not depend on $\varepsilon$. Hence we will drop the index $\varepsilon$ from $d\pi_{\lambda,\varepsilon}.$ 

Finally, it is well-known (see e.g. \cite{Pevzner})  that $d\pi_\lambda(X)$ is a first order differential operator on $V$ with polynomial coefficients of degree $\leq 2$. As a consequence we let the infinitesimal representation act on the Schwartz space $\mathcal S(V)$.


Let $(\lambda, \varepsilon) \in \mathbb C\times \{\pm\}$. For $f\in \mathcal S(V)$,  define the \emph{Knapp-Stein intertwining operator} by
\begin{equation}\label{KSdef}
I_{\lambda, \varepsilon} f(x) := \int_V \det(x-y)^{-\frac{2n}{r}+\lambda,\varepsilon} f(y) dy.
\end{equation}

For $\Re \lambda \gg 0$, the function $\det(x)^{-\frac{2n}{r}+\lambda, \varepsilon}$ is locally integrable and has polynomial growth at infinity, so that it can be considered as a tempered distribution on $V$, and hence \autoref{KSdef} defines an operator from $\mathcal S(V)$ into $\mathcal S'(V)$. Next, the definition is extended meromorphically to $\mathbb C$ by analytic continuation in the parameter $\lambda$, thus defining a family of operators from $\mathcal S(V)$ into $\mathcal S'(V)$ depending meromorphically on $\lambda$.

Let $f\in \mathcal D( V)$. For any $g\in G$ such that $g$ is defined on the support of $f$,
\[\big(I_{\lambda, \varepsilon} \circ \pi_{\lambda, \varepsilon}(g)\big)f = \big(\pi_{\frac{2n}{r}-\lambda,\varepsilon}(g) \circ I_{\lambda, \varepsilon}\big)f,
\]
as can be proved first  for $\Re \lambda\gg 0$, using \autoref{covdet}, followed by analytic continuation in the parameter $\lambda$. In particular it implies that 
\begin{equation}\label{71}
  I_{\lambda,\varepsilon}\circ d\pi_\lambda(X) = d\pi_{\frac{2n}{r}-\lambda}(X)\circ I_{\lambda,\varepsilon},\quad \forall X\in \mathfrak g,
\end{equation}
as an equality of operators from $\mathcal S(V)$ into $\mathcal S'(V).$
\section{The covariance property of the operators $F_{\lambda, \mu}$}\label{sec-cov-Fst}
\addsubsection{The infinitesimal covariance property}\label{subsec-inf-cov}
Recall that $M$ denotes the multiplication operator on $\mathcal S(V\times V)$  defined  by
\[Mf(x,y)=\det(x-y)f(x,y) .
\]
\begin{lemma}
Let $f$ be a smooth function on $V\times V$ with compact support. Let $g\in  G$ close to the neutral element so that $g$ (acting diagonally on $V\times V$) is defined on the support of $f$. Then,  for any $(\lambda, \varepsilon)$ and $(\mu, \eta)$ in $\mathbb C\times\{\pm\}$, we have
\[M\circ \big(\pi_{\lambda, \varepsilon}(g) \otimes \pi_{\mu, \eta}(g)\big) f = \left(\big(\pi_{\lambda-1, -\varepsilon}(g) \otimes \pi_{\mu-1, -\eta}(g)\big)\circ M\right) f.\]
\end{lemma}
\begin{proof} This is an immediate consequence of the covariance property \autoref{covdet} of the function $\det(x-y)$.
\end{proof}

For $X\in\mathfrak{g}$, let $d(\pi_\lambda\otimes\pi_\mu)(X)=d\pi_\lambda(X)\otimes\mathrm{Id}+\mathrm{Id}\otimes d\pi_\mu(X)$ be the infinitesimal representation of $\mathfrak{g} $ induced by $\pi_{\lambda,\varepsilon}\otimes \pi_{\mu,\eta}.$
The infinitesimal version of the above lemma reads :
 \begin{proposition}\label{prop71} 
For any $X\in \mathfrak g$, we have
\[M\circ  d(\pi_\lambda\otimes \pi_\mu)(X) =   d(\pi_{\lambda - 1} \otimes \pi_{\mu -1} )(X) \circ M.\]
\end{proposition}

To make connection with the notation of \autoref{sec-Est}  and \autoref{sec-DPS-KS}, first observe that
\[I_{\lambda, \varepsilon} = J_{-\frac{2n}{r}+\lambda, \varepsilon},
\]
where $J_{s,\varepsilon}$ was defined in \autoref{def-J-st}.
For notational convenience, we let
$$F_{\lambda, \mu} = E_{\frac{n}{r}-\lambda, \frac{n}{r}-\mu},$$
where $E_{s,t}$ is defined by \autoref{def-E(s,t)}. 
Then we may rephrase \autoref{them-main} as follows:
\begin{proposition}\label{72} For all $(\lambda,\varepsilon)$ and $(\mu,\eta)$ in $\mathbb C\times\{\pm\},$ we have 
$$M\circ (I_{\lambda, \varepsilon}\otimes I_{\mu, \eta}) = k(\lambda , \mu )(I_{\lambda+1,-\varepsilon}\otimes I_{\mu+1,-\eta}) \circ F_{\lambda, \mu},
$$
where $k(\lambda, \mu)$ is a meromorphic function on $\mathbb C\times\mathbb C$.
\end{proposition}
Further, we have the following covariance property of the operators $F_{\lambda,\mu}.$

\begin{theorem} For all  $\lambda, \mu\in \mathbb C$ and  for any $X\in \mathfrak g$, we have 
\begin{equation}\label{covE}
F_{\lambda, \mu} \circ \big[(d(\pi_\lambda\otimes \pi_\mu)(X)\big]  = \big[(d(\pi_{\lambda+1}\otimes \pi_{\mu+1})(X)\big]\circ F_{\lambda, \mu}\end{equation}
\end{theorem}
\begin{proof} In view of \autoref{prop71}, \autoref{72} and the identity \autoref{71}, we have  
\begin{eqnarray*}
 &&\big(I_{\lambda+1, -\varepsilon} \otimes I_{\mu+1,- \eta}\big)\circ F_{\lambda, \mu}\circ \big(d(\pi_\lambda\otimes \pi_\mu)(X)\big)\\
&&\qquad\qquad\qquad= \frac{1}{\kappa(\lambda, \mu)} M\circ (I_{\lambda, \varepsilon}\otimes I_{\mu, \eta})\circ d(\pi_\lambda\otimes d\pi_\mu)(X)\\
&&\qquad\qquad\qquad= \frac{1}{\kappa(\lambda, \mu)} M\circ (d(\pi_{\frac{2n}{r}-\lambda}\otimes \pi_{\frac{2n}{r}-\mu})(X))\circ \big(I_{\lambda, \varepsilon}\otimes I_{\mu, \eta}\big)\\
&&\qquad\qquad\qquad= \frac{1}{\kappa(\lambda, \mu)}d(\pi_{\frac{2n}{r}-\lambda-1}\otimes \pi_{\frac{2n}{r}-\mu-1})(X)\circ M\circ  \big(I_{\lambda, \varepsilon}\otimes I_{\mu, \eta}\big)\\
&&\qquad\qquad\qquad= d(\pi_{\frac{2n}{r}-\lambda-1}\otimes \pi_{\frac{2n}{r}-\mu-1})(X)\circ \big(I_{\lambda+1, -\varepsilon}\otimes I_{\mu+1, -\eta}\big)\circ F_{\lambda,\mu}\\
&&\qquad\qquad\qquad=\big(I_{\lambda+1, -\varepsilon}\otimes I_{\mu+1, -\eta}\big)\circ\big(d(\pi_{\lambda+1}\otimes \pi_{\mu+1})(X)\big)\circ F_{\lambda, \mu}.
\end{eqnarray*}
For generic $\lambda$, the operator $I_{\lambda,\varepsilon}$ is injective on $\mathcal S(V)$, so that \autoref{covE} follows. By continuity, \autoref{covE} extends to all $\lambda, \mu\in \mathbb C$.
\end{proof}

\addsubsection{The global covariance property}\label{subsec-global-cov}
There is a global formulation for the covariance property of the operators $F_{\lambda, \mu}$, but it is nicer to work in  the compact setting. Using the notation of \autoref{subsec-Co}, let $\overline P= \Theta(P)$ be the opposite parabolic subgroup of $G$. Let $\mathcal{X}= G/ \overline P$. The map
\[V\ni v\;  \longmapsto \; n_v\, {\overline P}
\] 
is a diffeomorphism onto an open dense subset of $\mathcal{X}$ (for that reason $\mathcal{X}$ is usually called the conformal compactification of $V$). Let $\chi_{\frac{1}{2}}$ be the character of $\overline{P}$ which is trivial on $\overline N=\Theta(N)$ and which on $L$ is defined by
\[\chi_{\frac{1}{2}}\big(\ell, \pm\chi(\ell)^{-\frac{1}{2}}\big) = \pm\chi(\ell)^{\frac{1}{2}},
\]
where $\chi$ is the character of $\Str(V)$ defined by \autoref{def-chi}.

For $(\lambda, \varepsilon)\in \mathbb C\times \{\pm\}$, consider the character $\chi_{\frac{1}{2}}^{\lambda, \varepsilon}$  of $ L$ and form the line bundle 
\[
\mathcal E_{\lambda, \varepsilon} =  G\otimes_{\overline P}\,  \mathbb C_{\chi_{\frac{1}{2}}^{\lambda, \varepsilon}}.
\]
Finally denote by $\Gamma(\mathcal E_{\lambda, \varepsilon})$ the space of smooth sections of $\mathcal E_{\lambda, \varepsilon}$. The natural actions of $G$ on $\Gamma(\mathcal E_{\lambda, \varepsilon})$ defines a smooth representation of $ G$ which is the compact realization of the representation $\pi_{\lambda, \varepsilon}$ considered previously in \autoref{degprinseries}. We keep same notation for this representation. 

Similarly, consider the \emph{normalized} Knapp-Stein intertwining operators 
$\widetilde J_{\lambda, \varepsilon} : \Gamma(\mathcal E_{\lambda, \varepsilon})\longrightarrow \Gamma(\mathcal E_{\frac{2n}{r}-\lambda, \varepsilon})$ (see \cite{Knapp}).  They form a \emph{holomorphic} family of intertwining operators. It is known that for generic $\lambda$, the representation $\pi_{\lambda, \varepsilon} $ is irreducible, so that 
\begin{equation}\label{depsilon}
 \widetilde J_{\frac{2n}{r}-\lambda, \varepsilon}\circ \widetilde J_{\lambda, \varepsilon}=d_\varepsilon(\lambda) \id 
 \end{equation}
for some holomorphic function $d_\varepsilon$.

By \autoref{covdet}, the kernel $\det(x-y)$ correspond to a $G$-invariant section of 
$\mathcal E_{-1,-} \otimes \mathcal E_{-1,-}$ and hence, for any $(\lambda, \varepsilon)$ and $(\mu,\eta),$ the corresponding multiplication operator  
$$M : \Gamma(\mathcal E_{\lambda,\varepsilon}\otimes \mathcal E_{\mu, \eta})\longrightarrow \Gamma(\mathcal E_{\lambda-1,-\varepsilon}\otimes \mathcal E_{\mu-1, -\eta})$$
intertwines $\pi_{\lambda, \varepsilon}\otimes \pi_{\mu, \eta}$ and $\pi_{\lambda-1, -\varepsilon}\otimes \pi_{\mu-1, -\eta}$.

Let $(\lambda, \varepsilon), (\mu,\eta)\in \mathbb C\times \{\pm\}$ and
 consider the operator
$$F_{(\lambda, \varepsilon),(\mu, \eta)} : \Gamma(\mathcal E_{\lambda,\varepsilon} \otimes \mathcal E_{\mu,\eta})\longrightarrow \Gamma(\mathcal E_{\lambda+1,-\varepsilon} \otimes \mathcal E_{\mu+1,-\eta})$$
defined by
$$F_{(\lambda, \varepsilon),(\mu, \eta)}= \big( \widetilde {J}_{\frac{2n}{r}-\lambda-1, -\varepsilon}\otimes \widetilde {J}_{\frac{2n}{r}-\mu-1, -\eta}\big)\circ M\circ \big(\widetilde {J}_{\lambda, \varepsilon}\otimes \widetilde {J}_{\mu, \eta}\big).$$
Notice that the family $F_{(\lambda, \varepsilon),(\mu, \eta)}$ depends holomorphically on $\lambda$ and $\mu.$ 
\begin{theorem} The operator $F_{(\lambda, \varepsilon),(\mu, \eta)}$ is a  {differential operator} which intertwines  the representations $\pi_{\lambda, \varepsilon}\otimes \pi_{\mu, \eta}$ and $\pi_{\lambda+1, -\varepsilon}\otimes \pi_{\mu+1, -\eta}$. 
\end{theorem}
\begin{proof}
The intertwining property of $F_{(\lambda, \varepsilon),(\mu, \eta)}$ follow immediately from its  construction. Now, by composing $F_{(\lambda, \varepsilon),(\mu, \eta)}$ with $\widetilde {J}_{\lambda+1,-\varepsilon}\otimes \widetilde  {J}_{\mu+1, -\eta}$ from the left and using \autoref{depsilon}, we get
$$\big(\widetilde {J}_{\lambda+1,-\varepsilon}\otimes \widetilde  {J}_{\mu+1, -\eta}\big) \circ F_{(\lambda, \varepsilon),(\mu, \eta)}= d_{-\varepsilon}(\lambda+1) d_{-\eta}(\mu+1) M\circ \big(\widetilde {J}_{\lambda, \varepsilon}\otimes \widetilde {J}_{\mu, \eta}\big).
$$
When translating this relation in the non-compact setting, the local expression of  $\widetilde {J}_{\lambda, \varepsilon}$ is (up to a  function of $\lambda$) equal to $I_{\lambda, \varepsilon}$ and $M$ corresponds to the multiplication by $\det(x-y)$. 
Comparing with the result obtained in \autoref{sec-cov-Fst}, this implies that the local expression of 
 $F_{(\lambda, \varepsilon),(\mu, \eta)}$ in the non-compact setting  is equal to $F_{\lambda, \mu}$, up to a meromorphic function of $\lambda$ and $\mu$. Hence $F_{(\lambda, \varepsilon),(\mu, \eta)}$ is a differential operator. This works for generic $\lambda$ and $ \mu$ and hence for every $\lambda$ and $ \mu$ by continuity. 
\end{proof}

\section{Conformally  covariant bi-differential operators}\label{sec-Bst}
Let $\mathbf{res} : \mathcal C^\infty(V\times V)\to \mathcal C^\infty(V)$ be the restriction map defined by 
$$\mathbf{res}(\varphi(x))=\varphi(x,x),\quad x\in V.$$
It satisfies 
$$\mathbf{res}\circ (\pi_{\lambda,\varepsilon}(g)\otimes \pi_{\mu,\eta}(g))=\pi_{\lambda+\mu,\varepsilon\eta}(g)\circ\mathbf{res},\qquad \forall g\in {G}.$$
For $\lambda, \mu\in\mathbb{C}$ and $N\in\mathbb{N}^*$ we define the bi-differential operator $ B_{\lambda,\mu}^{(N)}  : \mathcal C^\infty(V\times V)\longrightarrow  \mathcal C^\infty(V)$  by
$$B_{\lambda,\mu}^{(N)}=\mathbf{res}\circ F_{\lambda+N-1, \mu+N-1}\circ\cdots\circ  F_{\lambda, \mu}.$$
The covariance property of the operators $F_{\lambda, \mu}$ and of $\mathbf{res}$ imply the following statement.
\begin{theorem} For all $(\lambda,\varepsilon)$ and $ (\mu,\eta)$ in $\mathbb{C}\times \{\pm\}$ and for all $N$ in $\mathbb{N}^*$, the operator $B_{\lambda,\mu}^{(N)}$ is covariant with respect to $(\pi_{\lambda,\varepsilon}\otimes \pi_{\mu,\eta}, \pi_{\lambda+\mu+2N,\varepsilon\eta})$, i.e.
$$ B_{\lambda,\mu}^{(N)}\circ (\pi_{\lambda,\varepsilon}(g)\otimes \pi_{\mu,\eta}(g))= \pi_{\lambda+\mu+2N,\varepsilon\eta}(g)\circ B_{\lambda,\mu}^{(N)} ,\quad \forall g\in G.$$
\end{theorem}

 It is possible to give a slightly different presentation of the bi-differential operators $B_{\lambda,\mu}^{(N)}$. Below we will use the notation of \autoref{subsec-global-cov}. For $N\geq 1$, define the operator $F^{(N)}_{(\lambda, \varepsilon), (\mu,\eta)}$   by the following diagram 
\[
\begin{CD}
\mathcal H_{(\lambda,\varepsilon), (\mu,\eta)} @> F^{(N)}_{(\lambda, \varepsilon), (\mu,\eta)}>> \mathcal H_{(\lambda+N,\varepsilon_N),(\mu+N,\eta_N)}\\
@V\widetilde J_{\lambda,\varepsilon}\otimes \widetilde J_{\mu,\eta}VV @AA\widetilde J_{\frac{2n}{r}-\lambda-N, \varepsilon_N}\otimes\, \widetilde J_{\frac{2n}{r}-\mu-N,\eta_N}A\\
\mathcal H_{(\frac{2n}{r}-\lambda,\varepsilon), (\frac{2n}{r}-\mu,\eta)} @>M^N> >\mathcal H_{(\frac{2n}{r}-\lambda-N, \varepsilon_N),(\frac{2n}{r}-\mu-N,\eta_N)}
\end{CD}
\]
with the convention $\varepsilon_{N}=(-1)^N\varepsilon$. It is a generalization of the diagram presented in the introduction which corresponds to the case $N=1$. The operator $F^{(N)}_{(\lambda, \varepsilon), (\mu,\eta)}$ is by construction covariant with respect to $\pi_{\lambda,\varepsilon}\otimes \pi_{\mu,\eta}$ and  $\pi_{\lambda+N, \varepsilon_N} \otimes \pi_{\mu+N, \eta_N}$. By induction on $N$, it is possible to prove that $F^{(N)}_{(\lambda, \varepsilon), (\mu,\eta)}$ is a \emph{differential} operator on $\mathcal X\times \mathcal X$. In fact $F^{(N)}_{(\lambda, \varepsilon), (\mu,\eta)}$ coincides (up to a meromorphic function in $\lambda$ and $\mu$) with the operator
\[F_{(\lambda+N, \varepsilon_N),(\mu+N, \eta_N)}\circ\dots \circ F_{(\lambda,\varepsilon),(\mu, \eta)}\ .
\]
The corresponding bi-differential operator is
\[\res\circ F^{(N)}_{(\lambda, \varepsilon), (\mu, \eta)},
\]
and its expression in the non-compact realization coincides (up to a meromorphic function in $\lambda$ and $\mu$) with $B^{(N)}_{\lambda, \mu}$.
\section{The case of $\mathbb R^{p,q}$ with $ p\geq 2$ and $ q\geq 1$}\label{sec-Rpq}

  Let $E$ be a real vector space of dimension $n-1$, endowed with a non degenerate symmetric bilinear form $\beta:E\times E\rightarrow  \mathbb R$ of signature $(p-1,q)$ where $p+q=n$. Then
  $V:=\mathbb R \times  E$ is a simple real  Jordan algebra with multiplication  given by
$$(\lambda,v)\cdot (\mu,w) = (\lambda \mu-\beta(v,w), \lambda w+\mu v).$$ 
 The dimension of  $V$ equals  $n=p+q$ and  its rank is  $2.$ The   neutral element of $V$  is $e=(1,0_E)$ and its   determinant is given by
$$\det(\lambda, v) = \lambda^2+\beta(v,v),$$
 which is a quadratic form of signature $(p,q)$ and the trace  is
 $$\tr(\lambda,v)=2\lambda.$$Therefore,   $(\lambda, v)\in V$ is invertible if and only if $ \lambda^2+\beta(v,v) \neq 0$ and its inverse is given by
$$(\lambda, v)^{-1} = (\lambda^2+\beta(v,v))^{-1}(\lambda, -v).$$
 Fix a basis $(e_2,\dots, e_n)$ of $E$ with coordinates $(x_2,\dots, x_n)$ chosen so that, for $v=(x_2,\dots, x_n)$,
$$\beta(v,v) = x_2^2+\dots+x_p^2-x_{p+1}^2-\dots -x_n^2.$$
 Below we will denote $V$ by $\mathbb R^{p,q}$ and  its determinant   by $$P(x) = x_1^2+x_2^2+\dots+x_p^2-x_{p+1}^2-\dots -x_n^2,$$
 where $x=(x_1,x_2,\ldots, x_n)\in V.$
 In this notation the neutral element is $e=e_1$. 
 
 \addsubsection{The Zeta functional equation}
%

Let $V'$ be the dual space  of $V$. The symmetric bilinear form on $V$ associated to $P$ induces an isomorphism of $V'$ with $V$, and so we can transfer the quadratic form on $V$ to a quadratic form on $V',$ which we denote also by $P$. More explicitly, let $(e'_1,\dots, e'_n)$ be the basis of $V'$ which is the dual to the canonical basis of $V$, and denote by $(\xi_1,\dots, \xi_n)$  the coordinates of an arbitrary element $\xi\in V'$. Let 
\[P(\xi) = \xi_1^2+\cdots+\xi_p^2-\xi_{p+1}^2-\dots-\xi_n^2.
\]
The corresponding symmetric bilinear form on $V'\times V'$ is given by 
$$P(\xi,\zeta) = \xi_1\zeta_1+\dots +\xi_p\zeta_p-\xi_{p+1} \zeta_{p+1} -\dots -\xi_n\zeta_n.$$

In a departure from our general convention  \autoref{def-Fourier}, we define in this section    the Fourier transform $\mathcal F : \mathcal S(V) \longrightarrow \mathcal S(V')$ by
 $$\mathcal Ff(\xi) = \int_V e^{i(\xi,x)} f(x) dx,
 $$ and extend it by duality to $\mathcal S'(V)$. 
 
Below, for $(s,\varepsilon)\in \mathbb C\times \{\pm\},$  $P^{s,\varepsilon}(x)$ stands for $\det(x)^{s,\varepsilon}.$


 \begin{theorem}\label{E-F-Z-pq}
  For every $s\in \mathbb C$, we have 

 \begin{equation}\label{Fzetaquad}
 \mathcal F\begin{pmatrix}P^{s,+}\\P^{s,-}  \end{pmatrix}= \gamma(s){\mathbf A}(s)\begin{pmatrix}P^{-s-\frac{n}{2},+}\\P^{-s-\frac{n}{2},-}  \end{pmatrix}
 \end{equation}
 where 
 $$\mathbf{A}(s)=
  \left(\begin{array}{lr}
 \displaystyle \cos\frac{(p-q)\pi}{4} \left(-\sin(s+\frac{n}{4}) \pi+ \sin \frac{n\pi}{4}\right)
& \displaystyle 
\sin\frac{(p-q)\pi}{4} \left(\cos(s+\frac{n}{4}) \pi- \cos \frac{n\pi}{4}\right)\\
\displaystyle 
 \sin\frac{(p-q)\pi}{4} \left(\cos(s+\frac{n}{4}) \pi+ \cos \frac{n\pi}{4}\right)
& \displaystyle 
-\cos\frac{(p-q)\pi}{4} \left(\sin(s+\frac{n}{4}) \pi+ \sin \frac{n\pi}{4}
  \right)
\end{array}\right).
$$    and
\begin{equation}\label{gammma}\gamma(s) = 2^{2s+n}\pi^{\frac{n}{2}-1} \Gamma(s+1)\Gamma(s+\frac{n}{2}).\end{equation}
 \end{theorem}
 \begin{proof} Introduce \[P_+(x) = \left\{\begin{matrix}P(x) &\text{ on }\{P(x)>0\}\\ \\ 0 &\text{ on } \{P(x)<0\} \end{matrix} \right. 
\qquad P_-(x) = \left\{\begin{matrix}0 &\text{ on }\{P(x)>0\}\\ \\ -P(x)&\text{ on } \{P(x)<0\} \end{matrix} \right.\ .\]
For $s\in \mathbb C$ with  $\Re (s)>-1$,  the functions $P_+^s$ and $P_-^s$ are locally integrable  with moderate growth at infinity. Further, they can be extended, as tempered distributions, meromorphically for $s\in \mathbb C$. Their Fourier transforms are given by
 \[\mathcal F(P_+^s) = \gamma(s)\left\{ -\sin\left(\frac{q}{2}+s\right)\!\!\pi\,P_+^{-s-\frac{n}{2}}+ \sin \left(\frac{p\pi}{2}\right)\, P_-^{-s-\frac{n}{2}}\right\},
  \]
  and
  \[\mathcal F(P_-^s) =  \gamma(s)\left\{ \sin\left(\frac{q\pi}{2}\right)P_+^{-s-\frac{n}{2}} -\sin\left(s+\frac{p}{2}\right)\!\!\pi \,P_-^{-s-\frac{n}{2}}\right\},
  \]
 where $\gamma(s) = 2^{2s+n}\pi^{\frac{n}{2}-1} \Gamma(s+1)\Gamma(s+\frac{n}{2})$.
 See \cite[(2.8) and (2.9)]{GS} (or \cite{S}). Now
 \[P^{s,+} = P_+^s + P_-^s \quad\text{and}\quad  P^{s,-} = P_+^s-P_-^s,
 \]
 and \autoref{Fzetaquad} follows by routine computation.
 \end{proof}
We will use the following notation for the coefficients of the matrix ${\mathbf A}(s)$ :
$${\mathbf A}(s) = \begin{pmatrix} a_{+,+}(s)&a_{+,-}(s)\\ a_{-,+}(s)&a_{-,-}(s)\end{pmatrix}.$$
 Observe that the matrix-valued function ${\mathbf A}(s)$ is periodic of period $2$ and   the coefficients $a_{\varepsilon, \eta}$ satisfy 
 \begin{equation}\label{s+1}
 a_{\varepsilon, \eta}(s+1) = -a_{-\varepsilon, -\eta}(s),
 \end{equation}
 for every   $\varepsilon, \eta =\pm$.

\addsubsection{The main identity for $\mathbb R^{p,q}$}
Recall from  \autoref{Dstgen} that there exists a differential operator $D_{s,t}$ on $V'\times V'$ such that
\begin{equation}\label{defDst}
P\left(\frac{\partial}{\partial \xi}-\frac{\partial}{\partial \zeta}\right)P(\xi)^{s, \varepsilon}P(\zeta)^{t, \eta}f(\xi,\zeta)= P(\xi)^{s-1,-\varepsilon} P(\zeta)^{t-1,-\eta}D_{s,t} f(\xi,\zeta),
\end{equation}
for every   $f$ in $\mathcal S(V'\times V')$.    
\begin{theorem} For $V=\mathbb R^{p,q},$ the differential operator $D_{s,t}$ is given explicitly by 
\begin{equation}\label{DstRpq} 
\begin{array} {rcl}
D_{s,t} &=& \displaystyle  P(\xi)P(\zeta)\,P \left(\frac{\partial}{\partial \xi} -\frac{\partial}{\partial \zeta}\right)\\
&&\displaystyle  +4s\, P(\zeta) \sum_{i=1}^n \xi_j\left( \frac{\partial }{\partial \xi_j}-  \frac{\partial }{\partial \zeta_j}\right)+4t\,P(\xi)\sum_{j=1}^n \zeta_j\left( \frac{\partial }{\partial \zeta_j}-  \frac{\partial }{\partial \xi_j}\right)\\
&&\displaystyle  +2t(2t-2+n) P(\xi)-8stP(\xi,\zeta) +2s(2s-2+n) P(\zeta).
\end{array}
\end{equation}
 
\end{theorem}

\begin{proof} It is enough to prove the identity on the open subset \[\{(\xi, \zeta),\quad  P(\xi), P(\zeta) >0\} .\]
Elementary calculations show that for $f$ a smooth function on $\{P(\xi)>0\}$, we have 
\begin{multline*}P\left( \frac{\partial}{\partial \xi}\right)P(\xi)^s f(\xi)=
2s(2s-2+n) P(\xi)^{s-1} f(\xi) \\+ 4s \sum_{i=1}^n \xi_i\,P(\xi)^{s-1} \frac{ \partial f}{\partial \xi_i}(\xi)+ P(\xi)^{s} P\left(\frac{\partial }{\partial \xi}\right) f(\xi),
\end{multline*}
and for $f$ a smooth function on $\{P(\xi)>0,\,P(\zeta)>0\}$, we have 
$$\begin{array} {rcl}
&&\displaystyle P\left(\frac{\partial}{\partial \xi }, \frac{\partial}{\partial \zeta } \right)\big(P(\xi)^sP(\zeta)^t f(\xi,\zeta) \big)\\
&&\qquad\qquad\qquad\qquad\displaystyle=  4st\,P(\xi, \zeta) P(\xi)^{s-1} P(\zeta)^{t-1} f (\xi,\zeta) \\
&&\qquad\qquad\qquad\qquad\quad\displaystyle+2s\, P(\xi)^{s-1} P(\zeta)^t \sum_{i=1}^n \xi_i\frac{\partial f}{\partial \zeta_i}+2t P(\xi)^s P(\zeta)^{t-1}\sum_{i=1}^n \zeta_i\frac{\partial f}{\partial \xi_i} \\
&&\qquad\qquad\qquad\qquad\quad \displaystyle+P(\xi)^sP(\zeta)^t P\left(\frac{\partial}{\partial \xi }, \frac{\partial}{\partial \zeta } \right)f(\xi,\zeta).
\end{array}$$
Now the identity \autoref{DstRpq} is a matter of putting pieces together. 
\end{proof}

\addsubsection{Explicit form of the operator $E_{s,t}$}
Recall from \autoref{def-E(s,t)} the definition of the   differential operator $E_{s,t}$ defined on $V\times V$ by $  E_{s,t}  = \mathcal F^{-1}\circ D_{s,t} \circ \mathcal F.$
\begin{proposition}\label{Est} In the $V=\mathbb R^{p,q}$ case, the differential operator $E_{s,t}$ is  given explicitly by 
\begin{eqnarray*}
&&E_{s,t} = -P(x-y) P\left(\frac{\partial}{\partial x}\right)P\left(\frac{\partial}{\partial y}\right) + 
\\
&&\qquad \quad+4(s-1) \sum_{j=1}^n (x_j-y_j)\frac{\partial}{\partial x_j}P\left(\frac{\partial}{\partial y}\right) + 4(t-1) \sum_{j=1}^n (y_j-x_j)\frac{\partial}{\partial y_j}P\left(\frac{\partial}{\partial x}\right)
\\
&&\qquad\quad -2(s-1)(2s-n) P\left(\frac{\partial}{\partial y} \right)
+8(s-1)(t-1) P\left(\frac{\partial}{\partial x}, \frac{\partial}{\partial y}\right)
-2(t-1)(2t-n) P\left(\frac{\partial}{\partial x} \right).
\end{eqnarray*}
\end{proposition}

\begin{proof} 
From   \autoref{formule-Fourier} it follows that 
\begin{eqnarray*} && E_{s,t} = - P\left(\frac{\partial}{\partial x}\right)P\left(\frac{\partial}{\partial y}\right) P(x-y) \\
&&\qquad\quad+4s P\left(\frac{\partial}{\partial y}\right)\left(\sum_{i=1}^n \frac{\partial}{\partial x_i} (x_i-y_i)\right)\,+4t P\left(\frac{\partial}{\partial x}\right)\left(\sum_{i=1}^n \frac{\partial}{\partial y_i} (y_i-x_i)\right) 
\\
&&\qquad\quad-2t(2t-2+n) P\left(\frac{ \partial}{\partial x}\right) +8stP\left(\frac{ \partial}{\partial x}, \frac{ \partial}{\partial y}\right)-2s(2s-2+n) P\left( \frac{ \partial}{\partial y}\right) .
\end{eqnarray*}
It remains to put the differential operator in normal form, multiplication preceding differentiation.  First
\begin{eqnarray*}
&&P\left(\frac{\partial}{\partial x}\right)P\left(\frac{\partial}{\partial y}\right) P(x-y) \\
&&\qquad =  P(x-y) P\left(\frac{\partial}{\partial x} \right)P\left(\frac{\partial}{\partial y} \right)
+4\sum_{j=1}^n(x_j-y_j)\left(\frac{\partial}{\partial x_j} P\left( \frac{\partial}{\partial y}\right) - \frac{\partial}{\partial y_j} P\left( \frac{\partial}{\partial x}\right)\right)\\
&&\quad\qquad+2nP\left( \frac{\partial}{\partial x}\right)-8P\left( \frac{\partial}{\partial x},\frac{\partial}{\partial y}\right)+2nP\left( \frac{\partial}{\partial y}\right).
\end{eqnarray*}
Next,
$$P\left( \frac{\partial}{\partial y}\right)\left(\sum_{i=1}^n \frac{\partial}{\partial x_i} (x_i-y_i)\right)  =
nP\left( \frac{\partial}{\partial y} \right) -2P\left(\frac{\partial}{\partial x},\frac{\partial}{\partial y} \right)+\sum_{j=1}^n (x_j-y_j)\frac{\partial}{\partial x_j} P\left( \frac{\partial}{\partial y} \right).$$
Putting all pieces  together yields Proposition \ref{Est}.
\end{proof}

\addsubsection{The main theorem for $\mathbb R^{p,q}$}
For $(s,\varepsilon )\in \mathbb C\times\{ \pm\}$, recall the Knapp-Stein intertwining operator 
$$J_{s,\varepsilon}f(x) = \int_V  P(x-y)^{s,\varepsilon} f(y) dy.$$
The integral is well defined for $f \in \mathcal S(V)$ when $\Re s>-1$. It defines a convolution operator from $\mathcal S(V)$ into $\mathcal S'(V)$ and the function $s\longmapsto J_{s, \varepsilon}$ can be extended  meromorphically on $\mathbb C$.

Let $M$ be the multiplication operator given for $f\in \mathcal S(V\times V)$ (or $\mathcal S'(V\times V)$) by 
$$Mf(x,y) = P(x-y) f(x,y).$$ 

In view of the functional equation \autoref{Fzetaquad} of $P^{s,\varepsilon}$ and the relation \autoref{s+1}, the  proof of the statement below goes   along the same lines as that  of  \autoref{them-main} 
\begin{theorem}\label{maintheoremquad} For $(s,\varepsilon )$ and $ (t,\eta)$ in $\mathbb C\times\{ \pm\}$, we have
$$M\circ  \left( J_{s,\varepsilon}\otimes  J_{t,\eta} \right)=\kappa(s,t) \left(J_{s+1,-\varepsilon}\otimes J_{t+1,-\eta}\right)\circ E_{-s-\frac{n}{2},-t-\frac{n}{2}},$$ where 
 $$\kappa(s,t)=
\frac{1}{16(s+1)(s+\frac{n}{2})(t+1)(t+\frac{n}{2})} .$$
\end{theorem}

\addsubsection{The conformal group of $V=\mathbb R^{p,q}$}
For $t\in \mathbb R^*$, let $\delta_t $ be the dilation given by $V\ni  v\longmapsto tv$.
The structure group of $V$ is equal to 
\[ \Str(V) = \{ h\circ \delta_t,\quad  h\in {\rm O}(p,q),\; t\in \mathbb R_{>0}\}= {\rm O}(p,q)\times \mathbb R^*/{\sim} 
\] 
where $\sim$ is the equivalence relation $(h,\delta_t)\sim (-h,\delta_{-t})$. As 
\[P\big((h\circ \delta_t) (x)\big) = t^2 P(x),
\]
the character $\chi$ of $\Str(V)$ is given by $\chi(h\circ \delta_t) = t^2$ and therefore $\Str(V)^+ = \Str(V)$.

Let $W= \mathbb R \times V\times \mathbb R$. Set $e_0=(1,0,0)$ and $e_{n+1}=(0,0,1)$. We will use the following convention : for $w=(\alpha, v,\beta)\in W,$
\[\alpha(w) = \alpha,\quad (w)_V = v,\quad \beta(w) = \beta.
\]

On $W$ define the quadratic form
\[Q(\alpha, v, \beta) = P(v)-\alpha \beta.
\]
Notice that $Q$ is of signature $(p+1,q+1)$.
Let us consider the proper isotropic cone
\[\Xi = \big\{ w\in W, \quad w\neq 0, \;Q(w) = 0\big\},
\]
and let $X = \Xi/\mathbb R^*$, a real projective quadric. For $w\in W\smallsetminus\{0\}$, let $\llbracket w\rrbracket$ be its image in the projective space $\mathbb P(W)$.

\begin{lemma} The map $\kappa : V\longrightarrow X$ defined  by
\begin{equation}\label{completion}
V\ni v \longmapsto\kappa(v) =  \llbracket (1,v,P(v) )\rrbracket
\end{equation}
is a diffeomorphism of $V$ onto a dense open subset of $X$.
\end{lemma}
\begin{proof}
Since for $v\in V$,  $(1,v,P(v))$ belongs to $\Xi$ and is $\neq 0$, the map is well defined. It is smooth and injective. Let $(\alpha,v,\beta)$ be in $\Xi$. If $\alpha\neq 0$, then $\llbracket(\alpha,v,\beta)\rrbracket=\llbracket (1,v',P(v')\rrbracket$ where $v'=\alpha^{-1} v$. Hence the image of the map given by \autoref{completion} is equal to 
$X\cap \{\alpha(w) \neq 0\}$ and the lemma follows.
\end{proof}

In the sequel we let $\infty = \llbracket e_{n+1}\rrbracket$ and $o = \llbracket e_0\rrbracket$, both points of $X$.

Let $G={\rm O}(Q)\simeq {\rm O}(p+1,q+1)$ be the orthogonal group of the form $Q$. Then $G$ preserves $\Xi$ and commutes to the dilations, so acts on the space\footnote{Observe that this is a twofold covering of $O(Q)/\{\pm \Id\}$, which already acts on $X$.}$X$. This allows to define a (rational) action of $G$ on $V$ by setting
\[g(x) = \alpha\big(g\kappa(x)\big)^{-1} \,\big(g \kappa(x)\big)_V.
\]

Let $L$ be the subgroup of $G$ given by
\begin{equation*} L = \left\{
\begin{pmatrix} t^{-1}&0&0\\0&h&0\\0&0&t
 \end{pmatrix}, \quad  h\in  {\rm O}(p,q),\; t\in \mathbb R^*
 \right\}.
\end{equation*}
The elements $\pm \begin{pmatrix} t^{-1}&0&0\\0&h&0\\0&0&t
 \end{pmatrix}$ both act on $V$ by $v\longmapsto t\,hv$, which realizes $L$ as a twofold covering of the structure group $\Str(V)$.
 
For $a\in V,$ let $n_a$ be the linear transform of $W$ defined by
\[n_a(\alpha, v, \beta) = (\alpha, \alpha a +v, \alpha P(a)+2P(a,v)+\beta).
\]
It is easily verified that $n_a$ belongs to $G$.  The action of $n_a$ on $V$ is given by $n_a(v) = v+a$. Further, for $a,b\in V$, we have $n_a\circ n_b= n_{a+b}$   so that  $ N= \{ n_a,\; a\in V\}$ is an abelian subgroup of $G$. 
\begin{lemma} Let  $P=LN$. The stabilizer $G^\infty$ of $\infty$ in $G$ is equal to $P.$
\end{lemma}
\begin{proof} First, clearly $P$ stabilizes $\infty$. Next
let $g\in G^\infty$. 
Then $g$ preserves the subspace $(\mathbb Re_{n+1}\big)^\perp=  V\oplus \mathbb R e_{n+1}$. 
Let
$g e_0 = (\alpha, v,\beta)$. As $Q(e_0,e_{n+1}) =1$, $g e_0\notin (\mathbb Re_{n+1}\big)^\perp$ and  $\alpha$ is different from $0$. Moreover, \[0=Q(e_0) = Q(g e_0) = P(v) -\alpha \beta,\] so that
\[(\delta_{\frac{1}{\alpha}}\circ t_v) (1,0,0) = (\alpha,v,\beta).
\]
Let $g_1 = \big(\delta_{\frac{1}{\alpha}}\circ t_v \big)^{-1}\circ g$, so that
$g_1 e_0 = e_0$. Then $g_1$ stabilizes both $\infty$ and $o$. Hence $g_1$ stabilizes $(\mathbb R e_0\oplus \mathbb R e_{n+1})^\perp = V$ and  the restriction of $g_1$ to $V$ preserves the quadratic form $Q_{\vert V} = P$. So the matrix of $g_1$ is of the form
\[ \begin{pmatrix}s&0&0\\0&h&0\\0&0&t
 \end{pmatrix},
\]
where $h\in  {\rm O}(p,q)$ and $  st = 1$. In other words, $g_1$ belongs to $L$, and hence $g\in P$. The conclusion follows.
\end{proof}
Let $\imath$ be the element of $G$  defined by
$$\begin{pmatrix}  
0 & 0& 1\\
0& \mathrm{I}_{1,n-1}&0 \\
-1& 0& 1
\end{pmatrix}$$
where $\mathrm{I}_{1,n-1}=\mathrm{diag}(-1,1,\ldots,1)$.
Then for $x=(x_1,x_2,\dots, x_n)\in V^\times$, $\imath(x) = -\frac{\check{x}}{P(x)}=-x^{-1}$,  where    ${\check { x}} = (x_1,-x_2,\dots, -x_n)$. The group $G$ is generated by $P$ and $\imath$.

The character $\chi$ of $\Str(V)$    is given on $L$ by
\[\chi\left( \begin{pmatrix}t^{-1}&0&0\\0&h&0\\0&0&t
 \end{pmatrix}\right) = t^2,
\]
so that we choose $\chi_{\frac{1}{2}}$ to be defined  by
\[\chi_{\frac{1}{2}}\left( \begin{pmatrix}t^{-1}&0&0\\0&h&0\\0&0&t
 \end{pmatrix}\right) =t.
\]
It is then easily verified that the cocycle  $a(g,x)$, defined  in the general situation by \autoref{a}, is given by
$$a(g,x) = \alpha\big(g\kappa(x)\big).$$

For $(\lambda, \varepsilon)\in \mathbb C\times \{ \pm\}$, the  principal series representations $\pi_{\lambda, \varepsilon}$ is given by
\[\pi_{\lambda, \varepsilon}(g)f(x) = a(g,x)^{-\lambda, \varepsilon} f(g^{-1}(x)) .
\]
The Knapp-Stein intertwining operator is given by
\[I_{\lambda,\varepsilon} f(x) = \int_V f(y) P(x-y)^{-n+\lambda,\varepsilon} \,dy
\]
and satisfies
\[I_{\lambda, \varepsilon}\circ \pi_{\lambda, \varepsilon} (g)= \pi_{n-\lambda, \varepsilon}(g)\circ I_{\lambda, \varepsilon}.
\]

\addsubsection{The operator  $F_{\lambda,\mu}$}

As in the general case,   the main  \autoref{maintheoremquad} can be reinterpreted to give a covariance property for the differential operator $F_{\lambda, \mu}$ (or its global version as a differential operator on $X\times X$), just by the change of parameters $s={n\over 2}-\lambda$ and $ t={n\over 2}-\mu$.

The differential operator $F_{\lambda,\mu}$ is given by 
$$
\begin{array}{rl}
\displaystyle  F_{\lambda,\mu}
 &=  \displaystyle-P(x-y) P\left(\frac{\partial}{\partial x}\right) P\left(\frac{\partial}{\partial y}\right) 
\\
& \displaystyle\quad+4(-\lambda+\frac{n}{2}-1) \sum_{j=1}^n (x_j-y_j)\frac{\partial}{\partial x_j}P\left(\frac{\partial}{\partial y}\right)  + 4(-\mu+\frac{n}{2}-1) \sum_{j=1}^n (y_j-x_j)\frac{\partial}{\partial x_j}P\left(\frac{\partial}{\partial x}\right)
\\
&\displaystyle\quad 
+4\lambda (-\lambda+\frac{n}{2}-1)  P\left(\frac{\partial}{\partial y} \right)
+4\mu (-\mu+\frac{n}{2}-1)  P\left(\frac{\partial}{\partial x} \right)
\\
& \displaystyle\quad+8(-\lambda+\frac{n}{2}-1)(-\mu+\frac{n}{2}-1) P\left(\frac{\partial}{\partial x}, \frac{\partial}{\partial y}\right).
\end{array}$$

\begin{theorem}
For $(\lambda ,\varepsilon)$ and $ (\mu,\eta)$ in $ \mathbb C\times  \{ \pm\}$, the operator $F_{\lambda, \mu}$ is covariant with respect to 
$(\pi_{\lambda, \varepsilon} \otimes \pi_{\mu, \eta})$ and $ (\pi_{\lambda+1,-\varepsilon}\otimes \pi_{\mu+1, -\eta})$.
\end{theorem}  

The construction of covariant bi-differential operators for ${\rm O}(p+1,q+1)$ is then obtained as in the general case. Let us state the formula for $B^{(1)}_{\lambda, \mu}$, which is covariant with respect to 
$(\pi_{\lambda, \varepsilon} \otimes \pi_{\mu, \eta}, \pi_{\lambda+\mu+2, \varepsilon \eta})$ :
\begin{multline*}
B^{(1)}_{\lambda, \mu}=4\,\mathbf{res}\Big\{
\mu (-\mu+\frac{n}{2}-1) P\left(\frac{ \partial}{\partial x}\right) +\lambda (-\lambda+\frac{n}{2}-1)  P\left( \frac{ \partial}{\partial y}\right) \\
+2(-\lambda+\frac{n}{2}-1)(-\mu+\frac{n}{2}-1) P\left(\frac{ \partial}{\partial x}, \frac{ \partial}{\partial y}\right)\Big\}.
\end{multline*}
We point out  that these covariant bi-differential operators were already introduced in \cite{or}.

  \newpage
\section{Appendix A : Classification of simple real Jordan algebras}\label{Appendix-A}
 

  \noindent\resizebox{\textwidth}{8.5cm}{%
 \begin{threeparttable}
\begin{tabular}{||c||c|c|c|c||}
 \bottomrule
    
    &
   Euclidean &        
   Non-Euclidean&
  Non-Complex type  & Complex type\\
    
    &
   (split) &        
   (split)&
   (non split) &(non-split)\\

     &
  Type I &
   Type II  &
  Type III ($m=2\ell$) &Type IV \\
  \toprule
  \bottomrule
 $V$
    &$\mathrm{Sym}(m,\mathbb{R})$ & $\times$ & $\mathrm{Sym}(2\ell,\mathbb{R})\cap \mathrm{Mat}(\ell,\mathbb{H})$ &$\mathrm{Sym}(m,\mathbb{C})$\\
  $\mathfrak{co}(V)$
    & $\mathfrak{sp}(m,\mathbb{R})$&$\times$ &$\mathfrak{sp}(\ell,\ell)$ &$\mathfrak{sp}(m,\mathbb{C}$ \\
    $\mathfrak{str}(V)$
       &$\mathfrak{sl}(m,\mathbb{R})\oplus\mathbb{R}$  &$\times$ &$\mathfrak{sp}(2\ell)\oplus\mathbb{R}$ &$\mathfrak{sl}(m,\mathbb{C})\oplus\mathbb{C}$\\
 $n$
    &$m(m+1)/2$ &$\times$ &$\ell(2\ell+1)$ &$m(m+1)$\\
  $r$
    & $m$& $\times$& $2\ell$ &$2m$\\
   $d$
     & $1$&$\times$ &$4$ &$2$\\
    $e$
      &0 & $\times$&2 &1 \\
      $V^+$
         &$\mathrm{Sym}(m,\mathbb{R})$ & $\times$&$\mathrm{Herm}(\ell,\mathbb{C})$ &$\mathrm{Sym}(m,\mathbb{R})$\\  
        $r_+$
          &$m$ &$\times$ &$\ell$ &$m$ \\
          $d_+$
          &$1$ &$\times$ &$2$ &$1$ \\
        
\toprule
        \bottomrule
    $V$
      &$\mathrm{Herm}(m,\mathbb{C})$ & $\mathrm{Mat}(m,\mathbb{R})$ & $\mathrm{Mat}(\ell,\mathbb{H})$ &$\mathrm{Mat}(m,\mathbb{C})$\\
  $\mathfrak{co}(V)$
     & $\mathfrak{su}(m,m)$& $\mathfrak{sl}(2m,\mathbb{R})$&$\mathfrak{su}^*(4\ell)$ &$\mathfrak{sl}(2m,\mathbb{C})$\\
    $\mathfrak{str}(V)$
      & $\mathfrak{sl}(m,\mathbb{R})\oplus\mathbb{R}$&$\mathfrak{sl}(m,\mathbb{R})\oplus \mathfrak{sl}(m,\mathbb{R})\oplus\mathbb{R}$ &$\mathfrak{su}^*(2\ell)\oplus \mathfrak{su}^*(2\ell)\oplus\mathbb{R}$ &$\mathfrak{sl}(m,\mathbb{C})\oplus \mathfrak{sl}(m,\mathbb{C})\oplus\mathbb{C}$\\
 $n$
    &$m^2$ &$m^2$ &$4\ell^2$ &$2m^2$\\
  $r$
     & $m$& $m$&$2\ell$ &$2m$\\
   $d$
     &2 &2 &8 & 4\\
    $e$
       & 0&0 &3 &1\\
      $V^+$
         &$\mathrm{Herm}(m,\mathbb{C})$ &$\mathrm{Sym}(m,\mathbb{R})$ &$\mathrm{Herm}(\ell,\mathbb{H})$ &$\mathrm{Herm}(m,\mathbb{C})$\\  
        $r_+$
           & $m$&$m$ &$\ell$ &$m$\\
            $d_+$
           & $2$&$1$ &$4$ &$2$\\
    \toprule
        \bottomrule
   $V$
      &$\mathrm{Herm}(m,\mathbb{H})$ & $\mathrm{Skew}(2m,\mathbb{R})$ & $\times$ &$\mathrm{Skw}(2m,\mathbb{C})$\\
  $\mathfrak{co}(V)$
     & $\mathfrak{so}^*(4m)$& $\mathfrak{so}(2m,2m)$&$\times$ &$\mathfrak{so}(4m,\mathbb{C})$\\
    $\mathfrak{str}(V)$
       & $\mathfrak{su}^*(2m)\oplus\mathbb{R}$& $\mathfrak{sl}(2m,\mathbb{R})\oplus\mathbb{R}$&$\times$ &$\mathfrak{sl}(2m,\mathbb{C})\oplus\mathbb{C}$\\
 $n$
    &$m(2m-1)$ &$m(2m-1)$ &$\times$ &$2m(2m-1)$\\
  $r$
     & $m$& $m$&$\times$ &$2m$\\
   $d$
     & 4& 4&$\times$ & 8\\
    $e$
      &0 & 0&$\times$ &1\\
      $V^+$
        &$\mathrm{Herm}(m,\mathbb{H})$ &$\mathrm{Herm}(m,\mathbb{C})$ &$\times$ &$\mathrm{Herm}(m,\mathbb{H})$ \\  
        $r_+$
           & $m$& $m$&$\times$ &$m$\\
           $d_+$
           & $4$& $2$&$\times$ &$4$\\
    \toprule
        \bottomrule
  $V$
     &  $\mathbb{R}^{1,k-1}$ & $\mathbb{R}^{p,q} $ & $\mathbb{R}^{k,0}$ &$\mathbb{C}^{k}$\\
  $\mathfrak{co}(V)$
     &$\mathfrak{so}(2,k)$ & $\mathfrak{so}(p+1,q+1)$&$\mathfrak{so}(k+1,1)$ &$\mathfrak{so}(k+2,\mathbb{C})$\\
    $\mathfrak{str}(V)$
       & $\mathfrak{so}(1,k-1)\oplus\mathbb{R}$ & $\mathfrak{so}(p,q)\oplus\mathbb{R}$& $\mathfrak{so}(k)\oplus\mathbb{R}$ &$\mathfrak{so}(n,\mathbb{C})\oplus\mathbb{C}$\\
 $n$
   &$m$ &$p+q=k$ &$k$ &$2k$\\
  $r$
     &2 &2 &2 &4\\
   $d$
      &$k-2$ & $p+q-2$&0 &$2(k-2)$\\
    $e$
       & 0& 0&$k-1$ &1\\
      $V^+$
         &$\mathbb{R}^{1,k-1}$ &$\mathbb{R}^{1,q}$ &$\mathbb{R}^{1,0}$ &$\mathbb{R}^{1,k-1}$\\  
        $r_+$
           &2 &2 &1 &2\\
           $d_+$
           &$k-2$ &$q-1$ &0 &$k-2$\\
    \toprule
        \bottomrule
$V$
  &$\mathrm{Herm}(3,\mathbb{O})$ & $\mathrm{Herm}(3,\mathbb{O}_s)$ & $\times$ &$\mathrm{Herm}(3,\mathbb{O})_\mathbb{C}$\\
  $\mathfrak{co}(V)$
    & $\mathfrak{e}_{7(-25)}$&$\mathfrak{e}_{7(7)}$ &$\times$ &$\mathfrak{e}_7(\mathbb{C})$\\
    $\mathfrak{str}(V)$
        & $\mathfrak{e}_{6(-26)}\oplus\mathbb{R}$&$\mathfrak{e}_{6(6)}\oplus\mathbb{R}$ &$\times$ &$\mathfrak{e}_6(\mathbb{C})\oplus\mathbb{C}$\\
 $n$
   &27 & 27&$\times$ &54\\
  $r$
    & 3&3 &$\times$ & 16\\
   $d$
      &8 &8 &$\times$ &1\\
    $e$
       &0 &0 &$\times$ &1\\
      $V^+$
         & $\mathrm{Herm}(3,\mathbb{O})$& $\mathrm{Herm}(3,\mathbb{O})$&$\times$ &$\mathrm{Herm}(3,\mathbb{O})$\\  
        $r_+$
           & 3&3 &$\times$ &3\\
           $d_+$
           & 8& 4&$\times$ &8\\ 
    \toprule

  \end{tabular}
     \end{threeparttable}
}

 \section{Appendix B : Rank, generic minimal polynomial and determinant}\label{Appendix-B}
Because of lack of a convenient reference, the purpose of this appendix is to clarify the relations between a real simple Jordan algebra and its complexification.
  \addsubsection{Rank, generic minimal polynomial and determinant}
 
 Let $V$ be a unital Jordan algebra over $\mathbb F=\mathbb R \text{ or } \mathbb C$ and denote by $\mathbf{1}$ its unit element. Recall that the rank of an element $x\in V$ is defined by
 \[\mathrm{rk}_\mathbb{F}(x) = \min\{ k>0, (e,x,x^2,\dots,x^k) \text{ are $\mathbb F$-linearly dependant}\}
 \]

 The rank of $V$ is defined as 
 \[\mathrm{rk}_\mathbb{F}(V)= \max\{\mathrm{rk}_\mathbb{F}(x), x\in V\}.
 \]
 
  \begin{proposition} Let $V$ be a unital  real Jordan algebra and  $\mathbb V$ its complexification. Then 
 \[\mathrm{rk}_\mathbb{R}(V) = \mathrm{rk}_{\mathbb{C}}(\mathbb V).
 \]
 \end{proposition}
 \begin{proof}
 Let $\mathrm{rk}_\mathbb{R}(V) = r$. By assumption, there exists an element $x\in V$ such that
 \begin{equation}\label{r-1}
 \mathbf{1}\wedge x\wedge \dots\wedge x^{r-1}\neq 0,
 \end{equation}
 whereas for all elements $y\in V$,
 \begin{equation}\label{vanr}
 \mathbf{1}\wedge y \wedge \dots \wedge y^r=0.
 \end{equation}
 Let $\rho= \mathrm{rk}_\mathbb{C}(\mathbb V)$. Now \autoref{r-1} implies that $\rho\geq r$. Next the map 
 \[\mathbb V \ni z \longmapsto {\mathbf 1}\wedge z \wedge \dots \wedge z^r \in \Lambda^{r+1} (\mathbb V)\]
is holomorphic and vanishes on $V$ by \autoref{vanr}, hence everywhere. This shows that $\rho\leq r$. We can conclude that $r=\rho$.
 \end{proof}
 
 An element $x\in V$ is said to be regular if $\mathrm{rk}_\mathbb{F}(x) = \mathrm{rk}_\mathbb{F}(V)$. In the case where  $V$ is a  real Jordan algebra,    an element of $V$ is {\em regular} if and only if it is regular as an element of $\mathbb V$.
 
 Let $x\in V$. Then the subalgebra $\mathbb F [x]$ generated by $\mathbf{1}$ and $x$ is commutative and power associative. Let 
$ I(x) $ be the ideal of $\mathbb F[x]$ defined by
\[I(x)= \{ p\in \mathbb F [T], \quad p(x)=0\}.
\]
Since $\mathbb F[T]$ is a principal ring, $I(x)$ is generated by a monic polynomial, called the minimal polynomial of $x$ and denoted by $p_x$.

\begin{proposition} There exists polynomials $a_1,a_2,\dots, a_r\in \mathbb F[T]$ such that the minimal polynomial $p_x$ of every regular element of $V$ is given by
\[p_x(T) = T^r-a_1(x)T^{r-1}+\dots +(-1)^r a_r(x).
\]
\end{proposition}
The polynomial 
$$m(T,x)= m_x(T)=T^r-a_1(x)T^{r-1}+\dots +(-1)^r a_r(x)$$ is called the \emph{generic minimal polynomial} of $V$ at $x$.
The linear form $a_1$ is the {\em trace} of $V$ and the polynomial $a_r$ is the {\em determinant} of $V$,
$$\begin{array}{r @{\,=\,} l}
\tr(x)&a_1(x),\\
\det(x)&a_r(x).
\end{array}$$

\begin{proposition} Let $V$ be a real Jordan algebra and $\mathbb V$ its complexification. The generic minimal polynomial of $\mathbb V$ restricts to the generic minimal polynomial of $V$.
\end{proposition}
\begin{proof} This is a consequence of the local expression of these coefficients near a regular element, namely
\[a_j(x) = (-1)^{j-1}\frac{ \Det(\mathbf{1},x,\dots,x^{j-1},x^r,x^{j+1},\dots,x^{r-1},e_{r+1},\dots,e_n)}{\Det(\mathbf{1},x,x^2,\dots,x^{r-1},e_{r+1},\dots,e_n)},
\]
where $e_{r+1}, \dots, e_n$ are elements completing $e,x,x^2,\dots, x^{r-1}$ to a basis of $V$ (see \cite[proof of Proposition II.2.1]{FK}).
\end{proof}

\begin{corollary} Let $V$ be a real Jordan algebra, and let $\mathbb V$ be its complexification. Then the restriction to $V$ of the  determinant of $\mathbb V$ coincides with the determinant of $V$.
\end{corollary}
\addsubsection{Primitive idempotents} 

Let $x\in \mathbb V$. A complex number $\lambda$ is called an \emph{eigenvalue} of $x$ if  $\lambda$ is a root of the minimal polynomial $p_x$. An element is said to be semi-simple if its minimal polynomial has only simple roots. 

\begin{proposition} Let $(d_1,d_2,\dots, d_l)$ be a complete system of orthogonal idempotents, and let $(\mu_1,\dots, \mu_l)$ be $l$ distinct complex numbers. Let $x=\sum_{j=1}^l \mu_j d_j$. Then $x$ is semisimple, $d_j\in \mathbb C[x]$ for every $1\leq j\leq l $ and the eigenvalues of $x$ are $\mu_1,\dots, \mu_l$.
\end{proposition}

\begin{proof}Let $p\in \mathbb C[T]$. Then
\[p(x) = \sum_{j=1}^l p(\mu_j) d_j.
\]
For $1\leq j\leq l$, let $p_j$ be the unique polynomial of degree $l-1$ such that $p_j(\mu_i) = \delta_{ij}$. Then $p_j(x) = d_j$, so that $d_j\in \mathbb C[x]$.

Now let $p\in \mathbb C[T]$. Then
\[p(x)=0 \Longleftrightarrow p(\mu_j)=0.
\]
This shows that $p$ belongs to   the ideal $I(x)$ if and only if $p$ is a multiple of $\Pi(T) = \prod_{j=1}^l (T-\mu_j)$. In other words $\Pi$ is the minimal polynomial of $x$ and the conclusion follows.
\end{proof}
\begin{proposition} Let $x\in \mathbb V$ be a semi-simple element, with distinct eigenvalues $\lambda_1,\dots,\lambda_k$. There exists a unique (up to permutation of the indices) system of orthogonal idempotents $(c_1,c_2,\dots, c_k)$ in $\mathbb C[x]$ such that 
\[c_1+c_2+\dots+c_k=\mathbf{1},\qquad x=\lambda_1 c_1+\dots +\lambda_k c_k.
\]
\end{proposition}
This is (part of)  \cite[Proposition VIII.3.2]{FK}. The uniqueness statement comes from the fact that necessarily $c_j = p_j(x)$ where $p_j$ is the polynomial of degree $k-1$ which satisfies $p_j(\lambda_i) = \delta_{ij}$.
\begin{proposition} Let $\mathbb V$ be a complex simple Jordan algebra.The set of semi-simple regular elements of $\mathbb V$ is open and dense in $\mathbb V$.
\end{proposition}
\begin{proof} $\mathbb V$ has a real form $V$ which is euclidean and simple. Let $r$ be the rank of $V$ and let $(c_1,c_2,\dots, c_r)$ be a Jordan frame of $V$. Choose
$r$ distinct complex numbers $\lambda_1,\lambda_2,\dots,\lambda_r$, and let $x= \lambda_1c_1+\dots+\lambda_r c_r$. Then $x$ is semi-simple, and $x$ is regular as the minimal polynomial of $x$ is of degree $r$. The set of regular elements is an open Zariski subset of $\mathbb V$ (see \cite[Proposition IV.1.1]{FKKLR}). For $x$ regular, the minimal polynomial is equal to the generic minimal polynomial. The set  where the generic minimal polynomial $m_x$ has only simple roots is a Zariski open subset, as this is the set where the discriminant of $m_x$ (which is polynomial in $x$) vanishes. Hence the set of semi-simple regular elements is a Zariski open subset, which is non empty by the first part and hence dense. 
\end{proof}   
 A real Jordan algebra $V$ is said to be complex, if there exists a linear isomorphism $J$ of $V$ such that
 \[J^2 = -\Id,\qquad J(xy) = (Jx)y,\quad  \forall x,y\in V.
 \]
 \begin{proposition} Let $V$ be a simple real Jordan algebra. Let $\mathbb V$ be its complexification. Then $\mathbb V$ is simple, unless $V$ has a complex structure in which case $\mathbb V$ is isomorphic to $V\oplus V_{opp}$.
 \end{proposition}
 
 \begin{proof} Assume that $\mathbb V$ is \emph{not} simple. There exists a non trivial ideal $\mathbb J\subset \mathbb V$. Let $\sigma$ be the conjugation with respect to $V$. Then $\sigma(\mathbb J)$ is also an ideal of $\mathbb V$. Let 
 $\mathbb H = \mathbb J\cap \sigma(\mathbb J)$. Then $\mathbb H$ is an ideal of $\mathbb V$, which moreover is $\sigma$-stable. So $\mathbb H = H\oplus i  H$ for some subspace $H$ of $V$. Now $H$ is an ideal of $V$. Hence $H=0$ or $H=V$. But the second assumption leads to $\mathbb H=\mathbb V$ and hence $\mathbb J = \mathbb V$, a contradiction. Hence $\mathbb J\cap \sigma(\mathbb J) = \{0\}$. Along the same lines, one can prove that $\mathbb J+ \mathbb \sigma( \mathbb  J) = \mathbb V$. As a result, the Jordan algebra $\mathbb V$ splits as  $\mathbb V=\mathbb J\oplus \sigma(\mathbb J)$. Now consider the (real linear) map
 \[\mathbb J \ni x\ \longmapsto \ x+\sigma(x) \in V,
 \]
it is both injective and surjective, and hence an isomorphism. But as $\mathbb J$ and $\sigma(\mathbb J)$ are ideals of $\mathbb V$,
 \[(x+\sigma(x))(y+\sigma(y))= (xy+\sigma(xy))
 \]
 and hence $V$ is isomorphic to $\mathbb J$. Moreover, as $\sigma$ is $\mathbb C$-conjugate linear, the complex structure of $\sigma(\mathbb J)$ is the opposite complex of that of $\mathbb J$.
  \end{proof}

\begin{proposition}\label{primitiveRvsC}
Let $V$ be a real simple Jordan algebra. Let $c$ be a primitive idempotent element of $V$. Then $c$, viewed as an element of $\mathbb V$, is either  a primitive idempotent or it can be decomposed as $c=d+\sigma(d)$, where $\sigma$ is the conjugation w.r.t. $V$, and $d$ and $\sigma(d)$ are orthogonal idempotents of $\mathbb V$. 
\end{proposition}
\begin{proof} Assume $c$ is not primitive in $\mathbb V$. Let $\mathbb V_1=\mathbb V_1(c)$ and $V_1=V_1(c)$. Now $V_1$ is a real semi-simple Jordan algebra. As $V$ has $c$ as its unique idempotent, $V_1$ is simple. By \cite[Section 6]{Helwig}, $V_1$ is a \emph{Jordan field}, and there exists a \emph{positive-definite} bilinear $\beta$ form on $V_1$ such that
\[xy= \beta(x,c)y+\beta(c,y)x-\beta(x,y)c.
\]
Then $\beta(c,c)=1$.  Let $W= (\mathbb R c)^\perp\neq\{0\}$. Rewrite elements of $V_1$ as 
$sc+v$ where $s\in \mathbb R$ and $v \in W$, so that the Jordan product of $V_1$ can be written as
\[(sc+v)(tc+w) = (st-\beta(v,w))c +(sw+tv).
\]
Now let $\mathbb V_1$ the complexification of $V_1$ which can be seen as \[\mathbb V_1 = \{(\lambda c+\xi), \quad \lambda\in \mathbb C, \xi\in \mathbb W\}\]
with Jordan product given by
\[(\lambda c+\xi)(\mu c+ \eta) = \big(\lambda \mu-\beta(\xi,\eta)\big) c+ \mu \xi+\lambda\eta,
\]
where $\beta$ is the $\mathbb C$-bilinear extension of $\beta$ to $\mathbb V_1\times \mathbb V_1$.
As \[(\lambda c+\xi)^2 = \big(\lambda^2-\beta(\xi,\xi)\big)c+ 2\lambda \xi\ ,\] idempotents of $\mathbb V_1$ are of the form $(\frac{1}{2}c + \xi)$ where $\beta(\xi,\xi)=-\frac{1}{4}$.

Now choose $w\in W$ such that $\beta(w,w)=\frac{1}{4}$. Then let
$d=\frac{1}{2}c+iw$ and $f=\overline d = \frac{1}{2}c -iw$. Both $d$ and $f$ are idempotents, $df = 0$ and $c=d+f$, thus justifying the lemma.
\end{proof}

\addsubsection{Complex simple Jordan algebra viewed as a real Jordan algebra}

Let $\mathbb V$ be a simple complex Jordan algebra, and let $\Delta$ be its determinant. When viewed as a real Jordan algebra, it is still simple. Let $\det$ be its determinant.

\begin{lemma}
 For any $z\in \mathbb V$
\begin{equation}\label{detRC}
\det(z) = \Delta(z)\overline {\Delta(z)}.
\end{equation}
\end{lemma}
\begin{proof} Let $x\in \mathbb V$ and let $p$ be its minimal polynomial over $\mathbb C$. Then $p\overline p (x)=0$, and $p\overline p$ is a polynomial with real coefficients. As $\deg(p)\leq r$, where $r$ is the rank of $\mathbb V$, the degree of the minimal polynomial over $\mathbb R$ of an element is $\leq 2r$. Conversely, consider a complete system of primitive idempotents $(c_1,c_2,\dots, c_r)$ of $\mathbb V$. Let $x = \sum_{i} \lambda_jc_j$ with $\lambda_i\in \mathbb C$. Suppose that $p$ is a  polynomial with real coefficients such that $p(x)=0$. Then, necessarily
$p(\lambda_j)=0$ for $1\leq j\leq r$. But as $p$ is real valued, $p(\overline{\lambda}_j) = 0$ for $1\leq j\leq r$. Assume that $\lambda_j,\overline \lambda_j$ are all distinct. Then $p$ is a multiple of $\prod_{j=1}^r(T-\lambda_j)(T-\overline{\lambda}_j)$. Hence $\prod_{j=1}^r(T-\lambda_j)(T-\overline{\lambda}_j)$ \emph{is} its minimal polynomial over $\mathbb R$. Moreover $\det(x) = \prod_{j=1}^r\lambda_j\overline{\lambda}_j = \Delta(x)\overline {\Delta(x)}$. Hence the identity 
\autoref{detRC} is valid for the elements of the form considered. But clearly these elements form a dense set in $\mathbb V$, hence the identity holds in general.
\end{proof}

\addsubsection{Non complex non-split simple real Jordan algebra}

The algebra $V$ is said to be split if a (hence any) primitive idempotent is primitive in the complexification $\mathbb V$. Otherwise it is said to be non-split. 

Let $V$ be a non-split simple real Jordan algebra. Let $(c_1,c_2,\dots, c_\rho)$ be a maximal set of orthogonal primitive idempotents. Then there exists $d_j,f_j\in \mathbb V(c_j)$ , such that
\[f_j=\sigma(d_j), \qquad c_j=d_j+f_j,\quad 1\leq j\leq \rho
\]
By Proposition, \ref{primitiveRvsC}, $c_j$ decomposes as a sum $c_j=d_j+f_j$, and so $(d_j,f_j)_{1\leq j\leq \rho}$ is a maximal family of orthogonal idempotents of $\mathbb V$. Hence $r=2\rho$. 

\begin{proposition} Let $V$ be a real simple Jordan algebra which has no complex structure and which is non-split. Then $\det (x)\geq 0$ for every $x\in V$.
\end{proposition}
\begin{proof}
Let $x\in V$ be a regular semi-simple element. There exists a maximal set $(c_1,c_2,\dots, c_\rho)$ of orthogonal primitive idempotents such that  $x= \sum_{j=1}^\rho t_jc_j $.  For each $j, 1\leq j \leq \rho$, let $d_j,f_j\in \mathbb V(c_j)$ , such that
\[f_j=\sigma(d_j), \qquad c_j=d_j+f_j,\quad 1\leq j\leq \rho, 
\]
so that $x=\sum_{j=1}^\rho t_j d_j+\sum_{j=1}^\rho t_j f_j$ and hence ${\det}_V(x) = {\det}_{\mathbb V} (x) = \prod_{j=1}^\rho t_j^2\geq 0$.
The conclusion follows as regular semi-simple elements are dense in $V$.
\end{proof}

\end{document}